\documentclass[12pt]{amsart}

\input amssym.def
\input amssym.tex
\usepackage[noadjust]{cite}
\usepackage{multirow}
\usepackage[left=1in, right=1in, top=1.25in, bottom=1.25in]{geometry}
\usepackage{url}

\newtheorem{thm}{Theorem}[subsection]
\newtheorem{cor}[thm]{Corollary}
\newtheorem{lem}[thm]{Lemma}
\newtheorem{prop}[thm]{Proposition}

\newtheorem{athm}{Theorem}[section]
\newtheorem{alem}[athm]{Lemma}

\newtheorem*{resprime}{Theorem \ref{thm:resprime}}
\newtheorem*{henselhelper}{Lemma \ref{lem:henselhelper}}
\newtheorem*{AKT}{Theorem \ref{thm:AKT}}
\newtheorem*{AKP}{Theorem \ref{thm:AKP}}
\newtheorem*{CH}{Continuum Hypothesis}
\newtheorem*{GCH}{Generalized Continuum Hypothesis}

\theoremstyle{definition}

\newtheorem{defin}[thm]{Definition}
\newtheorem{exa}[thm]{Example}
\newtheorem{rmk}[thm]{Remark}

\newtheorem{adefin}[athm]{Definition}

\begin{document}

\title{The Ax-Kochen Theorem: An Application of Model Theory to Algebra}

\author{Alex Kruckman}

\begin{abstract}
The Ax-Kochen Theorem is a purely algebraic statement about the zeros of homogeneous polynomials over the $p$-adic numbers, but it was originally proved using techniques from mathematical logic. This document, the author's undergraduate honors thesis, provides an exposition of the theorem and its proof via model theory, assuming no previous experience with logic.
\end{abstract}

\maketitle

\tableofcontents

\section{An Opening Remark}\label{sec:NewIntro}

This document is the author's undergraduate honors thesis, completed at Brown University in Spring 2010. It is an exposition of a direct route to the proof of the Ax-Kochen Theorem, requiring no previous experience with mathematical logic (indeed the author knew very little model theory when he wrote it!), or with valued fields. 

As such, it is a bit old fashioned, and it omits proofs of some of the key ingredients from the theory of valued fields (Lemmas ~\ref{lem:unique_extension} through ~\ref{lem:transcendental_value_group}). A reader looking for a more sophisticated approach could consult the lecture notes by van den Dries \cite{vddries}. In these notes, the theory of valued fields is central, and the Ax-Kochen Principle is viewed as a consequence of a relative quantifier elimination result for Henselian valued fields of equicharacteristic $0$, in which satisfaction of sentences in the language of valued fields is reduced to satisfaction of sentences in the language of the residue field and the language of the value group, respectively.

\section{Introduction}\label{sec:Intro}

Certain fields have the property, called $C_i$, that every homogeneous polynomial with enough variables relative to its degree (specifically, $n>d^i$, where $n$ is the number of variables and $d$ is the degree) has a nontrivial zero.

Emil Artin conjectured that for all primes $p$, the $p$-adic field $\mathbb{Q}_p$ is $C_2$. This conjecture turned out to be false; in fact, $\mathbb{Q}_p$ is not $C_2$ for any $p$. However, in their paper \emph{Diophantine problems over local fields} \cite{Ax}, Ax and Kochen provided a partially positive result.

\begin{AKT}[Ax-Kochen Theorem]
For each degree $d\geq 1$, there exists a finite set of primes $P(d)$ such that for all $p\notin P(d)$, if $f$ is a homogeneous polynomial over $\mathbb{Q}_p$ of degree $d$ in $n$ variables such that $n>d^2$, then $f$ has a nontrivial zero in $\mathbb{Q}_p^n$. 
\end{AKT}

The methods used by Ax and Kochen come from model theory, a branch of mathematical logic. They were able to prove a much more general result, known as the Ax-Kochen Principle, which allows theorems about the fields $\mathbb{F}_p((t))$ of formal Laurent series over the finite fields $\mathbb{F}_p$ to be transferred to theorems about the fields $\mathbb{Q}_p$. 

\begin{AKP}[Ax-Kochen Principle]
Any first-order logical statement about valued fields which is true of all but finitely many of the fields $\mathbb{F}_p((t))$ is true of all but finitely many of the fields $\mathbb{Q}_p$.
\end{AKP}

This thesis provides an exposition of the algebra and model theory necessary to understand the Ax-Kochen Theorem and its proof. It should be accessible to any reader with a firm grasp of abstract algebra.

We begin in Section~\ref{subsec:ACF} by introducing homogeneous polynomials and the $C_i$ properties. In Sections~\ref{subsec:FF} and \ref{subsec:EF}, we prove $C_i$ properties for finite fields and algebraic and transcendental extension fields. In Section~\ref{subsec:VF}, we introduce valued fields and the completion of a discrete valued field, constructing the $p$-adic fields along the way. Finally, we prove that $\mathbb{F}_p((t))$ is $C_2$ for all $p$, the result that will be transferred to the $p$-adics to complete the proof of the Ax-Kochen Theorem. For the material in Chapter~\ref{sec:OQAC}, I have followed Greenberg \cite{Greenberg} closely.

In Chapter~\ref{sec:MT}, we introduce the reader to model theory, with a focus on those techniques and examples relevant to the Ax-Kochen Principle. I have modeled my notation and exposition after that in Marker \cite{Marker}, but some of the details (for example, the material on ultraproducts and the model theory of valued fields) come from Chang and Keisler \cite{Chang}.

Chapter~\ref{sec:AKP} is devoted to the proof of of the Ax-Kochen Principle. The proof relies on the result that the $\mathbb{F}_p((t))$ and $\mathbb{Q}_p$ are Henselian valued fields, and we introduce Hensel's Lemma and some of its consequences in Section~\ref{subsec:HL}. We give the proof of the Ax-Kochen Principle in Section~\ref{subsec:EEE}, the cornerstone of which is Theorem~\ref{thm:el_eq}, which implies that the ultraproducts of the fields $\mathbb{F}_p((t))$ and $\mathbb{Q}_p$ are elementarily equivalent. Finally, we derive the Ax-Kochen Theorem as a corollary in Section~\ref{subsec:AKT}. Again, the main reference for the proof is Chang and Keisler \cite{Chang}.

Some sections require a familiarity with the transfinite numbers. Their properties are covered in Appendix~\ref{sec:OCTI}. We will also use the resultant, an algebraic tool for comparing the roots of two polynomials. It is introduced in Appendix~\ref{sec:TR}. For simplicity, the proof of the Ax-Kochen Principle as given relies on the Continuum Hypothesis. Appendix~\ref{sec:SM} describes a method for eliminating the Continuum Hypothesis from the argument.

This thesis was written in partial fulfillment of the requirements for the degree of Bachelor of Science with Honors in Mathematics at Brown University. I would like to express my eternal gratitude to my advisors Dan Abramovich and Michael Rosen, who have been extremely generous with their time, suggestions, and support, and to my parents, for their devotion to my education.

\newpage 

\section{On Quasi-Algebraic Closure}\label{sec:OQAC}

\subsection{The $C_i$ Properties and Algebraically Closed Fields}\label{subsec:ACF}

\begin{defin}\label{def:hompoly}
A polynomial $f$ over a field $F$ is \emph{homogeneous} of degree $d\geq1$ in $n$ variables, $x_1,\hdots,x_n$, if all monomials of $f$ have degree $d$, that is, if it can be written in the form $ f(x_1,\hdots,x_n)=\sum_i a_ix_1^{b_{i,1}}\hdots x_n^{b_{i,n}}$ such that for all $i$, $\sum_{j=1}^n b_{i,j} = d$. 
\end{defin}

\begin{rmk}\label{rmk:homconst}
If $f$ is a homogeneous polynomial of degree $d$ in $n$ variables over $F$, then for all $c\in F$, $f(cx_1,\hdots,cx_n) = c^d f(x_1,\hdots,x_n)$.
\end{rmk}

\begin{exa}\label{exa:det}
The function which computes the determinant of an $n\times n$ matrix is a homogeneous polynomial of degree $n$ in $n^2$ variables, the matrix entries.
\end{exa}

Since a homogeneous polynomial cannot have a constant term, all homogeneous polynomials have the trivial zero $(0,\hdots,0)$. It is of interest to explore when homogeneous polynomials have nontrivial zeros.

\begin{exa}\label{exa:hompoly}
Let $f_n$ be the polynomial $x_1^2 + x_2^2 + \hdots + x_n^2$. For all $n>0$, $f_n$ is a homogeneous polynomial of degree 2 in $n$ variables. Over $\mathbb{R}$, $f_n$ has only the trivial zero for all $n$. But over $\mathbb{C}$, $f_n$ has nontrivial zeros (for example, $(1,i,0,\hdots,0)$) for all $n>1$. It is easy to check that over $\mathbb{F}_7$, $f_n$ has nontrivial zeros for all $n>2$, and in Section~\ref{subsec:FF} we will show that this is the case for all finite fields. The $2$ comes from the degree of $f_n$.
\end{exa}

\begin{defin}\label{def:C_i}
A field $F$ is called $C_i$ for $i\in\mathbb{N}$ if every homogeneous polynomial over $F$ of degree $d$ in $n$ variables such that $n>d^i$ has a nontrivial zero in $F^n$.
\end{defin}

We can easily characterize the $C_0$ fields.

\begin{thm}\label{thm:ACF}
A field is $C_0$ if and only if it is algebraically closed.
\end{thm}
\begin{proof}
Suppose $F$ is an algebraically closed field. Let $f(x_1,\hdots,x_n)$ be a homogeneous polynomial over $F$ of degree $d$ in $n>d^0=1$ variables. Write $f$ as a polynomial in one variable, $x_1$, with coefficients in $F[x_2,\hdots,x_n]$, $f = \sum_{i=1}^{d'} f_i(x_2,\hdots,x_n) x_1^i$. The degree of this polynomial, $d'$, is the highest power of $x_1$ appearing in any term of $f$.

If $d' = 0$, then no nonzero power of $x_1$ appears in $f$, so $f(1,0,\hdots,0) = f(0,0,\hdots,0) = 0$, and $(1,0,\hdots,0)$ is a nontrivial zero of $f$. Otherwise, if $d' > 0$, consider the leading coefficient, $f_{d'}(x_2,\hdots,x_n)$. We would like to find nontrivial $(\alpha_2,\hdots,\alpha_n)\in F^{n-1}$ which is \emph{not} a zero of $f_{d'}$. All algebraically closed fields are infinite, and a nonzero polynomial cannot have infinitely many zeros, so there exists $(\alpha_2,\hdots,\alpha_n)\in F^{n-1}$ such that $\alpha_j\neq 0$ for some $j$ and $f_{d'}(\alpha_2,\hdots,\alpha_n)\neq 0$. 

Let $\overline{f} = f(x_1,\alpha_2,\hdots,\alpha_n)$. We have simply evaluated the coefficients $f_i$ at $(\alpha_2,\hdots,\alpha_n)$, and the leading coefficient is nonzero, so $\overline{f}$ is a polynomial of degree $d'>0$ in one variable, $x_1$. Since $F$ is algebraically closed, $\overline{f}$ has a zero, $\alpha_1$. Then $(\alpha_1,\alpha_2,\hdots,\alpha_n)$ is a zero of $f$, and this zero is nontrivial, since $\alpha_j \neq 0$.

Conversely, suppose $F$ is a $C_0$ field. Let $f(x)=a_dx^d + a_{d-1}x^{d-1} + \hdots + a_0$ be a polynomial of degree $d\geq1$ over $F$. We would like to show that $f$ has a root in $F$. Let $\widehat{f}(x_1,x_2) = a_dx_1^d + a_{d-1}x_1^{d-1}x_2 + \hdots + a_1x_1x_2^{d-1} + a_0x_2^d$. Now $\widehat{f}$ is a homogeneous polynomial over $F$ of degree $d$ in $2$ variables, and $2>d^0=1$, so $\widehat{f}$ has a nontrivial zero $(\alpha_1,\alpha_2)\in F^2$. Note that $\alpha_2\neq0$, since otherwise $\widehat{f}(\alpha_1,\alpha_2) = \widehat{f}(\alpha_1,0) = a_d\alpha_1^d = 0$, so $\alpha_1=0$, and $(\alpha_1,\alpha_2)$ is trivial. 

We have $\widehat{f}(\alpha_1,\alpha_2) = 0$, so by Remark~\ref{rmk:homconst}, $\widehat{f}(\alpha_1\alpha_2^{-1},1) = (\alpha_2^{-1})^d \widehat{f}(\alpha_1,\alpha_2) = 0$. But substituting $1$ for $x_2$ in $\widehat{f}$, $\widehat{f}(x_1,1) = f(x_1)$, so $f(\alpha_1\alpha_2^{-1}) = \widehat{f}(\alpha_1\alpha_2^{-1},1) = 0$, and $\alpha_1\alpha_2^{-1}$ is a root of $f$ in $F$. Thus every polynomial over $F$ of nonzero degree has a root in $F$, and hence $F$ is algebraically closed.
\end{proof}

This theorem suggests that the $C_i$ properties can be seen as generalizations of the property of algebraic closure. For this reason, $C_1$ fields are called quasi-algebraically closed. In the following sections, we will show that many frequently encountered fields are $C_i$ for some $i$. As a special case, we will obtain our first main result: for all primes $p$, $\mathbb{F}_p((t))$, the field of formal Laurent series in one variable over the finite field with $p$ elements, is $C_2$.

\subsection{Finite Fields}\label{subsec:FF}

Throughout this section, let $K$ be a finite field of characteristic $p$ with $|K|=q$. Recall that
\begin{itemize}
\item $p$ is prime,
\item $q = p^v$ for some $v>0$, and
\item the multiplicative group $K^* = K\backslash\{0\}$ is cyclic of order $q-1$.
\end{itemize}

We will show that all finite fields are $C_1$. We begin with a simple but useful lemma. 

\begin{lem}\label{lem:C-W}
For $m>0$, 
\[
\sum_{a\in K}a^m=\begin{cases}-1 & \mbox{if}\,\, (q-1)\,|\,m\\0 & \mbox{otherwise}\end{cases}.
\]
\end{lem}
\begin{proof}
Suppose $q-1\,|\,m$. Then for all $a\in K^*$, $a^m = 1$, so 
\begin{eqnarray*}
\sum_{a\in K}a^m &=& 0^m + \sum_{a\in K^*}a^m\\
&=& \sum_{a\in K^*}1\\
&=& -1
\end{eqnarray*} 
since $q-1\equiv-1 \,\,(\mbox{mod}\,\,p)$.

Otherwise, if $(q-1)\nmid m$, let $b$ be a generator of the cyclic group $K^*$. Then $b^m\neq 1$, since $|K^*| = q-1$. Let $S = \sum_{a\in K}a^m = \sum_{a\in K^*}a^m$. Multiplication by $b$ permutes the elements of $K^*$, so
\begin{eqnarray*}
S &=& \sum_{a\in K^*}(ba)^m\\
&=& b^m \sum_{a\in K^*}a^m\\
&=& b^m S,
\end{eqnarray*}
and thus $(b^m - 1)S = 0$. But $b^m-1\neq 0$, so $S = 0$, as was to be shown.
\end{proof}

The next theorem implies that finite fields are $C_1$, but it is actually a stronger result about the number of zeros of any polynomial (not necessarily homogeneous) with more variables than its degree.

\begin{thm}[Chevalley-Warning, {\cite[Theorem 2.3]{Greenberg}}]\label{thm:C-W}
Let $f$ be a polynomial over $K$ of degree $d$ in $n$ variables, $x_1,\hdots,x_n$. If $n>d$, then the number of zeros of $f$ in $K^n$ is divisible by $p$.
\end{thm}
\begin{proof}
For $(\alpha_1,\hdots,\alpha_n)\in K^n$, 
\[
1-f(\alpha_1,\hdots,\alpha_n)^{q-1} = \begin{cases}1&\mbox{if}\,\,f(\alpha_1,\hdots,\alpha_n)=0\\0 & \mbox{otherwise}\end{cases}.
\]
We will count the number of zeros (mod $p$) of $f$ by summing the values of this expression over all $(\alpha_1,\hdots,\alpha_n)\in K^n$. There are $q^n$ such $n$-tuples.
\begin{eqnarray*}
\sum_{(\alpha_1,\hdots,\alpha_n)\in K^n} (1-f(\alpha_1,\hdots,\alpha_n)^{q-1}) &=& q^n - \sum_{(\alpha_1,\hdots,\alpha_n)\in K^n}f(\alpha_1,\hdots,\alpha_n)^{q-1} \\
&=& 0 - \sum_{(\alpha_1,\hdots,\alpha_n)\in K^n}f(\alpha_1,\hdots,\alpha_n)^{q-1}.
\end{eqnarray*}

Now $f^{q-1}$ has degree $d(q-1)$, and we can write it as a linear combination of monomials of at most that degree. Let $\prod_{i=1}^n x_i^{\mu_i}$ be one such monomial. The degree of this monomial is $\sum_{i=1}^n \mu_i \leq d(q-1)$. By assumption, $n > d$, so for at least one $j$, $\mu_j < q-1$. Consider the sum $\sum_{(\alpha_1,\hdots,\alpha_n)\in K^n}\prod_{i=1}^n \alpha_i^{\mu_i} = \prod_{i=1}^n\sum_{\alpha_i\in K}\alpha_i^{\mu_i}$. The $j^{th}$ term of this product is $\sum_{\alpha_j\in K}\alpha_j^{\mu_j}$. If $\mu_j = 0$, this is $\sum_{\alpha_j\in K}1 = q = 0$. Otherwise, $0 < \mu_j < q-1$, so the sum is 0 by Lemma \ref{lem:C-W}. Hence the product is 0, and the sum over $(\alpha_1,\hdots,\alpha_n)\in K^n$ of each monomial of $f^{q-1}$ is 0, so $\sum_{(\alpha_1,\hdots,\alpha_n)\in K^n}f(\alpha_1,\hdots,\alpha_n)^{q-1} = 0$. 

Thus the number of zeros of $f$ in $K^n$ is $0$ mod $p$.
\end{proof}

\begin{cor}\label{cor:FF}
Finite fields are $C_1$.
\end{cor}
\begin{proof}
Let $f$ be a homogeneous polynomial over the finite field $K$ of degree $d$ in $n$ variables, where $n>d$. By Theorem \ref{thm:C-W}, the number of zeros of $f$ is divisible by $p$. Now $f$ has at least one zero (the trivial zero), so it has at least $p$ zeros, and in particular it has at least $p-1$ nontrivial zeros. Thus $K$ is $C_1$. 
\end{proof}

\subsection{Extension Fields}\label{subsec:EF}

In this section, we will show that an extension field of a $C_i$ field of finite transcendence degree $j$ is $C_{i+j}$. The main idea is to expand a homogeneous polynomial according to a basis for the extension field into a vector of homogeneous polynomials over the base field. So we will need a tool (Theorem \ref{thm:L-N}) for finding nontrivial common zeros of sets of homogeneous polynomials. The proof of this theorem relies on the concept of a normic form. 

\begin{defin}\label{def:normicform}
A \emph{normic form} is a homogeneous polynomial $\phi$ of degree $d$ in $n$ variables such that $n=d$ and $\phi$ has only the trivial zero.
\end{defin}

The name normic form comes from the following example. 

\begin{defin}\label{def:norm}
Let $E$ be a finite algebraic extension of a field $F$. For all $x\in E$, let $m_x: E\rightarrow E$ be the linear transformation $m_x(y) = xy$. The \emph{norm} of $x$, denoted $N(x)$, is the determinant of $m_x$.
\end{defin}

\begin{exa}\label{exa:norm}
Consider $\mathbb{C}$ as an algebraic extension of $\mathbb{R}$ of degree $2$. Take $\{1,i\}$ as a basis for $\mathbb{C}$ over $\mathbb{R}$. For any complex numbers $a$ and $b$, we can write $a = a_1 + a_2 i$ and $b = b_1 + b_2 i$ according to this basis. Then $ab = (a_1 b_1 - a_2 b_2) + (a_1 b_2 + a_2 b_1)i$. 

Representing $b$ as a vector and multiplication by $a$ as a matrix, we have
\[
m_a(b) =
\left(
\begin{array}{cc}
a_1 & -a_2 \\
a_2 & a_1 \\
\end{array}
\right)
\left(\begin{array}{c}
b_1\\
b_2\\
\end{array}\right) = 
\left(
\begin{array}{c}
a_1 b_1 - a_2 b_2 \\
a_1 b_2 + a_2 b_1 \\
\end{array}
\right).
\]

Then $N(a) = |m_a| = a_1^2 + a_2^2$. Taking the coordinates $a_1$ and $a_2$ as variables, $N$ is a homogeneous polynomial of degree $2$ in $2$ variables over $\mathbb{R}$, and it has only the trivial zero in $\mathbb{R}^2$, so $N$ is a normic form.
\end{exa}

\begin{lem}[{\cite[Lemma 3.1]{Greenberg}}]\label{lem:norm}
If $E$ is a finite algebraic extension of $F$ of degree $d>1$, then the norm $N(x)$ is a normic form over $F$ of degree $d$, whose variables are the $d$ coordinates of $x$ after choosing a basis for $E$ as a vector space over $F$.
\end{lem}
\begin{proof}
Let $w_1,\hdots,w_d$ be a basis for $E$. We define the constants $c^{k,l}_j$ by
\[
w_kw_l = \sum_{j=1}^d c^{k,l}_jw_j
\]
for all $1\leq k,l \leq d$. That is, $c^{k,l}$ is $w_kw_l$ expressed as a vector.

We will write the variable $x$ as a vector in terms of this basis, $x=\sum_{k=1}^d x_kw_k$. Then for any $b\in E$, writing $b=\sum_{l=1}^d b_lw_l$,
\begin{eqnarray*}
m_x(b) &=& \left(\sum_{k=1}^d x_kw_k\right)\left(\sum_{l=1}^d b_lw_l\right)\\
&=& \sum_{k=1}^d \sum_{l=1}^d x_kb_lw_kw_l\\
&=& \sum_{j=1}^d \sum_{l=1}^d \sum_{k=1}^d x_kb_lc^{k,l}_jw_j\\
&=& \left(
\begin{array}{c}
\sum_{l=1}^d \sum_{k=1}^d x_kb_lc^{k,l}_1 \\
\vdots\\
\sum_{l=1}^d \sum_{k=1}^d x_kb_lc^{k,l}_d
\end{array}\right)\\
&=&
\left(
\begin{array}{ccc}

\sum_{k=1}^d x_kc^{k,1}_1 & \hdots & \sum_{k=1}^d x_kc^{k,d}_1 \\

\vdots & \ddots & \vdots \\

\sum_{k=1}^d x_kc^{k,1}_d & \hdots & \sum_{k=1}^d x_kc^{k,d}_d 

\end{array}
\right)
\left(
\begin{array}{c}
b_1\\
\vdots\\
b_d
\end{array}
\right),
\end{eqnarray*}
and we have determined the matrix representation of $m_x$.

The determinant of this matrix, $N(x)$, is a homogeneous polynomial of degree $d$ in the variables $x_1,\hdots,x_d$. For $a\in E$, if $a\neq 0$, $a$ has an inverse in $E$, so multiplication by $a$ is invertible, $m_a$ is an invertible matrix, and $N(a)\neq 0$. Thus $N$ has only the trivial zero, and $N$ is a normic form.
\end{proof}

\begin{lem}[{\cite[Lemma 3.2]{Greenberg}}]\label{lem:normic}
If a field $F$ is not algebraically closed, then there exist normic forms over $F$ of arbitrarily large degree.
\end{lem}
\begin{proof}
$F$ is not algebraically closed, so it has some finite algebraic extension of degree $d>1$. By Lemma \ref{lem:norm}, there is a normic form $\phi(x_1,\hdots,x_d)$ of degree $d$ over $F$. Let $\phi_1(x_{1,1},\hdots,x_{1,d}), \hdots, \phi_d(x_{d,1},\hdots,x_{d,d})$ be $d$ copies of $\phi$, each with a set of $d$ distinct variables. By substituting each $\phi_i$ for $x_i$ in $\phi$, we obtain $\phi^{(2)} = \phi(\phi_1,\hdots,\phi_d)$, which is a homogeneous polynomial of degree $d^2$ in $d^2$ variables.

Now $\phi$ has only the trivial zero, so at any zero of $\phi^{(2)}$, each $\phi_i$ must also take the value $0$. But each $\phi_i$ has only the trivial zero, so $\phi^{(2)}$ has only the trivial zero, and thus $\phi^{(2)}$ is a normic form of degree $d^2$.

For all $m>2$, we inductively define $\phi^{(m)} = \phi^{(m-1)}(\phi_1,\hdots,\phi_{d^{m-1}})$, where $\phi_1,\hdots,\phi_{d^{m-1}}$ are copies of $\phi$, each with a distinct set of $d$ variables. The same argument shows that $\phi^{(m)}$ is a normic form of degree $d^m$. Taking $m$ arbitrarily large produces normic forms of arbitrarily large degree.
\end{proof}

We will assume the following theorem. It is not necessary to prove the Ax-Kochen Theorem, but it will allow us to state Theorems \ref{thm:L-N} and \ref{thm:trans_ext} so that they also cover the $C_0$ case. The proof can be easily located in a book on algebraic geometry, for example Hartshorne~\cite{Hartshorne} Chapter 1, Theorem 7.2 is an equivalent statement.

\begin{thm}\label{thm:common_zero}
Let $F$ be an algebraically closed field. If $f_1,\hdots,f_r$ are homogeneous polynomials over $F$ in $n$ variables, where $n>r$, then they have a common nontrivial zero in $F^n$.
\end{thm}

When $F$ is not algebraically closed, but is $C_i$ for $i>0$, we can use normic forms to demonstrate the existence of nontrivial common zeros.

\begin{thm}[Lang-Nagata, {\cite[Theorem 3.4]{Greenberg}}]\label{thm:L-N}
Let $F$ be a $C_i$ field. Let $f_1,\hdots,f_r$ be homogeneous polynomials over $F$ of degree $d$ in $n$ variables. If $n>rd^i$, then they have a nontrivial common zero in $F^n$.
\end{thm}

\begin{proof}

If $F$ is algebraically closed, then $F$ is $C_0$ by Theorem \ref{thm:ACF}. So we have $n>rd^0=r$, and by Theorem \ref{thm:common_zero}, the polynomials have a nontrivial common zero in $F^n$.

Otherwise, there is a normic form $\phi$ over $F$ of degree $l \geq r$ in $l$ variables by Lemma \ref{lem:normic}. For all $m\geq 1$, we define $\phi^{(m)}$ inductively, and we define $D_m$ and $N_m$ to be the degree and number of variables of $\phi^{(m)}$ respectively:

\begin{eqnarray*}
\phi^{(1)} &=& \phi(f_{1,1},\hdots,f_{1,r},f_{2,1},\hdots,f_{2,r},\hdots,f_{\lfloor\frac{l}{r}\rfloor,1},\hdots,f_{\lfloor\frac{l}{r}\rfloor,r},0,\hdots,0)\\
\phi^{(m)} &=& \phi^{(m-1)}(f_{1,1},\hdots,f_{1,r},f_{2,1},\hdots,f_{2,r},\hdots,f_{\lfloor\frac{N_{m-1}}{r}\rfloor,1},\hdots,f_{\lfloor\frac{N_{m-1}}{r}\rfloor,r},0,\hdots,0),
\end{eqnarray*}
where each set of polynomials $f_{j,1},\hdots,f_{j,r}$ is a copy of the set $f_1,\hdots,f_r$ with a distinct set of $n$ variables, $x_{j,1},\hdots,x_{j,n}$. That is, for all $j$ and $k$, $f_{j,k} = f_k(x_{j,1},\hdots,x_{j,n})$. Note that we substitute as many complete sets as possible before padding with $0$s.  

We will prove by induction that for all $m\geq 1$, if $\phi^{(m)}$ has a nontrivial zero, then the $f_1,\hdots,f_r$ have a nontrivial common zero. For the base case, suppose that $\phi^{(1)}$ has a nontrivial zero, $\boldsymbol{\alpha}\in F^{N_1}$. We will denote by $\alpha_{j,k}$ the $x_{j,k}$-coordinate of $\boldsymbol{\alpha}$. Consider the values of the polynomials $f_{j,k}$ substituted into $\phi$ in the definition of $\phi^{(1)}$ at $\boldsymbol{\alpha}$. Since $\phi$ is normic, it has only the trivial zero, and thus all of the $f_{j,k}$ are $0$. This means that for all $j$, $\alpha_{j,1},\hdots,\alpha_{j,n}$ is a common zero for $f_1,\hdots,f_r$. Since $\boldsymbol{\alpha}$ is nontrivial, at least one of the $\alpha_{j,k}$ is nonzero, so for at least one $j$, $\alpha_{j,1},\hdots,\alpha_{j,n}$ is a nontrivial common zero of $f_1,\hdots,f_r$.

Now suppose that for $m>1$, $\phi^{(m)}$ has a nontrivial zero, $\boldsymbol{\alpha}\in F^{N_m}$. Consider the values of the polynomials $f_{j,k}$ substituted into $\phi^{(m-1)}$ at $\boldsymbol{\alpha}$. If they are all $0$, then for at least one $j$, $\alpha_{j,1},\hdots,\alpha_{j,n}$ is a nontrivial common zero of $f_1,\hdots,f_r$. If the values of the $f_{j,k}$ are not all $0$, then these values constitute a nontrivial zero for $\phi^{(m-1)}$, and by induction we have a nontrivial common zero for $f_1,\hdots,f_r$.

Since $F$ is $C_i$, it remains to show that for some $m$, $N_m>(D_m)^i$, since then $\phi^{(m)}$ has a nontrivial zero. We have $D_1 = dl$ and $N_1 = n\lfloor\frac{l}{r}\rfloor$, and for all $m > 1$, $D_m=dD_{m-1}$ and $N_m=n\lfloor\frac{N_{m-1}}{r}\rfloor$. Now,

\begin{eqnarray*}
\frac{N_m}{(D_m)^i} &\geq& \frac{n(\frac{N_{m-1}}{r})}{(dD_{m-1})^i}\\
&\geq& \left(\frac{n}{rd^i}\right)\left(\frac{N_{m-1}}{(D_{m-1})^i}\right)
\end{eqnarray*}

Expanding inductively,

\begin{eqnarray*}
\frac{N_m}{(D_m)^i} &\geq& \left(\frac{n}{rd^i}\right)^{m-1}\left(\frac{N_1}{(D_1)^i}\right)\\
&\geq& \left(\frac{n}{rd^i}\right)^{m-1}\left(\frac{n\left(\frac{l}{r}\right)}{(dl)^i}\right)\\
&\geq& \left(\frac{n}{rd^i}\right)^ml^{1-i}.
\end{eqnarray*}

By assumption, $n>rd^i$, so $\lim_{m\rightarrow\infty}\frac{N_m}{(D_m)^i} = \infty$, and in particular $N_m>(D_m)^i$ for $m$ large enough, as was to be shown.
\end{proof}

We are now in a position to prove our results about extension fields.

\begin{thm}[{\cite[Theorem 3.5]{Greenberg}}]\label{thm:alg_ext}
If $F$ is a $C_i$ field, then every algebraic extension of $F$ is $C_i$.
\end{thm}
\begin{proof}
It suffices to prove the theorem for any finite extension of $F$, since the coefficients of any polynomial lie in a finite extension.

Let $E$ be a finite algebraic extension of $F$ of degree $e$, and let $w_1,\hdots,w_e$ be a basis for $E$ over $F$. Let $f$ be a homogeneous polynomial over $E$ of degree $d$ in $n$ variables, $x_1,\hdots,x_n$, where $n>d^i$. We will write each variable in terms of the basis for $E$, substituting $\sum_{k=1}^e x_{j,k}w_k$ for each $x_j$ and letting the $x_{j,k}$ vary over $F$.

Expanding, and writing $f$ in terms of the basis for $E$, $f = \sum_{k=1}^e f_kw_k$, where the $f_k$ are polynomials in the variables $x_{j,k}$. Each $f_k$ is a linear combination of monomials of degree $d$, so the $f_k$ are homogeneous polynomials of degree $d$ in $en$ variables over $F$.

Now $f$ has a nontrivial zero in $E^n$ if and only if the $f_k$ have a nontrivial common zero in $F^{en}$. Such a zero exists by Theorem \ref{thm:L-N}, since $en>ed^i$.
\end{proof}

\begin{thm}[{\cite[Theorem 3.6]{Greenberg}}]\label{thm:trans_ext}
If $F$ is a $C_i$ field, and $E$ is an extension of $F$ of finite transcendence degree $j$, then $E$ is $C_{i+j}$.
\end{thm}
\begin{proof}
By Theorem \ref{thm:alg_ext}, we can reduce to the case in which $E$ is a purely transcendental extension. Any purely transcendental extension of $F$ of degree $j$ is isomorphic to the field of rational functions in $j$ variables, $F(t_1,\hdots,t_j)$. We will show that when $E=F(t)$, $E$ is $C_{i+1}$. A simple induction on $j$ then completes the proof.

The coefficients of homogeneous polynomials over $F(t)$ are, in general, rational functions. However, it suffices to consider homogeneous polynomials with coefficients in $F[t]$, the ring of polynomials, since we can clear denominators. That is, for $f\in F(t)[x_1,\hdots,x_n]$, if $g$ is the product of the denominators of the coefficients of $f$, then $g^df\in F[t][x_1,\hdots,x_n]$. But if $(a_1,\hdots,a_n)$ is a nontrivial zero of $g^df$, then $(ga_1,\hdots,ga_n)$ is a nontrivial zero of $f$ by Remark \ref{rmk:homconst}.

Let $f$ be a homogeneous polynomial over $F[t]$ of degree $d$ in $n$ variables, $x_1,\hdots,x_n$, where $n>d^{i+1}$. For some $s>0$, which we leave unspecified for now, substitute $\sum_{k=0}^s x_{j,k}t^k$ for each $x_j$, where the $x_{j,k}$ vary over $F$. If $r$ is the highest degree (in terms of $t$) of the coefficients of $f$, then combining like powers of $t$, $f = \sum_{k=0}^{ds+r}f_kt^k$, where the $f_k$ are polynomials in the variables $x_{j,k}$. Each $f_k$ is a linear combination of monomials of degree $d$, so the $f_k$ are homogeneous polynomials of degree $d$ in $n(s+1)$ variables over $F$.

We can apply Theorem \ref{thm:L-N} if $n(s+1)>(ds+r+1)d^i$, or equivalently, if $(n-d^{i+1})s>(r+1)d^i-n$. By assumption, $n>d^{i+1}$, so this inequality is satisfiable by picking $s$ large enough. Then the theorem gives us a nontrivial common zero for the $f_k$ in $F^{n(s+1)}$, which supplies a nontrivial zero of $f$ in $E^n$.
\end{proof}

\subsection{Valued fields}\label{subsec:VF}

The goal of this section is to prove that the field of formal Laurent series over any finite field is $C_{2}$. Along the way we will develop the theory of discrete valued fields and their completions, which will allow us to define the $p$-adic fields. We begin with some definitions.

\begin{defin}\label{def:LOG}
A \emph{linearly ordered abelian group} is an abelian group $G$, together with an order relation $\leq$, such that for all $a,b,c\in G$,
\begin{enumerate}
\item $\leq$ is a linear order on $G$, that is,
\begin{enumerate}
\item$a \leq b$ or $b \leq a$,
\item if $a\leq b$ and $b\leq a$, then $a=b$,
\item if $a\leq b$ and $b\leq c$, then $a\leq c$, and
\end{enumerate}
\item if $a \leq b$, then $a + c \leq b + c$. 
\end{enumerate}
\end{defin}

We will sometimes write $b\geq a$ instead of $a\leq b$, and we will write $a<b$ to mean $a\leq b$ and $a\neq b$.

\begin{defin}\label{def:VF}
Let $G$ be a linearly ordered abelian group, where we extend the order and group operation on $G$ to include $\infty$, so that for all $b\in G\cup\{\infty\}$, $b\leq\infty$ and $b + \infty = \infty$. Given a field $F$ and a map $\mathfrak{v}:F\rightarrow G \cup \{\infty\}$, such that for all $a,b\in F$,
\begin{enumerate}
\item $\mathfrak{v}(a)=\infty$ if and only if $a=0$,
\item $\mathfrak{v}(ab)=\mathfrak{v}(a) + \mathfrak{v}(b)$, and
\item $\mathfrak{v}(a+b)\geq \min(\mathfrak{v}(a),\mathfrak{v}(b))$,
\end{enumerate}
we call $F$ a \emph{valued field} and $\mathfrak{v}$ a \emph{valuation} on $F$.
\end{defin}

\begin{exa}\label{exa:TrivialVF}
Given a field $F$ and a linearly ordered abelian group $G$, $F$ can be equipped with the trivial valuation $\mathfrak{v}:F\rightarrow G\cup\{\infty\}$ which maps $0_F$ to $\infty$ and all other elements to $0_G$.
\end{exa}

\begin{exa}\label{exa:v_p}
For any prime $p$, define $\mathfrak{v}_p:\mathbb{Z}\setminus\{0\}\rightarrow\mathbb{N}$ by $\mathfrak{v}_p(a) = k$, where $k$ is the maximum integer such that $p^k\,|\,a$. We can extend $\mathfrak{v}_p$ to a function $\mathbb{Q}\rightarrow\mathbb{Z}\cup\{\infty\}$ by setting $\mathfrak{v}_p(0)=\infty$ and $\mathfrak{v}_p(\frac{a}{b}) = \mathfrak{v}_p(a) - \mathfrak{v}_p(b)$. It is easy to verify that this extension is well-defined, and that $\mathbb{Q}$ and $\mathfrak{v}_p$ satisfy the conditions given in Definition \ref{def:VF}.
\end{exa}

\begin{exa}\label{exa:v_t}
For any field $F$, we can define a similar valuation on the field of rational functions over $F$. Define $\mathfrak{v}_t:F[t]\setminus\{0\}\rightarrow\mathbb{N}$ by $\mathfrak{v}_t(f(t)) = k$, where $k$ is the maximum integer such that $t^k\,|\,f(t)$. If $f(t) = a_nt^n + \hdots + a_0$, $\mathfrak{v}_t(f(t))$ is the minimum $k$ such that $a_k\neq 0$. We can extend $\mathfrak{v}_t$ to a function $F(t)\rightarrow\mathbb{Z}\cup\{\infty\}$ by setting $\mathfrak{v}_t(0)=\infty$ and $\mathfrak{v}_t(\frac{f}{g}) = \mathfrak{v}_t(f) - \mathfrak{v}_t(g)$. It is easy to verify that this extension is well-defined, and that $F(t)$ and $\mathfrak{v}_t$ satisfy the conditions given in Definition \ref{def:VF}.
\end{exa}

We can immediately establish some simple facts about valuations.

\begin{lem}\label{lem:VF}
Let $F$ be a valued field with valuation $\mathfrak{v}:F\rightarrow G\cup\{\infty\}$. For all $a,b\in F$,
\begin{enumerate}
\item $\mathfrak{v}(1)=0$,
\item $\mathfrak{v}(a^{-1}) = -\mathfrak{v}(a)$
\item $\mathfrak{v}(-a)=\mathfrak{v}(a)$,
\item if $\mathfrak{v}(a)\neq \mathfrak{v}(b)$, then $\mathfrak{v}(a+b)=\min(\mathfrak{v}(a),\mathfrak{v}(b))$.
\end{enumerate}
\end{lem}
\begin{proof}
Property (2) in Definition \ref{def:VF} shows that $\mathfrak{v}$ is a homomorphism from $F^*$ to $G$. Thus it maps the identity of $F^*$ to the identity of $G$: $\mathfrak{v}(1) = 0$. Also, $\mathfrak{v}$ maps inverses in $F^*$ to inverses in $G$: $\mathfrak{v}(a^{-1}) = -\mathfrak{v}(a)$.

Now by property (1a) in Definition~\ref{def:LOG}, either $0\leq \mathfrak{v}(-1)$ or $\mathfrak{v}(-1)\leq 0$. Suppose $0\leq \mathfrak{v}(-1)$. Then by property (2), $\mathfrak{v}(-1) \leq \mathfrak{v}(-1)+\mathfrak{v}(-1) = \mathfrak{v}(-1\cdot-1) = \mathfrak{v}(1) = 0$, so $\mathfrak{v}(-1) = 0$ by property (1b). The same argument holds if we assume $\mathfrak{v}(-1)\leq 0$, in which case $\mathfrak{v}(-1) \geq \mathfrak{v}(-1) + \mathfrak{v}(-1) = 0$, so $\mathfrak{v}(-1) = 0$. Hence for all $a\in F$, $\mathfrak{v}(-a) = \mathfrak{v}(-1) + \mathfrak{v}(a) = \mathfrak{v}(a)$.

If $\mathfrak{v}(a)\neq \mathfrak{v}(b)$, say $\mathfrak{v}(a)< \mathfrak{v}(b)$, then $\mathfrak{v}(a+b) \geq \min(\mathfrak{v}(a),\mathfrak{v}(b)) = \mathfrak{v}(a) = \mathfrak{v}(a + b + -b) \geq \min(\mathfrak{v}(a+b),\mathfrak{v}(-b))$. But we have the strict inequality $\mathfrak{v}(a) < \mathfrak{v}(b) = \mathfrak{v}(-b)$, so $\min(\mathfrak{v}(a+b),\mathfrak{v}(-b)) = \mathfrak{v}(a+b)$, and thus we have equality above: $\mathfrak{v}(a+b) = \min(\mathfrak{v}(a),\mathfrak{v}(b))$. The same argument holds if $\mathfrak{v}(b) < \mathfrak{v}(a)$.
\end{proof}

All valued fields come with a few interesting structures:
\begin{itemize}
\item $\mathfrak{v}$ is a homomorphism from $F^*$ to $G$, so the image $\mathfrak{v}(F^*)$ is a linearly ordered subgroup of $G$, called the \emph{value group}. Note that since $\mathfrak{v}(0) = \infty$, $\mathfrak{v}(F) = \mathfrak{v}(F^*)\cup\{\infty\}$.
\item We define $\mathcal{O}_{F}=\{a\in F \,|\, \mathfrak{v}(a)\geq 0\}$. The set $\mathcal{O}_{F}$ contains $1$ and $0$, and is closed under addition, multiplication, and additive inverse, so it is a subring of $F$, called the \emph{valuation ring}. We will denote the valuation ring by $\mathcal{O}$ when there is no ambiguity. For all $a\in F$, $a\in\mathcal{O}$ or $a^{-1}\in\mathcal{O}$, since if $\mathfrak{v}(a) < 0$, then $\mathfrak{v}(a) + \mathfrak{v}(a^{-1}) < \mathfrak{v}(a^{-1})$, and thus $\mathfrak{v}(a^{-1}) > \mathfrak{v}(1) = 0$. Hence the field of fractions of $\mathcal{O}$ is $F$.
\item For $a\in \mathcal{O}$, if also $a^{-1}\in \mathcal{O}$, then $\mathfrak{v}(a) + \mathfrak{v}(a^{-1}) = \mathfrak{v}(a) + -\mathfrak{v}(a) = 0$, but both $\mathfrak{v}(a)\geq 0$ and $\mathfrak{v}(a^{-1})\geq 0$, so $\mathfrak{v}(a) = \mathfrak{v}(a^{-1}) = 0$. Conversely, if $\mathfrak{v}(a) = 0$, then $\mathfrak{v}(a^{-1}) = -\mathfrak{v}(a)  = 0$, and $a^{-1}\in \mathcal{O}$. Thus $\{a\in F\,|\,\mathfrak{v}(a)=0\}$ is $\mathcal{O}^*$, the group of units of $\mathcal{O}$.
\item We define $I_1=\{a\in \mathcal{O} \,|\, \mathfrak{v}(a)> 0\}$. The set $I_1$ is an ideal in $\mathcal{O}$, since it is closed under addition, and if $a\in I_1$, $b\in \mathcal{O}$, then $\mathfrak{v}(ab) = \mathfrak{v}(a) + \mathfrak{v}(b) > 0$, so $ab\in I_1$. Moreover, it is a maximal ideal, since if $I$ is an ideal in $\mathcal{O}$ properly containing $I_1$, then there is $a\in I$ with $\mathfrak{v}(a) = 0$, so $a$ is a unit, and thus $I=\mathcal{O}$.
\item We define $\overline{F} = \mathcal{O}/I_1$. Since $I_1$ is a maximal ideal, $\overline{F}$ is a field, called the \emph{residue class field}. The residue class of $a\in \mathcal{O}$ mod $I_1$ is denoted $\overline{a}$.
\end{itemize}

\subsubsection*{Discrete Valued Fields} 
The value group of a discrete valued field is isomorphic to $\mathbb{Z}$. The structure imposed by the valuation axioms allows us to complete such a field in a way analogous to how $\mathbb{R}$ is obtained by a completion of $\mathbb{Q}$. Our two main examples of valued fields, $\mathbb{F}_p(t)$ and $\mathbb{Q}$, are discrete valued fields with completions $\mathbb{F}_p((t))$ and $\mathbb{Q}_p$ respectively.

\begin{defin}\label{def:DVF}
A valued field $F$ with valuation $\mathfrak{v}$ is called \emph{discrete} if its value group $\mathfrak{v}(F^*)$ is isomorphic to $\mathbb{Z}$ with its usual ordering. Call the isomorphism $\phi$. An element $\pi\in F$ is called a \emph{prime element} if $\phi(\mathfrak{v}(\pi)) = 1$.
\end{defin} 

For notational convenience, we will suppress the isomorphism $\phi$ and identify the value groups of our discrete valued fields with the integers.

The fields $\mathbb{Q}$ and $F(t)$ with valuations $\mathfrak{v}_p$ and $\mathfrak{v}_t$ defined in Examples ~\ref{exa:v_p} and \ref{exa:v_t} are discrete valued fields.

\begin{lem}\label{lem:ideals}
Let $F$ be a discrete valued field with valuation $\mathfrak{v}$. Let $\pi$ be a prime element in $F$. Then for $n\in\mathbb{Z}$, any $a\in F^*$ with $\mathfrak{v}(a) = n$ can be written as $u \pi^n$ with $u\in \mathcal{O}^*$, and
for all $n\geq 1$, the set $I_n = \{a\in F \,|\, \mathfrak{v}(a) \geq n\}$ is a principal ideal of $\mathcal{O}$, generated by $\pi^n$.
\end{lem}
\begin{proof}
For all $a\in F^*$, let $n = \mathfrak{v}(a)$. Then $\mathfrak{v}(a\pi^{-n}) = \mathfrak{v}(a) + \mathfrak{v}(\pi^{-n}) = n - n = 0$, so $u = a\pi^{-n}$ is a unit in $\mathcal{O}$. We can write $a = u\pi^n$, with $u\in\mathcal{O}^*$. 

Now for all $n\geq 1$, $I_n$ is closed under addition, and if $a\in I_n$, $y\in \mathcal{O}$, then $\mathfrak{v}(ay) = \mathfrak{v}(a) + \mathfrak{v}(y) \geq n + 0=n$, so $ay\in I_n$. Thus $I_n$ is an ideal. For all $a\in I_n$, $a=u\pi^{\mathfrak{v}(a)} = u \pi^n \pi^{\mathfrak{v}(a)-n}$ for some unit $u$, and $\pi^n\in I_n$, so $I_n$ is generated by $\pi^n$.
\end{proof}

For all $n\geq 0$, we define the ring $\mathcal{O}_n = \mathcal{O}/I_{n+1}$, the ring of cosets mod $\pi^{n+1}$. Note that $\mathcal{O}_0 = \mathcal{O}/I_{1} = \overline{F}$. Choose a set of coset representatives $A = \{\alpha_i\}\subset \mathcal{O}$ for the elements of $\overline{F}$. For any $a\in \mathcal{O}$, if $\overline{a} = \overline{\alpha_{i_0}}$, then $a-\alpha_{i_0} \in M$, so $a = \alpha_{i_0} + a_1\pi$ for some $a_1 \in \mathcal{O}$. Repeating this process, if $\overline{a_1} = \overline{\alpha_{i_1}}$, then $a_1 = \alpha_{i_1} + a_2\pi$ for some $a_2 \in \mathcal{O}$, so $a = \alpha_{i_0} + \alpha_{i_1}\pi + a_2\pi^2$. Thus $a \equiv \alpha_{i_0} +  \alpha_{i_1}\pi\,\,(\mbox{mod}\,\,\pi^2)$, and $\alpha_{i_0} +  \alpha_{i_1}\pi$ is the image of $a$ in $\mathcal{O}_1$.

Continuing in this way, we can represent any element of $\mathcal{O}_n$ uniquely as $\alpha_{i_0} + \alpha_{i_1}\pi  + \hdots +  \alpha_{i_n}\pi^n$ for $\alpha_{i_0},\hdots,\alpha_{i_n}\in A$. For all $n>0$, let $\phi_n$ be the canonical homomorphism $\mathcal{O}_n\rightarrow \mathcal{O}_{n-1}$ which maps an element of $\mathcal{O}_n$ to its coset mod $\pi^n$. Under this representation of $\mathcal{O}_n$, $\phi_n$ simply omits the leading term $ \alpha_{i_n}\pi^n$. 

\begin{defin}\label{def:completion}
Let $F$ be a discrete valued field. Define the rings $\mathcal{O}_n$ for all $n\geq 0$ and homomorphisms $\phi_n$ for all $n > 0$ as above. The \emph{completion} of $\mathcal{O}$, $\widehat{\mathcal{O}}$, is defined by 
\[
\widehat{\mathcal{O}} = \{(a_0,a_1,\hdots)\in\prod_{n\geq0}\mathcal{O}_n\,|\,\forall n>0, \phi_n(a_n)=a_{n-1}\}.
\]
$\widehat{\mathcal{O}}$ is a subring of the product ring $\prod_{n\geq0}\mathcal{O}_n$. The \emph{completion} of $F$, $\widehat{F}$, is defined to be the field of fractions of $\widehat{\mathcal{O}}$.
\end{defin}

Those familiar with category theory will recognize this as the inverse limit construction.

Using the representation of $\mathcal{O}_n$ as $\alpha_{i_0} + \alpha_{i_1}\pi  + \hdots + \alpha_{i_n}\pi^n$ for $\alpha_{i_0},\hdots,\alpha_{i_n}\in A$, an arbitrary element of the completion $a\in\widehat{\mathcal{O}}$ looks like $a = (\alpha_{i_0},\,\alpha_{i_0}+\alpha_{i_1}\pi,\,\alpha_{i_0}+\alpha_{i_1}\pi+\alpha_{i_2}\pi^2,\,\hdots)$ with $\alpha_{i_0},\alpha_{i_1},\hdots\in A$. For convenience, we will express this element as an infinite sum: $\alpha_{i_0}+\alpha_{i_1}\pi+\alpha_{i_2}\pi^2+\hdots$, which is well-defined, since the $k^{th}$ coordinate of $a$ provides the coefficient $\alpha_{i_k}$ of $\pi^k$, while agreeing with the previous coordinates on the coefficients $\alpha_{i_j}$ for all $j<k$. 

\begin{rmk}\label{rmk:absolute value}
The completion of $F$, which we have defined purely algebraically, is isomorphic to the analytic completion of $F$ under the metric induced by the absolute value $||a||_\mathfrak{v} = 2^{-\mathfrak{v}(a)}$. The elements of $\widehat{\mathcal{O}}$ correspond to equivalence classes of Cauchy sequences under this metric.
\end{rmk}

\begin{exa}\label{exa:F((t))}
For any field $F$, the field of rational functions $F(t)$ with the valuation $\mathfrak{v}_t$ defined in Example~\ref{exa:v_t} is a discrete valued field. We will see that its completion is $F((t))$, the field of formal Laurent series over $F$.

We have $\mathfrak{v}_p(t) = 1$, and we will choose $\pi = t$ as a prime element. Writing all fractions in lowest terms, we have $\mathcal{O} = \{\frac{f}{g}\in F(t)\,|\,t\nmid g\}$, with maximal ideal $I_1 = \{\frac{f}{g}\in F(t)\,|\,t\,|\,f, t\nmid g\}$. 

Now for any rational function $\frac{f}{g}\in F(t)$, with $t\nmid g$, let $h\in F$ be the inverse of the constant term of $g$. Then $hg \equiv 1$ (mod $t$). Let $l = \frac{f(hg-1)}{t}\in F[t]$. Then $\frac{f}{g} + \frac{tl}{g} = \frac{f + tl}{g} = \frac{f + f(hg - 1)}{g} = \frac{fhg}{g} = fh \in F[t]$. We chose $\frac{tl}{g}\in I_1$, so this shows that any element of $\mathcal{O}$ is congruent to an an element of $F[t]$ mod $I_1$. Since $t\in I_1$, any element of $\mathcal{O}$ is congruent to an element of $F$ mod $I_1$. 

Hence the residue class field $\overline{F(t)}$ is isomorphic to $F$, and we can take $F$ as our set of coset representatives.

Now we will take the completion of $\mathcal{O}$. The resulting ring is $\widehat{\mathcal{O}} = F[[t]]$, the field of formal power series over $F$. As we saw above, the elements of the completion can be uniquely represented in the form $\alpha_0 + \alpha_1t + \alpha_2t^2 + \hdots$, with each $\alpha_i\in F$.

The fraction field of $F[[t]]$ is the completion $\widehat{F(t)}$. Let $x = \frac{\alpha_0 + \alpha_1t + \hdots}{\beta_0 + \beta_1t + \hdots} \in \widehat{F(t)}$. Let $k$ be the least integer such that $\beta_k\neq 0$. Now factoring out the leading term $\beta_kt^k$, we can write 
\[
x = \left(\frac{1}{\beta_kt^{k}}\right)\left(\frac{\alpha_0 + \alpha_1t + \hdots}{1 + \gamma_{1}t + \hdots}\right),
\]
where $\gamma_i = \beta_{k+i}\beta_k^{-1}\in F$ for all $i\geq 1$.

We claim that the inverse of the denominator, $(1 + \gamma_{1}t + \hdots)^{-1}$, is an element of $F[[t]]$. We have \begin{eqnarray*}
\frac{1}{1 + \gamma_1t + \hdots} &=& \frac{1}{1 - (-\gamma_1t - \hdots)}\\
&=& 1 + (-\gamma_1t - \hdots) + (-\gamma_1t - \hdots)^2 + \hdots,
\end{eqnarray*}
applying the geometric series formula. Now for all $n\geq 0$, $t^n$ appears in only finitely many terms of the infinite sum, so the coefficient of each $t^n$ is well-defined, and this is a well-defined element of $F[[t]]$.

Letting $y = \frac{1}{1 + \gamma_1t + \hdots}\in F[[t]]$, we can write $x = \beta_k^{-1} t^{-k}y(\alpha_0 + \alpha_1t + \hdots)$, and this has the form $c_{-k}t^{-k} + \hdots + c_{-1}t^{-1} + c_0 + c_1t + \hdots$, with each $c_i\in F$. All elements of $\widehat{F(t)}$ can be uniquely represented in this form. We call the completion the field of formal Laurent series over $F$ and denote it by $F((t))$.
\end{exa}

\begin{exa}\label{exa:Q_p}
For all primes $p$, we define the field of p-adic numbers, $\mathbb{Q}_p$ to be the completion of $\mathbb{Q}$ according to the valuation $\mathfrak{v}_p$ defined in Example \ref{exa:v_p}. 

We have $\mathfrak{v}_p(p) = 1$, and we will choose $\pi = p$ as a prime element. Writing all fractions in lowest terms, we have $\mathcal{O} = \{\frac{a}{b}\in\mathbb{Q}\,|\,p\nmid b\}$, with maximal ideal $I_1 = \{\frac{a}{b}\in\mathbb{Q}\,|\,p\,|\,a, p\nmid b\}$. 

Now for any $\frac{a}{b}\in\mathbb{Q}$, with $p\nmid b$, there is some $d\in\mathbb{Z}$ such that $db \equiv 1$ (mod $p$). Let $c = \frac{a(db-1)}{p}\in\mathbb{Z}$. Then $\frac{a}{b} + \frac{pc}{b} = \frac{a + pc}{b} = \frac{a + a(db - 1)}{b} = \frac{adb}{b} = ad \in\mathbb{Z}$. We chose $\frac{pc}{b}\in I_1$, so this shows that any element of $\mathcal{O}$ is congruent to an integer mod $I_1$. Since all integer multiples of $p$ are in $I_1$, any element of $\mathcal{O}$ is congruent to one of $\{0,1,\hdots,p-1\}$ mod $I_1$. 

Thus the residue class field has $p$ elements, $\overline{\mathbb{Q}}\cong \mathbb{F}_p$, and we can take as our set of coset representatives $A = \{0,1,\hdots,p-1\}$.

Now we will take the completion of $\mathcal{O}$. The resulting ring is $\widehat{\mathcal{O}} = \mathbb{Z}_p$, the $p$-adic integers. Its elements can be uniquely represented in the form $\alpha_0 + \alpha_1p + \alpha_2p^2 + \hdots$, with each $\alpha_i\in A$.

The p-adic field $\mathbb{Q}_p$ is the field of fractions of $\mathbb{Z}_p$. By a similar argument to the one in Example~\ref{exa:F((t))}, elements of $\mathbb{Q}_p$ can be uniquely represented in the form $c_{-k}p^{-k} + \hdots + c_{-1}p^{-1} + c_0 + c_1p + \hdots$, with each $c_i\in A$.
\end{exa}

The elements of $\mathbb{Q}_p$ look very similar to the elements of $\mathbb{F}_p((t))$. They can be thought of as formal Laurent series in a single ``variable'', $p$, with coefficients in $\mathbb{F}_p$. The similarity between the two fields is significant to us because it was the motivation for Artin's conjecture that $\mathbb{Q}_p$ is $C_2$ (see Theorem~\ref{thm:F((t))}). However, the fields $\mathbb{Q}_p$ and $\mathbb{F}_p((t))$ are not isomorphic; their arithmetic is very different. In particular, $\mathbb{Q}_p$ has characteristic $0$, while $\mathbb{F}_p((t))$ has characteristic $p$. Informally speaking, elements of $\mathbb{Q}_p$ add and multiply with carries, while elements of $\mathbb{F}_p((t))$ do not. Nevertheless, the Ax-Kochen Principle demonstrates that the similarity between the fields is not just skin-deep.

\subsubsection*{Homogeneous Polynomials over Complete Discrete Valued Fields}
For any discrete valued field $F$ with valuation ring $\mathcal{O}$, there is a homomorphism $i: \mathcal{O}\rightarrow \widehat{\mathcal{O}}$ which maps $x\in \mathcal{O}$ to the images of $x$ in $\mathcal{O}_n$ for all $n\geq 0$. The only element of $\mathcal{O}$ divisible by all powers of $\pi$ is $0$, so $i$ is injective. Thus we regard $\mathcal{O}$ as a subring of $\widehat{\mathcal{O}}$ and $F$ as a subfield of $\widehat{F}$. If $i$ is surjective, then $\mathcal{O}\cong\widehat{\mathcal{O}}$ and $F\cong\widehat{F}$. In this case, we say that $\mathcal{O}$ and $F$ are complete.

\begin{lem}\label{lem:complete}
As one would hope, the completion of a discrete valued field is a complete discrete valued field.
\end{lem}
\begin{proof}
Let $F$ be a discrete valued field with valuation $\mathfrak{v}$, valuation ring $\mathcal{O}$, and prime element $\pi$. Let $\widehat{\mathcal{O}}$ be the completion of $\mathcal{O}$, and let $\widehat{F}$ be the completion of $F$ (the field of fractions of $\mathcal{O}$). First we must show that $\widehat{F}$ is a discrete valued field. 

Define a function $\widehat{\mathfrak{v}}:\widehat{O}\setminus\{0\}\rightarrow\mathbb{N}$ which takes an element of $\widehat{O}$, $(a_0,a_1,\hdots)$ to the least integer $k$ such that $a_k\neq 0$. Note that $\widehat{\mathfrak{v}}$ agrees with $\mathfrak{v}$ on the subring $\mathcal{O}$, since for $a\in\mathcal{O}$, the image of $a$ in $\widehat{\mathcal{O}}$ is $(a_0,a_1,\hdots)$ where $a_i$ is the image of $a$ mod $\pi^{i+1}$, and $k$ is the least integer such that $a_k \neq 0$ if and only if $k$ is the greatest integer such that $\pi^k\,|\,a$.

Now we can extend $\widehat{\mathfrak{v}}$ to a function $\widehat{F}\rightarrow\mathbb{Z}\cup\{\infty\}$ by setting $\widehat{\mathfrak{v}}(0) = \infty$ and $\widehat{\mathfrak{v}}(\frac{a}{b}) = \widehat{\mathfrak{v}}(a) - \widehat{\mathfrak{v}}(b)$. It is easy to verify that $\widehat{\mathfrak{v}}$ satisfies the valuation axioms. Thus $\widehat{F}$ is a discrete valuation field.

Since $\widehat{\mathfrak{v}}$ agrees with $\mathfrak{v}$ on elements of $\mathcal{O}$, we can choose the same element $\pi$ as a prime element of $\widehat{F}$. Then for each $n\geq 1$, the ideal $\widehat{I}_n$ consists of all elements which are $0$ in their first $n$ coordinates, and for $n\geq 0$, elements of the ring $\widehat{\mathcal{O}}_n = \widehat{\mathcal{O}}/\widehat{I}_{n+1}$ are cosets consisting of elements which agree on their first $n+1$ coordinates.

The inclusion $i:\mathcal{O}\rightarrow\widehat{\mathcal{O}}$ induces homomorphisms $i_n:\mathcal{O}_n\rightarrow\widehat{\mathcal{O}}_n$. Included in $\widehat{\mathcal{O}}$, the ideal $I_n$ consists of all elements which are $0$ in their first $n$ coordinates, and elements of the ring $\mathcal{O}_n$ are cosets consisting of elements which agree on their first $n+1$ coordinates, so the $i_n$ are bijections, and therefore isomorphisms.

Now $\widehat{\mathcal{O}}_n\cong\mathcal{O}_n$ for all $n>0$, and since the prime element $\pi$ is the same, $\widehat{\phi_n}$ and $\phi_n$ act on $\widehat{\mathcal{O}}_n$ and $\mathcal{O}_n$ in the same way. Hence the completions of $\widehat{\mathcal{O}}$ and $\mathcal{O}$ are isomorphic, that is, the completion of $\widehat{\mathcal{O}}$ is isomorphic to $\widehat{\mathcal{O}}$, and thus $\widehat{\mathcal{O}}$ and $\widehat{F}$ are complete.
\end{proof}

We now return to studying homogeneous polynomials and the $C_i$ properties. Let $f$ be a homogeneous polynomial over a discrete valued field $F$ of degree $d$, and suppose that all of the coefficients of $f$ are in the valuation ring $\mathcal{O}$. Fixing an $m\geq 0$ and a prime element $\pi$, we will denote by $\overline{f}$ the reduction of $f$ mod $\pi^{m+1}$, the coefficients of which are in the quotient ring $\mathcal{O}_m$. Note that $\overline{f}$ is either a homogeneous polynomial of degree $d$ or the zero polynomial, if all coefficients are divisible by $\pi^{m+1}$.

\begin{defin}\label{def:primitive_zero}
Let $f$ be a homogeneous polynomial in $n$ variables over a discrete valued field $F$ with prime element $\pi$, and suppose that all of the coefficients of $f$ are in the valuation ring $\mathcal{O}$. We will call a zero $(\alpha_1,\hdots,\alpha_n)\in\mathcal{O}^n$ of $f$ \emph{primitive} if for some $j$, $\pi\nmid\alpha_j$. Similarly, for $m\geq 0$, we call a zero  $(\overline{\alpha_1},\hdots,\overline{\alpha_n})\in\mathcal{O}_m^n$ of $\overline{f}$ \emph{primitive} if for some $j$, $\overline{\pi}\nmid \overline{\alpha_j}$, where $\overline{\pi}$ is the coset of $\pi$ mod $\pi^{m+1}$.
\end{defin}

Our goal is to reduce the problem of finding zeros of $f$ in $\mathcal{O}^n$ to the problem of finding zeros of $\overline{f}$ in $\mathcal{O}^n_m$ for all $m\geq 0$. The advantage of working with primitive zeros is that a primitive zero cannot become trivial upon reduction mod $\pi^{m+1}$.

\begin{thm}[{\cite[Theorem 4.7]{Greenberg}}]\label{thm:VF-common_zeros}
Let $F$ be a complete discrete valued field with prime element $\pi$. Suppose that the residue class field $\overline{F}$ is finite. Then a homogeneous polynomial $f$ over $\mathcal{O}$ of degree $d$ in $n$ variables has a primitive zero in $\mathcal{O}^n$ if and only if $\overline{f}$ has a primitive zero $\mathcal{O}_m^n$ for all $m\geq 0$.
\end{thm}
\begin{proof}

Suppose $(\alpha_1,\hdots,\alpha_n)\in\mathcal{O}^n$ is a primitive zero of $f$. Then for all $m\geq 0$, reduction mod $\pi^{m+1}$ is a homomorphism $\mathcal{O}\rightarrow\mathcal{O}_m$, so $\overline{f}(\overline{\alpha_1},\hdots,\overline{\alpha_n}) = 0$. Since $(\alpha_1,\hdots,\alpha_n)$ is primitive, for some $j$, $\pi\,\nmid\,\alpha_j$. Suppose $\overline{\pi}\,|\,\overline{\alpha_j}$. Then there is some $\overline{d}\in\mathcal{O}_m$ such that $\overline{\pi}\overline{d} = \overline{\alpha_j}$, so lifting to $\mathcal{O}$, $\pi d - \alpha_j \in I_{m+1}\subset I_1$. Now $\pi d\in I_1$, so $\alpha_j\in I_1$, and thus $\pi$ divides $\alpha_j$. This is a contradiction, so $\overline{\pi}\,\nmid\,\overline{\alpha_j}$, and $(\overline{\alpha_1},\hdots,\overline{\alpha_n})$ is a primitive zero in $\mathcal{O}_m^n$.

Conversely, for all $m\geq 0$, define $S_m\subseteq\mathcal{O}_m^n$ to be the set of primitive zeros of $\overline{f}$ in $\mathcal{O}_m^n$, and suppose that $S_m$ is nonempty for all $m$. If $\alpha\in S_{m+1}$ is a primitive zero mod $\pi^{m+2}$, then its image mod $\pi^{m+1}$ is also a primitive zero; that is, $\phi_{m+1}(\alpha)\in S_{m}$, so $\phi_{m+1}(S_{m+1}) \subseteq S_{m}$. For all $j<m$, define $S_{m,j} = \phi_{j+1}(\phi_{j+2}(\hdots \phi_m(S_m))) \subseteq S_j$. The set $S_{m,j}$ is the set of primitive zeros in $\mathcal{O}_j^n$ which lift to primitive zeros in $\mathcal{O}_m^n$. Since all the $S_m$ are nonempty, all the $S_{m,j}$ are nonempty. 

For all $k\geq 0$, define $T_k = \bigcap_{m > k} S_{m,k}$. $T_k$ is the set of all solutions in $\mathcal{O}_k^n$ which lift to solutions in $\mathcal{O}_m^n$ for all $m>k$. Since $\overline{F}$ is finite, all of the $\mathcal{O}_m$ are finite. The chain $S_{k+1,k}\supseteq S_{k+2,k}\supseteq\hdots$ must break off at some $l>k$, with $S_{m,k} = S_{l,k}$ for all $m\geq l$, since the decreasing sequence of integers $|S_{k+1,k}|\geq |S_{k+2,k}|\geq\hdots$ is bounded below by 1. Thus $T_k = S_{l,k}$ is nonempty for all $k$.

Pick a zero $a_0 = (\alpha_{0,0},\hdots,\alpha_{0,n}) \in T_0$. For all $m$, $a_0$ lifts to a solution $a_m = (\alpha_{m,0},\hdots,\alpha_{m,n})$ in $\mathcal{O}_m^n$. That is, assuming that $a_{i-1}\in T_{i-1}$, we can choose $a_i \in \phi_{i}^{-1}(a_{i-1}) \cap T_i$. By construction, the sequence $(a_0, a_1, \hdots)$ satisfies $\phi_m(a_m) = a_{m-1}$, so the sequences $(\alpha_{0,i},\alpha_{1,i},\hdots)$ are elements of the completion $\widehat{\mathcal{O}}$ for all $i$. 

Since $a_m\in S_m$, $a_m$ is a zero of $\overline{f}$ mod $\pi^{m+1}$ for all $m\geq 0$. Hence, viewing $f$ as a polynomial in the completion by the isomorphism between the complete ring $\mathcal{O}$ and $\widehat{\mathcal{O}}$, 
\[
f((\alpha_{0,0},\alpha_{1,0},\hdots), \hdots, (\alpha_{0,n},\alpha_{1,n},\hdots)) = (\overline{f}(\alpha_{0,0},\hdots,\alpha_{0,n}), \overline{f}(\alpha_{1,0},\hdots,\alpha_{1,n}),\hdots) = 0,
\]
and this is a primitive zero of $f$ in $\widehat{\mathcal{O}}^n$. But $\mathcal{O}\cong\widehat{\mathcal{O}}$, so each $(\alpha_{0,i}, \alpha_{1,i}, \hdots)$ corresponds to an element of $\mathcal{O}$, and this zero corresponds to a primitive zero of $f$ in $\mathcal{O}^n$.
\end{proof}

Theorem~\ref{thm:VF-common_zeros} also holds without the assumption that $\overline{F}$ is finite, but the proof of the general version is more difficult, and we will only need the finite case for the Ax-Kochen Theorem.

\begin{thm}[{\cite[Corollary 4.9]{Greenberg}}]\label{thm:F((t))}
If $F$ is a finite field, then $F((t))$ is $C_2$. 
\end{thm}
\begin{proof}
It suffices to consider homogeneous polynomials with coefficients in $F[[t]]$, the valuation ring of $F((t))$, since we can clear denominators. That is, for $f\in F((t))[x_1,\hdots,x_n]$ homogeneous of degree $d$ in $n$ variables, if $c$ is the minimum valuation among the coefficients of $f$, then $t^{cd}f\in F[[t]][x_1,\hdots,x_n]$ is also homogeneous of degree $d$ in $n$ variables. If $(\alpha_1,\hdots,\alpha_n)$ is a nontrivial zero of $t^{cd}f$, then $f(t^{c}\alpha_1,\hdots,t^{c}\alpha_n) = t^{cd}f(\alpha_1,\hdots,\alpha_n) = 0$ by Remark \ref{rmk:homconst}, so $(t^{c}\alpha_1,\hdots,t^{c}\alpha_n)$ is a nontrivial zero of $f$.

Let $f$ be a homogeneous polynomial over $F[[t]]$ of degree $d$ in $n$ variables, where $n>d^2$. The residue class field $\overline{F((t))} \cong F$ is finite, and $F((t))$ is complete by Lemma~\ref{lem:complete}, so we can apply Theorem \ref{thm:VF-common_zeros}. Since any primitive zero in $F[[t]]^n$ is a nontrivial in $F((t))^n$, it suffices to find a primitive zero of $\overline{f}$ in the residue ring mod $t^{m+1}$ for all $m\geq 0$. 

Fixing $m\geq 0$, let $\widetilde{f}$ be the polynomial obtained by ignoring the terms of degree greater than $m$ in each coefficient of $f$. Each coefficient of $\widetilde{f}$ is then a polynomial in $t$ of degree at most $m$. Now $\widetilde{f}$ is either the zero polynomial or a homogeneous polynomial of degree $d$ in $n$ variables. In the first case, each coefficient of $f$ is divisible by $t^{m+1}$, so reducing mod $t^{m+1}$, $\overline{f}$ is the zero polynomial, which clearly has a primitive zero, and we are done.

Otherwise, we will view $\widetilde{f}$ as a polynomial over $F(t)$. By Corollary \ref{cor:FF}, $F$ is $C_1$, and by Theorem \ref{thm:trans_ext}, $F(t)$ is $C_2$. Since $n>d^2$, $\widetilde{f}$ has a nontrivial zero $(\alpha_1,\hdots,\alpha_n)$ in $F(t)$. Using the homogeneity of $\widetilde{f}$, we can normalize to find another zero in $F[t]$ which is primitive. 

Let $\alpha_j$ be the coordinate with minimum (possibly negative) valuation, and let $c = \mathfrak{v}(\alpha_j)$. Now let $(\beta_1,\hdots,\beta_n) = (t^{-c}\alpha_1,\hdots,t^{-c}\alpha_n)$. All of the $\beta_i$ are elements of $F[t]$, since $\mathfrak{v}(\beta_i) = \mathfrak{v}(t^{-c}) + \mathfrak{v}(\alpha_i) \geq -c + c = 0$. 

Now $t^{cd}\widetilde{f}(\beta_1,\hdots,\beta_n) = \widetilde{f}(t^c\beta_1,\hdots,t^c\beta_n) = \widetilde{f}(\alpha_1,\hdots,\alpha_n) = 0$. Now $F[t]$ is a domain so since $t^{cd}\neq 0$, $\widetilde{f}(\beta_1,\hdots,\beta_n) = 0$. Moreover, $\mathfrak{v}(\beta_j) = -c + c = 0$, so $t\nmid\beta_j$, and $(\beta_1,\hdots,\beta_n)$ is a primitive zero.

Finally, since $\widetilde{f}(\beta_1,\hdots,\beta_n) = 0$, and $\widetilde{f}$ corresponds to $\overline{f}$ mod $t^{m+1}$, viewing $\beta_1,\hdots,\beta_n$ as elements of $F[[t]]$ by the natural inclusion, $(\overline{\beta_1},\hdots,\overline{\beta_n})$ is a zero of $\overline{f}$ mod $t^{m+1}$. Since the zero is primitive, its reduction mod $t^{m+1}$ is also primitive, as was to be shown.
\end{proof}

As a special case of Theorem \ref{thm:F((t))}, we have that $\mathbb{F}_p((t))$ is $C_2$ for all primes $p$. Because of the resemblance between the fields $\mathbb{F}_p((t))$ and $\mathbb{Q}_p$, Artin conjectured that $\mathbb{Q}_p$ is also $C_2$ for all primes $p$. This conjecture turned out to be false, but Ax and Kochen were able to prove a weaker statement: for each degree $d$ there exists a finite set of primes $P(d)$ such that the $C_2$ property holds for polynomials of degree $d$ in $\mathbb{Q}_p$ for all $p\notin P(d)$.

The Ax-Kochen Theorem is a corollary of a much more general Ax-Kochen Principle: any first-order logical statement which is true of all but finitely many of the fields $\mathbb{F}_p((t))$ is true of all but finitely many of the fields $\mathbb{Q}_p$. This statement is what could be called a ``meta-theorem'', since it quantifies over logical statements. In order to prove it we will need to develop techniques for reasoning about logical statements in general.

\newpage

\section{The Language of Model Theory}\label{sec:MT}
\subsection{Languages, Models, and Theories}\label{subsec:LMT}

Model theory is concerned with the study of mathematical structures and the logical statements about them. Logical statements about a structure are built from the familiar boolean operators and quantifiers, but they must also refer to the distinguished elements, functions, and relations which are inherent to the structure in question. Thus we work in terms of formal languages of symbols representing these elements, functions, and relations.

\begin{defin}\label{def:language}
A \emph{language} is the union of
\begin{itemize}
\item $\mathcal{C}$, a set of constant symbols, 
\item $\mathcal{F}$, a set of function symbols, with an integer $n_f>0$ for each $f\in\mathcal{F}$, and
\item $\mathcal{R}$, a set of relation symbols, with an integer $n_R>0$ for each $R\in\mathcal{R}$.
\end{itemize}
\end{defin}

The integers $n_f$ and $n_R$ are called the arities of the corresponding functions and relations. An $n$-ary function takes $n$ arguments, and an $n$-ary relation is a relation on $n$ elements. Most of the function and relation symbols we encounter will have $n = 1$ or $n = 2$, called unary and binary respectively.

\begin{defin}\label{def:structure}
Let $\mathcal{L} = \mathcal{C}\cup\mathcal{F}\cup\mathcal{R}$ be a language. An \emph{$\mathcal{L}$-structure} $\mathcal{M}$ is
\begin{itemize}
\item a set $M\neq\emptyset$, the domain,
\item an element $c^\mathcal{M}\in M$ for all $c\in \mathcal{C}$,
\item a function $f^\mathcal{M}: M^{n_f}\rightarrow M$ for all $f\in \mathcal{F}$, and
\item a relation $R^\mathcal{M} \subseteq M^{n_R}$ for all $R\in \mathcal{R}$.
\end{itemize}
\end{defin}

The elements, functions, and relations $c^\mathcal{M}$, $f^\mathcal{M}$, and $R^\mathcal{M}$ are called the interpretations of the $\mathcal{L}$-symbols in $\mathcal{M}$. The distinction between a symbol and its interpretation in a given structure is very important. This division between syntax and semantics will allow us to define and study logical theories independently of any particular structure.

\begin{exa}\label{exa:L_G}
Let $\mathcal{L}_G$ be the language $\{\cdot,e\}$, where $\cdot$ is a binary function symbol and $e$ is a constant symbol. These symbols are necessary to describe the theory of groups, and the symbols of $\mathcal{L}_G$ can be interpreted in any group. For instance, the group $\left<\mathbb{Z},+,0\right>$ is an $\mathcal{L}_G$-structure under the interpretations $\cdot^\mathbb{Z} = +$ and $e^\mathbb{Z} = 0$. But any nonempty set with any binary function can also be an $\mathcal{L}_G$ structure. For example, if $X = \{a,b,c\}$, then $X$ is an $\mathcal{L}_G$ structure under the interpretations $\cdot^X: (x,y) \mapsto b$ for all $x,y\in X$, and $e^X = c$, despite the fact that $\langle X,\cdot^X,c\rangle$ is clearly not a group.
\end{exa}

A valued field is more difficult to formalize as a structure, since its definition relies on an auxiliary structure, the value group. We will use a property called cross section to deal with the value group as a substructure of the field itself.

\begin{defin}\label{def:VF-CS}
A valued field $F$ is called a \emph{valued field with cross section} if there is an injective map $i:\mathfrak{v}(F)\rightarrow F$ such that $i$ is a group homomorphism from $\mathfrak{v}(F^*)$ to $F^*$, and for all $x\in \mathfrak{v}(F)$, $\mathfrak{v}(i(x)) = x$. 
\end{defin}

Any discrete valued field can be given cross section, once we choose a prime element $\pi$, by defining the embedding $i(n) = \pi^n$ for all $n\in \mathbb{Z}$, and $i(\infty) = 0$. For all $n\in\mathbb{Z}$, we have $\mathfrak{v}(i(n)) = \mathfrak{v}(\pi^n) = n$, and $\mathfrak{v}(i(\infty)) = \mathfrak{v}(0) = \infty$. For the remainder of this thesis, we will identify the value group of all discrete valuation fields with the multiplicative group $\{\pi^n\,|\,n\in\mathbb{Z}\}$ and suppress the embedding $i$.

\begin{exa}\label{exa:L_VF}
In order to write down logical statements about valued fields (with cross section), we will need a number of symbols. Let $\mathcal{L}_{VF}$ be the language $\{+,\cdot,-,0,1,V,\leq,\mathfrak{v}\}$, where $+$ and $\cdot$ are binary function symbols, $-$ is a unary function symbol, $0$ and $1$ are constant symbols, $V$ is a unary relation symbol, $\leq$ is a binary relation symbol, and $\mathfrak{v}$ is a unary function symbol.

The cross section property will be useful so that we can refer to elements of the value group within the domain of the valued field structure. When interpreting the symbols of $\mathcal{L}_{VF}$ in a structure which is a valued field, we will use $+,\cdot,-,0,1$ to represent the field operations, the additive inverse function, and the additive and multiplicative identities, $V$ to pick out the elements of the value group (that is, $x\in V$ if and only if $x$ is in the value group), $\leq$ to represent the ordering on the value group, and $\mathfrak{v}$ to represent the valuation.

Note again that these function and relation symbols may be interpreted as any functions and relations of the appropriate arities on any set. In order to require that our $\mathcal{L}_{VF}$ structures be valued fields, we will need some logical statements, the valued field axioms. 
\end{exa}

\subsubsection*{Terms, Formulas, and Satisfaction}
The building blocks of our logical statements are the symbols of a language $\mathcal{L}$, an infinite set of variables $\mathcal{V}=\{v_1,v_2,\hdots\}$, and the formal symbols $=$, $\land$, $\lor$, $\lnot$, $\exists$, $\forall$, $($, and $)$. The symbols $\land$, $\lor$, and $\lnot$ (read as ``and'', ``or'', and ``not'') are called Boolean operators, and the symbols $\forall$ and $\exists$ (read as ``for all'' and ``there exists'') are called quantifiers. Certain finite strings of these symbols, called $\mathcal{L}$-terms, $\mathcal{L}$-formulas, and $\mathcal{L}$-sentences, can be interpreted to have semantic meaning. Intuitively, given an $\mathcal{L}$-structure, we use $\mathcal{L}$-terms to refer to elements of that structure, $\mathcal{L}$-formulas to express properties of particular elements of the structure, and $\mathcal{L}$-sentences to express properties of the structure itself.

In order to analyze these strings systematically, we define them with a specific inductive structure. The simplest are $\mathcal{L}$-terms, which are constructed from constants and variables by means of function applications.

\begin{defin}\label{def:term}
A finite string $t$ is an \emph{$\mathcal{L}$-term} if and only if
\begin{itemize}
\item it is a constant symbol, $t\in\mathcal{C}$, or
\item it is a variable, $t\in\mathcal{V}$, or
\item it has the form $f(t_1,\hdots,t_{n_f})$, where $f\in\mathcal{F}$ is a function symbol, and $t_1,\hdots,t_{n_f}$ are $\mathcal{L}$-terms.
\end{itemize}
\end{defin} 

Binary function symbols, such as $+$ or $\cdot$, will often be written in the usual (infix) way. That is, when constructing $\mathcal{L}$-terms, we will write $t_1 + t_2$ instead of $+(t_1,t_2)$ and $t_1 \cdot t_2$ instead of $\cdot(t_1,t_2)$.

If an $\mathcal{L}$-term $t$ contains variables from $v_1,\hdots,v_n$, we will often write it as $t(v_1,\hdots,v_n)$. We do not require all of the variables $v_1,\hdots,v_n$ to appear in $t$. Let $a_1,\hdots,a_n$ be elements of the domain of some $\mathcal{L}$-structure $\mathcal{M}$. We will denote by $t^\mathcal{M}(a_1,\hdots,a_n)$ the interpretation of $t$ in $\mathcal{M}$ on the elements $a_1,\hdots,a_n$. The interpretation is obtained by substituting for each $v_i$ the corresponding element $a_i$, substituting for each constant symbol $c$ its interpretation $c^\mathcal{M}$, and substituting for each function symbol $f$ its interpretation $f^\mathcal{M}$. Applying functions in the natural way, $t^\mathcal{M}(a_1,\hdots,a_n)$ is an element of the domain of $\mathcal{M}$.

Some examples of $\mathcal{L}_{VF}$-terms include $0$, $1+1$, $v_1 \cdot 1$, and $\mathfrak{v}(v_1 + v_2)$. If $t$ is $\mathfrak{v}(v_1 + v_2)$, then taking $\mathbb{Q}_3$ as an $\mathcal{L}_{VF}$-structure (with elements written as ``Laurent series'' in $3$), where we interpret $+$ as addition and $\mathfrak{v}$ as the valuation $\mathfrak{v}_3$, we have $t^{\mathbb{Q}_3}(1,2+2\cdot 3 + 3^2) = \mathfrak{v}_3(1 +^{\mathbb{Q}_3} (2+2\cdot 3 + 3^2)) = \mathfrak{v}_3(2\cdot 3^2) = 3^2$ (recall that the value group in $\mathbb{Q}_3$ as a valued field with cross section is $\{3^n\,|\,n\in\mathbb{Z}\})$. 

Next, we define $\mathcal{L}$-formulas. The simplest of these, called atomic $\mathcal{L}$-formulas, express the properties that two terms are equal or that a collection of terms satisfy a relation. General $\mathcal{L}$-formulas are constructed from atomic $\mathcal{L}$-formulas by means of Boolean operators and quantifiers.

\begin{defin}\label{def:formula}
A finite string is an \emph{$\mathcal{L}$-formula} if and only if
\begin{itemize}
\item it has the form $t_1=t_2$, where $t_1$ and $t_2$ are $\mathcal{L}$-terms, or
\item it has the form $R(t_1,\hdots,t_{n_R})$, where $R\in\mathcal{R}$ is a relation symbol, and $t_1,\hdots,t_{n_R}$ are $\mathcal{L}$-terms.
\item it has the form $\lnot\phi$, $\phi\land\psi$, $\phi\lor\psi$, $\exists v\, \phi$, or $\forall v\, \phi$, where $\phi$ and $\psi$ are $\mathcal{L}$-formulas and $v\in\mathcal{V}$ is a variable.
\end{itemize}
\end{defin}

Some binary relation symbols, such as $\leq$, will also be written in the usual (infix) way. Instead of the atomic formula $\leq(t_1,t_2)$, we will write $t_1 \leq t_2$. 

We will use parentheses for grouping in the natural way to avoid ambiguity. We will omit the formalization of this, as it is straightforward but rather time consuming.

We will also employ the standard abbreviations $\phi\rightarrow\psi$ (read as ``$\phi$ implies $\psi$'' or ``if $\phi$ then $\psi$'') for $\lnot\phi\lor\psi$ and $\phi\leftrightarrow\psi$ (read as ``$\phi$ if and only if $\psi$'') for $(\phi\rightarrow\psi)\land(\psi\rightarrow\phi)$. We could have omitted $\lor$ and $\forall$ from our definition of $\mathcal{L}$-formula, since $\phi\lor\psi$ and $\forall v\, \phi$ can be viewed as abbreviations for $\lnot(\lnot\phi\land\lnot\psi)$ and $\lnot(\exists v\, \lnot\phi)$ respectively.

Some examples of $\mathcal{L}_{VF}$-formulas include $(v_1 + v_2) + v_3 = v_1 + (v_2 + v_3)$, $\mathfrak{v}(v_1)\leq \mathfrak{v}(v_1 + v_2)$, $\forall v_1\, V(\mathfrak{v}(v_1))$, and $\lnot(v_1=0)\rightarrow(\exists v_2\, v_1\cdot v_2 = 1)$.

Upon interpreting an $\mathcal{L}$-formula in a particular $\mathcal{L}$-structure, $\mathcal{M}$, the quantifiers $\forall$ and $\exists$ are understood to quantify over the elements of $M$, the domain of $\mathcal{M}$. This is what makes the formula ``first-order''. In first-order logic, we cannot express statements like ``every bounded subset has a least upper bound'' or ``$\forall i\in\mathbb{Z}\,, x^i \neq 0$'', since the first quantifies over subsets, not elements, and the second quantifies over a specific structure, the integers. Because of this restriction, first-order logic is less expressive than other logics, but it has more structure which can be exploited mathematically.

A variable $v$ is called bound if it occurs inside a $\forall v$ or $\exists v$ quantifier. Otherwise it is called free. In the $\mathcal{L}_{VF}$-formula $\exists v_1\,v_1\leq v_2$, $v_1$ is bound, but $v_2$ is free. To avoid ambiguity, we will require that no variable occurs both free and bound in a formula, and that no variable is bound by more than one quantifier. When combining formulas, we can ensure these conditions by substituting unused variables for any variable which appears in more than one context.

If an $\mathcal{L}$-formula $\phi$ contains free variables from $v_1,\hdots,v_n$, we often write it as $\phi(v_1,\hdots,v_n)$. We do not require all of the variables $v_1,\hdots,v_n$ to appear in $\phi$. If we substitute elements $a_1,\hdots,a_n$ from the domain of an $\mathcal{L}$-structure $\mathcal{M}$ for the variables $v_1,\hdots,v_n$, and if we interpret the $\mathcal{L}$-symbols in $\mathcal{M}$ and interpret the boolean operators and quantifiers in the natural way, then $\phi(a_1,\hdots,a_n)$ is either true or false in $\mathcal{M}$. If it is true, we write $\mathcal{M}\models\phi(a_1,\hdots,a_n)$ and say that $\mathcal{M}$ satisfies $\phi(a_1,\hdots,a_n)$. Otherwise, we write $\mathcal{M}\not\models\phi(a_1,\hdots,a_n)$. 

For example, let $\phi_\mathcal{O}$ be the formula $1\leq\mathfrak{v}(v_1)$. Let $F$ be a discrete valued field with valuation $\mathfrak{v}$ and prime element $\pi$, taken as an $\mathcal{L}_{VF}$-structure in the natural way. For all $x\in F$, $F\models \phi_\mathcal{O}(x)$ if and only if $\mathfrak{v}(x)\geq \pi^0$, that is, if and only if $x$ is an element of $\mathcal{O}_F$. In this way, the formula expresses a property of elements of a valued field, namely, that an element is in the valuation ring.

\begin{defin}\label{def:satisfies}
Let $\phi$ be an $\mathcal{L}$-formula with free variables from $v_1,\hdots,v_n$. Let $\mathcal{M}$ be an $\mathcal{L}$-structure with domain $M$, and let $a_1,\hdots,a_n\in M^n$ be elements of the domain. We define $\mathcal{M}\models\phi(a_1,\hdots,a_n)$ inductively as follows:
\begin{itemize}
\item If $\phi$ is $t_1(v_1,\hdots,v_n) = t_2(v_1,\hdots,v_n)$, then $\mathcal{M}\models\phi(a_1,\hdots,a_n)$ if and only if 
\[
t_1^\mathcal{M}(a_1,\hdots,a_n)=t_1^\mathcal{M}(a_1,\hdots,a_n).
\] 
\item If $\phi$ is $R(t_1(v_1,\hdots,v_n),\hdots,t_{n_r}(v_1,\hdots,v_n))$, then $\mathcal{M}\models\phi(a_1,\hdots,a_n)$ if and only if 
\[
(t^\mathcal{M}_1(a_1,\hdots,a_n),\hdots,t^\mathcal{M}_{n_R}(a_1,\hdots,a_n))\in R^\mathcal{M}.
\]
\item If $\phi$ is $\lnot\psi(v_1,\hdots,v_n)$, then $\mathcal{M}\models\phi(a_1,\hdots,a_n)$ if and only if
\[
\mathcal{M}\not\models\psi(a_1,\hdots,a_n).
\] 
\item If $\phi$ is $\psi(v_1,\hdots,v_n)\land\theta(v_1,\hdots,v_n)$, then $\mathcal{M}\models\phi(a_1,\hdots,a_n)$ if and only if
\[
\mathcal{M}\models\psi(a_1,\hdots,a_n) \,\text{and}\, \mathcal{M}\models\theta(a_1,\hdots,a_n).
\] 
\item If $\phi$ is $\psi(v_1,\hdots,v_n)\lor\theta(v_1,\hdots,v_n)$, then $\mathcal{M}\models\phi(a_1,\hdots,a_n)$ if and only if
\[
\mathcal{M}\models\psi(a_1,\hdots,a_n) \,\text{or}\, \mathcal{M}\models\theta(a_1,\hdots,a_n).
\]
\item If $\phi$ is $\exists v\,\psi(v_1,\hdots,v_n,v)$, then $\mathcal{M}\models\phi(a_1,\hdots,a_n)$ if and only if there exists $b\in M$ such that 
\[
\mathcal{M}\models\psi(a_1,\hdots,a_n,b).
\]
\item If $\phi$ is $\forall v\,\psi(v_1,\hdots,v_n,v)$, then $\mathcal{M}\models\phi(a_1,\hdots,a_n)$ if and only if for all $b\in M$, 
\[
\mathcal{M}\models\psi(a_1,\hdots,a_n,b).
\]
\end{itemize}
\end{defin}

This definition may seem pedantic, but it is another key separation between syntax and semantics, and it clearly demonstrates the inductive structure of $\mathcal{L}$-formulas.

\subsubsection*{Sentences and Theories}
For each prime $p$, consider the $\mathcal{L}_{VF}$-formula
\[
Char_p: \underbrace{1 + 1 + \hdots + 1}_{p \,times} = 0,
\]
which expresses the property that an $\mathcal{L}_{VF}$-structure has characteristic $p$. Once again taking valued fields as $\mathcal{L}_{VF}$-structures in the natural way, we have $\mathbb{F}_5((t))\models Char_5$, but $\mathbb{Q}_5\not\models Char_5$. 

\begin{defin}\label{def:sentence}
An \emph{$\mathcal{L}$-sentence} is an $\mathcal{L}$-formula which has no free variables.
\end{defin}

Note that since $Char_p$ has no free variables, we are able to state whether a structure $\mathcal{M}$ satisfies $Char_p$ without choosing any elements from the domain of $\mathcal{M}$ to substitute. Sentences express properties of structures, not of individual elements.

Suppose that we want to express the property that a structure has characteristic zero as an $\mathcal{L}_{VF}$-sentence. That is, we want to say that a structure does not have characteristic $p$ for any prime $p$. We can write a sentence which expresses the property that a structure does not have characteristic $p$ for some finite number of primes $p_1,\hdots,p_n: \lnot Char_{p_1} \land \lnot Char_{p_2} \land \hdots \land \lnot Char_{p_n}$. But sentences have finite length by definition, so this approach will not work for all $p$. It turns out that the only way to express the property characteristic zero in $\mathcal{L}_{VF}$ is with an infinite set of sentences.

\begin{defin}\label{def:theory}
An \emph{$\mathcal{L}$-theory} is a set of $\mathcal{L}$-sentences. For $\mathcal{M}$ an $\mathcal{L}$-structure, and $T$ an $\mathcal{L}$-theory, we say that $\mathcal{M}$ is a \emph{model} of $T$, written $\mathcal{M}\models T$, if $\mathcal{M}\models\phi$ for all sentences $\phi\in T$.
\end{defin}

In order to express the property that a structure has characteristic zero, we can define an $\mathcal{L}_{VF}$-theory, $Char_0 = \{\lnot Char_p \,|\, p \,\mbox{prime}\}$. Then $\mathbb{Q}_5 \models Char_0$, since $\mathbb{Q}_5\models \lnot Char_p$ for all primesm $p$, but $\mathbb{F}_5((t))\not\models Char_0$, since $\mathbb{F}_5((t))\not\models \lnot Char_5$.

Of course, every valued field with characteristic zero is a model for the theory $Char_0$, but $Char_0$ has other models which are not even fields. If we add the field axioms expressed as $\mathcal{L}_{VF}$ sentences to the theory, then the models for this theory will be exactly the class of $\mathcal{L}_{VF}$-structures which are fields with characteristic zero. 

\begin{defin}\label{def:el_class}
A class of $\mathcal{L}$-structures, $\mathcal{K}$, is called \emph{elementary} if there exists an $\mathcal{L}$-theory $T$ such that $\mathcal{K}$ contains exactly those $\mathcal{L}$-structures which are models for $T$. The structure $T$ is called a set of \emph{axioms} for $\mathcal{K}$.
\end{defin}

\begin{exa}\label{exa:VFel}
We will show that the class of valued fields with cross section is elementary by providing a set of axioms in $\mathcal{L}_{VF}$. We will call this theory $VF$. As an exercise, make sure you understand what property each of the following $\mathcal{L}_{VF}$-sentences expresses.
\begin{itemize}
\item Field axioms:
\begin{enumerate}
\item $\forall v_1\forall v_2\forall v_3\,(v_1+v_2)+v_3=v_1+(v_2+v_3)$
\item $\forall v_1\,v_1+0 = v_1$
\item $\forall v_1\,v_1+-(v_1)=0$
\item $\forall v_1\forall v_2\,v_1+v_2=v_2+v_1$
\item $\forall v_1\forall v_2\forall v_3\,(v_1\cdot v_2)\cdot v_3=v_1\cdot (v_2\cdot v_3)$
\item $\forall v_1\,v_1\cdot 1=v_1$
\item $\forall v_1 \,\lnot(v_1=0)\rightarrow(\exists v_2\,v_1\cdot v_2=1)$
\item $\forall v_1\forall v_2\,v_1\cdot v_2=v_2\cdot v_1$
\item $\forall v_1\forall v_2\forall v_3\, v_1\cdot (v_2 + v_3) = v_1\cdot v_2 + v_1\cdot v_3$
\item $\lnot(1=0)$
\end{enumerate}
\item Valuation axioms:
\begin{enumerate}
\item $\forall v_1\, \mathfrak{v}(v_1)=0\leftrightarrow v_1 = 0$
\item $\forall v_1\forall v_2\, \mathfrak{v}(v_1\cdot v_2) = \mathfrak{v}(v_1)\cdot \mathfrak{v}(v_2)$
\item $\forall v_1\forall v_2\, \mathfrak{v}(v_1)\leq \mathfrak{v}(v_2)\rightarrow \mathfrak{v}(v_1)\leq \mathfrak{v}(v_1 + v_2)$
\end{enumerate}
\item The value group is a subgroup of the multiplicative group:
\begin{enumerate}
\item $\forall v_1\forall v_2\, (V(v_1)\land V(v_2))\rightarrow V(v_1\cdot v_2)$
\item $\forall v_1\, V(v_1)\rightarrow (\exists v_2\, V(v_2)\land v_1\cdot v_2 = 1)$
\end{enumerate}
\item Linear order axioms for the value group:
\begin{enumerate}
\item $\forall v_1 \forall v_2\, (V(v_1)\land V(v_2))\rightarrow (v_1\leq v_2 \lor v_2\leq v_1)$
\item $\forall v_1 \forall v_2\, (V(v_1)\land V(v_2)\land (v_1\leq v_2) \land (v_2\leq v_1))\rightarrow (v_1 = v_2)$
\item $\forall v_1\forall v_2\forall v_3\, (V(v_1) \land V(v_2) \land V(v_3)\land (v_1\leq v_2) \land (v_2\leq v_3))\rightarrow (v_1\leq v_3)$
\item $\forall v_1\forall v_2\forall v_3\, (V(v_1) \land V(v_2) \land V(v_3)\land v_1\leq v_2)\rightarrow (v_1\cdot v_3\leq v_1\cdot v_3)$
\end{enumerate}
\item Cross section axioms:
\begin {enumerate}
\item $\forall v_1\, V(\mathfrak{v}(v_1))$
\item $\forall v_1\, V(v_1)\rightarrow (\mathfrak{v}(v_1)=v_1)$
\end{enumerate}
\end{itemize}
\end{exa}

Of course, there are many other $\mathcal{L}_{VF}$-sentences which are true of all valued fields but are not included in the axioms. We call these sentences logical consequences of the theory, and denote this relationship with the already overloaded symbol $\models$.

\begin{defin}\label{def:consequence}
Let $T$ be an $\mathcal{L}$-theory and $\phi$ an $\mathcal{L}$-sentence. We call $\phi$ a \emph{logical consequence} of $T$ and write $T\models\phi$ if $\mathcal{M}\models\phi$ for all models $\mathcal{M}\models T$. An $\mathcal{L}$-theory $T$ is called \emph{complete} if for all $\mathcal{L}$-sentences $\phi$, $T\models\phi$ or $T\models\lnot\phi$. 
\end{defin}

Note that in the definition of $T\models\phi$, we do not claim that one can provide a proof of $\phi$ given the assumptions in $T$, merely that in any structure in which the sentences in $T$ hold, $\phi$ also holds. One of G\"odel's remarkable results was proving that these concepts, provability and model theoretic consequence, are actually equivalent.

The mathematical study of proof and proof systems belongs to another branch of logic, and we will not make the notion of proof completely rigorous in this thesis. But we will note that in this context, a proof means a finite sequence of $\mathcal{L}$-sentences, some of which are introduced as assumptions and some of which follow from previous sentences by rules of inference. If a proof of an $\mathcal{L}$-sentence exists involving only assumptions from an $\mathcal{L}$-theory $T$, we write $T\vdash \phi$.

\begin{thm}[G\"odel's Completeness Theorem]\label{thm:GCT}
Let $T$ be an $\mathcal{L}$-theory and $\phi$ an $\mathcal{L}$-sentence. Then $T\models\phi$ if and only if $T\vdash\phi$. 
\end{thm}

One consequence of the Completeness Theorem is that any theory which does not imply a contradiction has a model.

\begin{defin}\label{def:consistent}
An $\mathcal{L}$-theory $T$ is called \emph{inconsistent} if there is an $\mathcal{L}$-sentence $\phi$ such that $T\vdash (\phi\land\lnot\phi)$. Otherwise, $T$ is called \emph{consistent}. 
\end{defin}

\begin{defin}\label{def:sat}
An $\mathcal{L}$-theory $T$ is called \emph{satisfiable} if it has a model.
\end{defin}

\begin{cor}[{\cite[Corollary 2.1.3]{Marker}}]\label{cor:consistent}
An $\mathcal{L}$-theory $T$ is satisfiable if and only if is consistent.
\end{cor}
\begin{proof}
Suppose $T$ is satisfiable. Then there is a model $\mathcal{M}\models T$. If $T\models(\phi\land\lnot\phi)$ for some $\mathcal{L}$-sentence $\phi$, then $\mathcal{M}\models(\phi\land\lnot\phi)$, which is impossible, since by definition, we would have $\mathcal{M}\models\phi$ and $\mathcal{M}\not\models\phi$, a contradiction. Thus $T$ is consistent.

Now suppose $T$ is not satisfiable. For any $\mathcal{L}$-sentence $\phi$, $(\phi\land\lnot\phi)$ is true in every model of $T$ trivially, since $T$ has no models. Thus $T\models(\phi\land\lnot\phi)$, and by the Completeness Theorem, $T\vdash (\phi\land\lnot\phi)$. So $T$ is inconsistent.
\end{proof}

Another easy consequence of the Completeness Theorem is the Compactness Theorem, a powerful result which is central to Model Theory.

\begin{thm}[Compactness Theorem {\cite[Theorem 2.1.4]{Marker}}]\label{thm:compactness}
An $\mathcal{L}$-theory $T$ is satisfiable if and only if every finite subset of $T$ is satisfiable. 
\end{thm}
\begin{proof}
One direction is obvious. Any model of $T$ is also a model of any subset of $T$, so if $T$ is satisfiable, then every finite subset of $T$ is satisfiable.

Conversely, suppose that every finite subset of $T$ is satisfiable. Assume for the sake of contradiction that $T$ is not satisfiable. Then by Corollary~\ref{cor:consistent}, there is some $\mathcal{L}$-sentence $\phi$ such that $T\vdash(\phi\land\lnot\phi)$. Now since proofs are finite in length, the proof of $\phi\land\lnot\phi$ can use as assumptions only finitely many elements of $T$. Call this finite set $\Delta$.

Then $\Delta\vdash(\phi\land\lnot\phi)$, and by Corollary~\ref{cor:consistent}, $\Delta$ is not satisfiable. But this is a contradiction, and hence $T$ is satisfiable.
\end{proof}

There are several other proofs of the Completeness Theorem which do not rely on the Completeness Theorem. These other methods are in a sense more constructive, and they can give us more information about the satisfying model, including us an upper bound on its cardinality. See Marker \cite{Marker}.

\begin{thm}[Compactness Theorem with cardinality {\cite[Theorem 2.1.11]{Marker}}]\label{thm:cardinal_compactness}
Let $T$ be an $\mathcal{L}$-theory such that every finite subset of $T$ is satisfiable. Then there is a model of $T$ of cardinality $|\mathcal{L}|$.
\end{thm}

\begin{exa}\label{exa:compactness}
Compactness is a powerful tool for constructing models with desired properties. As an example, consider the language $\mathcal{L} = \{\cdot,+,<,0,1\}$ and the $\mathcal{L}$-structure $\mathbb{N}$, with the $\mathcal{L}$-symbols interpreted in the usual way. Let $Th(\mathbb{N})$ be the full $\mathcal{L}$-theory of $\mathbb{N}$, that is, the set of all $\mathcal{L}$-sentences which are true in $\mathbb{N}$.

Now we will extend the language by adding a new constant symbol, $c$. Let $\mathcal{L}' = \mathcal{L}\cup\{c\}$. We will also extend the theory by adding new sentences expressing that $c$ is larger than every natural number. For each $n\in\mathbb{N}$, let $\phi_n$ be the $\mathcal{L}'$-sentence
\[
\underbrace{1 + 1 + \hdots + 1}_{n\, \text{times}}<c.
\]
Let $T' = Th(\mathbb{N}) \cup \{\phi_n\,|\,n = 1,2,\hdots\}$. We can consider the $\mathcal{L}$-sentences in $Th(\mathbb{N})$ as $\mathcal{L}'$-sentences, since $\mathcal{L}\subset\mathcal{L}'$, so $T'$ is an $\mathcal{L}'$-theory.

We will use Compactness to show that $T'$ has a model. Let $\Delta$ be a finite subset of $T'$. We claim that $\mathbb{N}\models\Delta$ under an appropriate interpretation of $c$. Since $\Delta$ is finite, it consists of finitely many sentences of $Th(\mathbb{N})$ and finitely many $\phi_n$. Let $M$ be the greatest integer such that $\phi_M\in \Delta$ (or $0$ if $\Delta$ contains no $\phi_n$). Then consider $\mathbb{N}$ as an $\mathcal{L}'$-structure, with the interpretation $c^\mathbb{N} = M+1$. For each $\phi_n\in T'$, $\mathbb{N}\models \phi_n$, since $n < M+1$. For each other $\psi\in T'$, $\psi\in Th(\mathbb{N})$, so $\mathbb{N}\models\psi$ by definition.

Thus with this interpretation of $c$, $\mathbb{N}\models\Delta$. Hence every finite subset of $T'$ is satisfiable, and by Compactness, $T'$ has a model.

This means that there is an $\mathcal{L}'$-structure, $\mathcal{N}$, such that every $\mathcal{L}$-sentence which is true in $\mathbb{N}$ is true in $\mathcal{N}$. That is, $\mathbb{N}$ and $\mathcal{N}$ cannot be distinguished using any first-order $\mathcal{L}$-sentence. But the interpretation of $c$ in $\mathcal{N}$ is greater than every natural number, so $\mathcal{N}$ has ``infinite'' elements. 
\end{exa}

Model theory is filled with counterintuitive results like these, and much of the theory is devoted to exploring the properties of unusual models for familiar theories. In fact, part of the proof of the Ax-Kochen Theorem requires the use of very large models for the theory of valued fields (see Section~\ref{subsec:TaSM}). 

\subsubsection*{Homomorphisms and Elementary Maps}
As usual when defining new mathematical objects, we will define the maps between them which preserve structure.

\begin{defin}\label{def:hom}
Given two $\mathcal{L}$-structures $\mathcal{M}$ and $\mathcal{N}$ with domains $M$ and $N$, an \emph{$\mathcal{L}$-homomorphism} from $\mathcal{M}$ to $\mathcal{N}$ is a map $\eta: M\rightarrow N$ which preserves interpretation of $\mathcal{L}$-symbols. That is,
\begin{itemize}
\item if $c$ is a constant symbol, then $\eta(c^\mathcal{M}) = c^\mathcal{N}$,
\item if $f$ is a function symbol, then for all $(a_1,\hdots,a_{n_f})\in M^{n_f}$, $\eta(f^\mathcal{M}(a_1,\hdots,a_{n_f})) = f^\mathcal{N}(\eta(a_1),\hdots,\eta(a_{n_f}))$, and
\item if $R$ is a relation symbol, then for all $(a_1,\hdots,a_{n_f})\in M^{n_R}$, $(a_1,\hdots,a_{n_R})\in R^\mathcal{M}$ if and only if $(\eta(a_1),\hdots,\eta(a_{n_R}))\in R^\mathcal{N}$.
\end{itemize}
An \emph{$\mathcal{L}$-isomorphism} is a bijective $\mathcal{L}$-homomorphism. If there is an $\mathcal{L}$-isomorphism from $\mathcal{M}$ to $\mathcal{N}$, we write $\mathcal{M}\cong\mathcal{N}$.
\end{defin}

This definition of homomorphism generalizes the notion of homomorphism in many settings. For instance, if we interpret groups as $\mathcal{L}_G$-structures as in Example~\ref{exa:L_G}, then all group homomorphisms are $\mathcal{L}_G$-homomorphisms. Similarly, homomorphisms of valued fields, which must preserve the field structure, the valuation, and the ordering on the value group, are $\mathcal{L}_{VF}$-homomorphisms. 

The existence of $\mathcal{L}$-homomorphisms between $\mathcal{L}$-structures does not give us much information about which $\mathcal{L}$-formulas are satisfied in these structures. A stronger notion is that of an elementary homomorphism, a map which preserves not just the interpretation of the language, but also the satisfaction of formulas.

\begin{defin}\label{def:el_emb}
An \emph{elementary} $\mathcal{L}$-homomorphism is an $\mathcal{L}$-homomorphism $j:\mathcal{M}\rightarrow\mathcal{N}$ between $\mathcal{L}$-structures $\mathcal{M}$ and $\mathcal{N}$ such that for all $\mathcal{L}$-formulas $\phi(v_1,\hdots,v_n)$ and elements $a_1,\hdots,a_n$ in the domain of $\mathcal{M}$, $\mathcal{M}\models\phi(a_1,\hdots,a_n)$ if and only if $\mathcal{N}\models\phi(j(a_1),\hdots,j(a_n))$.
\end{defin}

\begin{defin}\label{def:substructure}
An $\mathcal{L}$-structure $\mathcal{M}$ with domain $M$ is a \emph{substructure} of an $\mathcal{L}$-structure $\mathcal{N}$ with domain $N$ if $M\subseteq N$ and the inclusion map is an $\mathcal{L}$-homomorphism.

If $\mathcal{M}$ is a substructure of $\mathcal{N}$ and the inclusion map is elementary, then $\mathcal{M}$ is an \emph{elementary substructure} of $\mathcal{N}$ and $\mathcal{N}$ is an \emph{elementary extension} of $\mathcal{M}$.
\end{defin}

One of the most surprising of the foundational theorems of model theory relates to the existence of elementary extensions and substructures. The L\"owenheim-Skolem Theorem, given here without proof, intuitively states that given an infinite structure, there are elementary extensions and elementary substructures of all infinite cardinalities.

\begin{thm}[L\"owenheim-Skolem Theorem Up {\cite[Theorem 2.3.4]{Marker}}]\label{thm:L-SUp}
Let $\mathcal{M}$ be an infinite $\mathcal{L}$-structure with domain $M$, and let $\kappa$ be an infinite cardinal such that $\kappa\geq|M|+|\mathcal{L}|$. Then there is an $\mathcal{L}$-structure $\mathcal{N}$ of cardinality $\kappa$ and an elementary embedding $j:\mathcal{M}\rightarrow\mathcal{N}$, so that $\mathcal{N}$ is an elementary extension of $j(\mathcal{M})$. 
\end{thm}

\begin{thm}[L\"owenheim-Skolem Theorem Down {\cite[Theorem 2.3.7]{Marker}}]\label{thm:L-S}
Let $\mathcal{M}$ be an $\mathcal{L}$-structure with domain $M$, and let $X\subseteq M$. Then there is an elementary substructure $\mathcal{N}$ of $\mathcal{M}$ with domain $N$ such that $X\subseteq N$ and $|N|\leq |X| + |\mathcal{L}| + \aleph_0$. 
\end{thm}

\begin{rmk}\label{rmk:el_eq}
It is immediate from the definition of elementary $\mathcal{L}$-homomorphism that if there is an elementary $\mathcal{L}$-homomorphism $j:\mathcal{M}\rightarrow\mathcal{N}$, then $\mathcal{M}$ and $\mathcal{N}$ satisfy exactly the same $\mathcal{L}$-sentences, since for any $\mathcal{L}$-sentence $\phi$, $\mathcal{M}\models\phi$ if and only if $\mathcal{N}\models\phi$. Such structures are called elementarily equivalent.
\end{rmk}

Given an $\mathcal{L}$-structure $\mathcal{M}$, we define the full $\mathcal{L}$-theory of $\mathcal{M}$, $Th(\mathcal{M}) = \{\phi \,|\, \mathcal{M}\models\phi\}$. By definition, if $\phi$ is an $\mathcal{L}$-sentence, $\mathcal{M}\models\phi$ or $\mathcal{M}\models\lnot\phi$, so $Th(\mathcal{M})\models\phi$ or $Th(\mathcal{M})\models\lnot\phi$, and thus $Th(\mathcal{M})$ is complete.

\begin{defin}\label{def:el_eq}
Let $\mathcal{M}$ and $\mathcal{N}$ be $\mathcal{L}$-structures. We say that $\mathcal{M}$ and $\mathcal{N}$ are \emph{elementarily equivalent}, written $\mathcal{M}\equiv\mathcal{N}$, if $Th(\mathcal{M}) = Th(\mathcal{N})$, that is, for any $\mathcal{L}$-sentence $\phi$, $\mathcal{M}\models\phi$ if and only if $\mathcal{N}\models\phi$.
\end{defin}

The converse to Remark~\ref{rmk:el_eq} does not hold in general. That is, it is possible to have $\mathcal{L}$-structures $\mathcal{M}$ and $\mathcal{N}$ such that $\mathcal{M}\equiv\mathcal{N}$, but there is no elementary $\mathcal{L}$-homomorphism from $\mathcal{M}$ to $\mathcal{N}$. The statement that $\mathcal{M}$ and $\mathcal{N}$ are elementarily equivalent only requires that they satisfy the same \emph{$\mathcal{L}$-sentences}, but if there is an elementary $\mathcal{L}$-homomorphism from $\mathcal{M}$ to $\mathcal{N}$, then this homomorphism must respect the satisfaction of all \emph{$\mathcal{L}$-formulas}. By choosing a suitably extended language, we can turn these formulas into sentences.

Let $\mathcal{M}$ be an $\mathcal{L}$-structure with domain $M$. We will extend the language $\mathcal{L}$ by adding a new constant symbol $c_m$ for every element $m\in M$. Let $\mathcal{L}_\mathcal{M} = \mathcal{L}\cup\{c_m\,|\,m\in M\}$. Then $\mathcal{M}$ can be viewed as an $\mathcal{L}_\mathcal{M}$-structure, where we interpret each constant in the natural way, $c_m^\mathcal{M} = m$.

\begin{defin}\label{def:diag}
The \emph{elementary diagram} of an $\mathcal{L}$-structure $\mathcal{M}$ with domain $M$, denoted $Diag_{el}(\mathcal{M})$, is the $\mathcal{L}_\mathcal{M}$-theory which captures all the information about satisfaction of $\mathcal{L}$-formulas in $\mathcal{M}$.
\[
Diag_{el}(\mathcal{M}) = \{\phi(c_{m_1},\hdots,c_{m_n})\,|\,m_1,\hdots,m_n\in M\,\text{and}\,\mathcal{M}\models\phi(m_1,\hdots,m_n)\}.
\]
\end{defin}

\begin{lem}\label{lem:diag}
Let $\mathcal{M}$ be an $\mathcal{L}$-structure with domain $M$. Suppose $\mathcal{N}$ is an $\mathcal{L}_\mathcal{M}$-structure such that $\mathcal{N}\models Diag_{el}(\mathcal{M})$. Then viewing $\mathcal{N}$ as an $\mathcal{L}$-structure (by ``forgetting'' the interpretations of the symbols $c_m$), there is an elementary embedding of $\mathcal{M}$ into $\mathcal{N}$.
\end{lem}
\begin{proof}
Define $j:\mathcal{M}\rightarrow\mathcal{N}$ by $j(m) = c_m^\mathcal{N}$, the interpretation of the corresponding constant symbol in $\mathcal{N}$. Suppose $m_1,m_2\in M$ with $m_1\neq m_2$. Then $\lnot (c_{m_1} = c_{m_2})$ is in $Diag_{el}(\mathcal{M})$, so $\mathcal{N}\models \lnot (c_{m_1} = c_{m_2})$, and in $\mathcal{N}$, $c_{m_1}^\mathcal{N}\neq c_{m_2}^\mathcal{N}$, so $j(m_1)\neq j(m_2)$. Thus $j$ is injective.

If $\mathcal{M}\models\phi(m_1,\hdots,m_n)$ for an $\mathcal{L}$-formula $\phi$, then $\phi(c_{m_1},\hdots,c_{m_n})\in Diag_{el}(\mathcal{M})$. Since $\mathcal{N}\models Diag_{el}(\mathcal{M})$, $\mathcal{N}\models \phi(c_{m_1},\hdots,c_{m_n})$, and thus $\mathcal{N}\models \phi(j(m_1),\hdots,j(m_n))$, so $j$ is an elementary $\mathcal{L}$-homomorphism.
\end{proof}

As one would expect, $\mathcal{L}$-isomorphic structures are elementarily equivalent. We will conclude our whirlwind tour of the basics of model theory with a proof of this fact, which will also serve as a first example of the technique of induction on terms and formulas. 

The idea is that all terms and formulas are built from atomic elements (terms from variables and constants, formulas from atomic formulas) in a finite number of steps. To prove a claim, we show that it is true for all atomic elements. Then we show that if we construct a new term (or formula) from a set of terms (or formulas) for which our claim is true, then our claim is also true on the new term (or formula). This shows that the claim is true for all terms (or formulas). 

\begin{thm}[{\cite[Theorem 1.1.10]{Marker}}]\label{thm:isomorphism}
For $\mathcal{L}$-structures $\mathcal{M}$ and $\mathcal{N}$, if $\mathcal{M}\cong\mathcal{N}$, then $\mathcal{M}\equiv\mathcal{N}$. 
\end{thm} 
\begin{proof}
Let $j: \mathcal{M}\rightarrow\mathcal{N}$ be an $\mathcal{L}$-isomorphism mapping $M$, the domain of $\mathcal{M}$, bijectively to $N$, the domain of $\mathcal{N}$. If $\overline{a} = (a_1,\hdots,a_n)\in M^n$, let $j(\overline{a}) = (j(a_1),\hdots,j(a_n))\in N^n$. 

We will prove by induction on terms the following claim: if $t$ is an $\mathcal{L}$-term with free variables from $\overline{v} = (v_1,\hdots,v_n)$, then for all $\overline{a} = (a_1,\hdots,a_n)\in M^n$, $j(t^\mathcal{M}(\overline{a})) = t^\mathcal{N}(j(\overline{a}))$. 

If $t$ is a constant symbol $c$, then $j(t^\mathcal{M}(\overline{a})) = j(c^\mathcal{M}) = c^\mathcal{N} = t^\mathcal{N}(j(\overline{a}))$.

If $t$ is a variable $v_i$, then $j(t^\mathcal{M}(\overline{a})) = j(a_i) = t^\mathcal{N}(j(\overline{a}))$.

If $t$ is $f(t_1(\overline{v}),\hdots,t_{n_f}(\overline{v}))$, where $f$ is a function symbol and $t_1,\hdots,t_{n_f}$ are $\mathcal{L}$-terms for which the claim is true, then
\begin{eqnarray*}
j(t^\mathcal{M}(\overline{a})) &=& j(f^\mathcal{M}(t_1^\mathcal{M}(\overline{a}),\hdots,t_{n_f}^\mathcal{M}(\overline{a}))) \\
&=& f^\mathcal{N}(j(t_1^\mathcal{M}(\overline{a})),\hdots,j(t_{n_f}^\mathcal{M}(\overline{a}))) \,\,\mbox{since $j$ is an $\mathcal{L}$-homomorphism}\\
&=& f^\mathcal{N}(t_1^\mathcal{N}(j(\overline{a})),\hdots,t_{n_f}^\mathcal{N}(j(\overline{a}))) \,\,\mbox{by induction}\\
&=& t^\mathcal{N}(j(\overline{a})).
\end{eqnarray*}

This completes the induction on terms and the proof of the claim. 

Next we will prove by induction on formulas that if $\phi$ is an $\mathcal{L}$-formula with free variables from $\overline{v} = (v_1,\hdots,v_n)$, then for all $\overline{a} = (a_1,\hdots,a_n)\in M^n$, $\mathcal{M}\models\phi(\overline{a})$ if and only if $\mathcal{N}\models\phi(j(\overline{a}))$. 

If $\phi(\overline{v})$ is $t_1(\overline{v}) = t_2(\overline{v})$, where $t_1$ and $t_2$ are  $\mathcal{L}$-terms, then
\begin{eqnarray*}
\mathcal{M}\models\phi(\overline{a}) &\mbox{iff}& t_1^\mathcal{M}(\overline{a}) = t_2^\mathcal{M}(\overline{a}) \\
&\mbox{iff}& j(t_1^\mathcal{M}(\overline{a})) = j(t_2^\mathcal{M}(\overline{a})) \,\,\mbox{because $j$ is injective}\\
&\mbox{iff}& t_1^\mathcal{N}(j(\overline{a})) = t_2^\mathcal{N}(j(\overline{a})) \,\,\mbox{applying the claim}\\
&\mbox{iff}& \mathcal{N}\models \phi(j(\overline{a})).
\end{eqnarray*}

If $\phi(\overline{v})$ is $R(t_1(\overline{v}),\hdots,t_{n_R}(\overline{v}))$, where $R$ is a relation symbol and $t_1,\hdots,t_{n_R}$ are $\mathcal{L}$-terms, then

\begin{eqnarray*}
\mathcal{M}\models\phi(\overline{a}) &\mbox{iff}& (t_1^\mathcal{M}(\overline{a}),\hdots,t_{n_R}^\mathcal{M}(\overline{a}))\in R^\mathcal{M}\\
&\mbox{iff}& (j(t_1^\mathcal{M}(\overline{a})),\hdots,j(t_{n_R}^\mathcal{M}(\overline{a})))\in R^\mathcal{N} \,\,\mbox{because $j$ is a homomorphism}\\
&\mbox{iff}& (t_1^\mathcal{N}(j(\overline{a})),\hdots,t_{n_R}^\mathcal{N}(j(\overline{a})))\in R^\mathcal{N}\,\,\mbox{applying the claim}\\
&\mbox{iff}& \mathcal{N}\models\phi(j(\overline{a})).
\end{eqnarray*}

If $\phi(\overline{v})$ is $\lnot\psi(\overline{v})$, where $\psi$ is an $\mathcal{L}$-formula for which our assertion is true, then 
\begin{eqnarray*}
\mathcal{M}\models\phi(\overline{a}) &\mbox{iff}& \mathcal{M}\not\models\psi(\overline{a})\\
&\mbox{iff}& \mathcal{N}\not\models\psi(j(\overline{a}))\,\,\mbox{by induction}\\
&\mbox{iff}& \mathcal{N}\models\phi(j(\overline{a})).
\end{eqnarray*}

If $\phi(\overline{v})$ is $\psi(\overline{v})\land\theta(\overline{v})$, where $\psi$ and $\theta$ are $\mathcal{L}$-formulas for which our assertion is true, then
\begin{eqnarray*}
\mathcal{M}\models\phi(\overline{a}) &\mbox{iff}& \mathcal{M}\models\psi(\overline{v}) \,\mbox{and}\,\mathcal{M}\models\theta(\overline{v})\\
&\mbox{iff}& \mathcal{N}\models\psi(j(\overline{v}))\,\mbox{and}\,\mathcal{N}\models\theta(j(\overline{v}))\,\,\mbox{by induction}\\
&\mbox{iff}& \mathcal{N}\models\phi(j(\overline{v})).
\end{eqnarray*}

If $\phi(\overline{v})$ is $\exists w\,\psi(\overline{v},w)$, where $w$ is a variable and $\psi$ is an $\mathcal{L}$-formula for which our assertion is true, then $\mathcal{M}\models\phi(\overline{a})$ if and only if there exists some $b\in M$ such that $\mathcal{M}\models\psi(\overline{a},b)$. Now if there exists such a $b$, then by induction, $\mathcal{N}\models\psi(j(\overline{a}),j(b))$, so there exists $c\in N$ (take $c = j(b)$) such that $\mathcal{N}\models\psi(j(\overline{a}),c)$, and thus $\mathcal{N}\models\phi(j(\overline{a}))$. Conversely, if $\mathcal{N}\models\phi(j(\overline{a}))$, then there exists $c\in N$ such that $\mathcal{N}\models\psi(j(\overline{a}),c)$. $j$ is surjective, so there exists $b\in M$ such that $j(b) = c$, and by induction $\mathcal{M}\models\psi(\overline{a},b)$. 

This completes the proof by induction on formulas. We do not need to consider formulas constructed using $\lor$ or $\forall$, since these can be re-written to use only $\lnot$, $\land$, and $\exists$. This also completes the proof of the theorem, since we have shown that if $\phi$ is an $\mathcal{L}$-sentence, then $\mathcal{M}\models\phi$ if and only if $\mathcal{N}\models\phi$. So $\mathcal{M}\equiv\mathcal{N}$.
\end{proof}

\subsection{Ultraproducts}\label{subsec:UP}

The fields $\mathbb{Q}_p$ and $\mathbb{F}_p((t))$ are not elementarily equivalent for any $p$ (for instance, the sentence $Char_p$
is true in $\mathbb{F}_p((t))$ but false in $\mathbb{Q}_p$). However, we will prove that the theory of $\mathcal{L}_{VF}$-sentences which are true in $\mathbb{Q}_p$ for all but finitely many $p$ is the same as the theory of $\mathcal{L}_{VF}$-sentences which are true in  $\mathbb{F}_p((t))$ for all but finitely many $p$.

A construction called the ultraproduct will allow us to build new structures from the $\mathbb{Q}_p$ and $\mathbb{F}_p((t))$ in which a sentence is true if and only if it is true in ``almost all'' of these fields. 
The precise meaning of ``almost all'' is described by the definition of a filter on a set.

\begin{defin}\label{def:filter}
Given a set $I$, a \emph{filter} on $I$ is a subset of the power set $\mathcal{D}\subseteq\mathcal{P}(I)$ such that
\begin{itemize}
\item $\emptyset\notin\mathcal{D}$ and $I\in\mathcal{D}$,
\item if $A\in \mathcal{D}$ and $B\in\mathcal{D}$, then $A\cap B\in \mathcal{D}$, and
\item if $A\in \mathcal{D}$ and $A\subseteq B\subseteq I$, then $B\in \mathcal{D}$.
\end{itemize}
\end{defin}

\begin{exa}\label{exa:filters}
The following are examples of filters on a set $I$:
\begin{itemize}
\item $\mathcal{D}_T = \{I\}$. $\mathcal{D}_T$ is called the trivial filter.
\item For $j\in I$, $\mathcal{D}_j = \{X\subseteq I\,|\,j\in X\}$. $\mathcal{D}_j$ is called the principal filter generated by $j$.
\item $\mathcal{D}_F = \{X\subseteq I\,|\, I\setminus X \,\mbox{is finite}\}$. $\mathcal{D}_F$ is called the Frechet filter. Note that $\mathcal{D}_F$ is a filter only when $I$ is infinite, since otherwise $\emptyset\in \mathcal{D}_F$. 
\end{itemize}
\end{exa}

Some intuition for the notion of a filter can be built by thinking of the sets in the filter as those containing ``almost all'' elements of $I$. If two sets both contain almost all elements, their intersection should also contain almost all elements. If a set contains almost all elements, any superset should also contain almost all elements. Of course, for different filters, ``almost all'' has different meanings. A principal filter, for instance, gives great preference to its generating element.

\begin{defin}\label{def:ultrafilter}
A filter $\mathcal{D}$ on $I$ is an \emph{ultrafilter} if for all $X\subseteq I$, $X\in \mathcal{D}$ or $I\setminus X \in \mathcal{D}$.
\end{defin}

All principal filters are ultrafilters. The next lemma and theorem demonstrate that nonprincipal ultrafilters can be obtained by extending the Frechet filter on an infinite set.

\begin{lem}\label{lem:filter_extension}
Given a filter $\mathcal{D}$ on a set $I$ and a subset $X\subset I$ such that $X\notin\mathcal{D}$, we can extend $\mathcal{D}$ to a filter $\mathcal{D}_X$ on $I$ such that $\mathcal{D}\subseteq\mathcal{D}_X$ and $I\setminus X\in\mathcal{D}_X$.
\end{lem}
\begin{proof}
Let $\mathcal{D}_X = \{Y\subseteq I \,|\, \text{there exists}\,Z\in\mathcal{D}\,\text{such that}\, Z\setminus X\subseteq Y\}$.

$\mathcal{D}_X$ is a filter on $I$:
\begin{itemize}
\item If $\emptyset\in\mathcal{D}_X$, then there exists $Z\in\mathcal{D}$ such that $Z\setminus X = \emptyset$, that is, $Z\subseteq X$. But then $X\in \mathcal{D}$, which contradicts our assumption. So $\emptyset\notin\mathcal{D}_X$. Also, for any $Z \in \mathcal{D}$, $Z\setminus X\subseteq I$, so $I\in\mathcal{D}_X$. 
\item For $A,B\in\mathcal{D}_X$, there exist sets $Z_A, Z_B\in \mathcal{D}$ such that $Z_A\setminus X\subseteq A$ and $Z_B\setminus X\subseteq B$. Then $(Z_A\cap Z_B)\setminus X \subseteq (A\cap B)$, so $A\cap B\in \mathcal{D}_X$. 
\item For $A\in\mathcal{D}_X$ and $A\subseteq B\subseteq I$, there exists $Z_A\in\mathcal{D}$ such that $Z_A\setminus X \subseteq A \subseteq B$, so $B\in\mathcal{D}_X$.
\end{itemize}

For any $Y\in\mathcal{D}$, take $Z = Y$. $Y\setminus X\subseteq Y$, so $Y\in \mathcal{D}_X$. Thus $\mathcal{D}\subseteq \mathcal{D}_X$. 

To show that $I\setminus X\in\mathcal{D}_X$, take $Z = I$. $I\setminus X\subseteq I\setminus X$, so $I\setminus X\in\mathcal{D}_X$.
\end{proof}

\begin{thm}[{\cite[Proposition 4.1.3]{Chang}}]\label{thm:ultrafilter}
For any filter $\mathcal{D}$ on $I$, there exists an ultrafilter $\mathcal{U}$ on $I$ with $\mathcal{D}\subseteq \mathcal{U}$.
\end{thm}
\begin{proof}
Let $\mathcal{F}$ be the set of all filters on $I$ extending $\mathcal{D}$, ordered by the subset relation. If $(C_{\alpha}:\alpha<\beta)$ is a chain in $\mathcal{F}$, then $\mathcal{C} = \bigcup_{\alpha<\beta} \mathcal{C}_\alpha$ is a filter on $I$:
\begin{itemize}
\item We have $\emptyset\notin\mathcal{C}$ since $\emptyset\notin\mathcal{C}_\alpha$ for all $\alpha$, and $I\in \mathcal{C}$ since $I\in\mathcal{C}_\alpha$ for all $\alpha$. 
\item If $A,B\in\mathcal{C}$, then $A\in\mathcal{C}_\alpha$ and $B\in\mathcal{C}_\beta$ for some $\alpha$ and $\beta$. Then $A,B\in\mathcal{C}_{\max\{\alpha,\beta\}}$, and $A\cap B\in\mathcal{C}_{\max\{\alpha,\beta\}}\subseteq\mathcal{C}$.
\item If $A\in\mathcal{C}$ and $A\subseteq B\subseteq I$, then $A\in\mathcal{C}_\alpha$ for some $\alpha$, so $B\in\mathcal{C}_\alpha\subseteq\mathcal{C}$. 
\end{itemize}
$\mathcal{C}$ extends $\mathcal{D}$, since it is the union of a set of filters extending $\mathcal{D}$, so $\mathcal{C}\in\mathcal{F}$. Moreover, $\mathcal{C}$ is an upper bound for $(\mathcal{C}:\alpha<\beta)$, since $\mathcal{C}_\alpha\subseteq\mathcal{C}$ for all $\alpha$. Applying Zorn's lemma, $\mathcal{F}$ has maximal elements. Let $\mathcal{U}$ be a maximal element. We claim that $\mathcal{U}$ is an ultrafilter.

Let $X\subseteq I$ be a subset such that $X\notin\mathcal{U}$. By Lemma \ref{lem:filter_extension} we can find a filter $\mathcal{U}_X$ on $I$ such that $\mathcal{U}\subseteq\mathcal{U}_X$ and $I\setminus X\in\mathcal{U}_X$. Since $\mathcal{U}_X$ extends $\mathcal{U}$, it also extends $\mathcal{D}$. But $\mathcal{U}$ is a maximal element among the filters on $I$ extending $\mathcal{D}$, and thus $\mathcal{U}_X = \mathcal{U}$. We have shown that for any $X\subset I$, if $X\notin\mathcal{U}$, then $I\setminus X\in\mathcal{U}$, so $\mathcal{U}$ is an ultrafilter.
\end{proof}

Note that no ultrafilter $\mathcal{U}$ on $I$ extending the Frechet filter is principal, since for any $j\in I$, $I\setminus\{j\}\in \mathcal{D}_F\subset \mathcal{U}$. In fact, all nonprincipal ultrafilters are extensions of the Frechet filter.

\begin{lem}\label{lem:npuf}
The intersection $\bigcap\mathcal{D}$ of all nonprincipal ultrafilters $\mathcal{D}$ on an infinite set $I$ is the Frechet filter $\mathcal{D}_F$ on $I$. 
\end{lem}
\begin{proof}
Let $\mathcal{D}$ be a nonprincipal ultrafilter. For all $j\in I$, there exists $X_j\in \mathcal{D}$ such that $j\notin X_j$ (otherwise $\mathcal{D}$ would be principal generated by $j$). Then $Y_j = I\setminus\{j\}\in\mathcal{D}$, since $X_j\subseteq Y_j$. For all $A\in\mathcal{D}_F$, we have $A = I\setminus\{a_i\}_{i=1}^n$ for some finite set $\{a_i\}_{i=1}^n\subset I$. Then $A = \bigcap_{i=1}^n Y_{a_i} \in \mathcal{D}$. Thus $\mathcal{D}_F\subseteq \mathcal{D}$ for all nonprincipal ultrafilters $\mathcal{D}$, and $\mathcal{D}_F\subseteq\bigcap\mathcal{D}$.

Conversely, suppose $A\in\bigcap\mathcal{D}$, $A\notin \mathcal{D}_F$. Then by Lemma \ref{lem:filter_extension} we can extend $\mathcal{D}_F$ to a filter $\mathcal{D}_A$ containing $I\setminus A$. By Theorem \ref{thm:ultrafilter} we can extend $\mathcal{D}_A$ to an ultrafilter $\mathcal{U}$ containing $I\setminus A$. This filter is nonprincipal, since it extends $\mathcal{D}_F$, and $A\notin\mathcal{U}$. But then $A\notin\bigcap\mathcal{D}$, contradicting our assumption. Thus $\bigcap\mathcal{D} \subseteq \mathcal{D}_F$, and we have shown that $\bigcap\mathcal{D} = \mathcal{D}_F$.
\end{proof}

\subsubsection*{The Ultraproduct Construction}
Now that we have ultrafilters at our disposal, we are ready to introduce the ultraproduct construction.

Let $\{\mathcal{M}_i\}_{i\in I}$ be a collection of $\mathcal{L}$-structures indexed by an infinite set $I$, and let $\mathcal{D}$ be an ultrafilter on $I$. We will view the Cartesian product $\prod M_i$ of the domains of the $\mathcal{M}_i$ as the set of choice functions $\{f:I\rightarrow\bigcup M_i\,|\,\forall i\in I, f(i)\in M_i\}$. We define a relation $\sim_\mathcal{D}$ on $\prod M_i$ by $f\sim_\mathcal{D} g$ if and only if $\{i\in I\,|\,f(i)=g(i)\}\in \mathcal{D}$. 

\begin{prop}
The relation $\sim_\mathcal{D}$ is an equivalence relation.
\end{prop}
\begin{proof}

Let $f,g,h\in\prod M_i$.

\begin{itemize}
\item The set $\{i\in I\,|\, f(i)=f(i)\} = I\in\mathcal{D}$, so $f\sim_\mathcal{D} f$, and $\sim_\mathcal{D}$ is reflexive.
\item If $f\sim_\mathcal{D} g$, then $\{i\in I\,|\, g(i)=f(i)\} = \{i\in I\,|\, f(i)=g(i)\}\in\mathcal{D}$, so $g\sim_\mathcal{D} f$, and $\sim_\mathcal{D}$ is symmetric.
\item Suppose $f\sim_\mathcal{D} g$ and $g\sim_\mathcal{D} h$. Let $A = \{i\in I\,|\, f(i)=g(i)\}$, $B = \{i\in I\,|\, g(i)=h(i)\}$, and $C = \{i\in I\,|\, f(i)=h(i)\}$. $A\in\mathcal{D}$ and $B\in\mathcal{D}$, so $A\cap B\in\mathcal{D}$. $A\cap B\subseteq C$, so $C\in\mathcal{D}$. Thus $f\sim_\mathcal{D} h$, and $\sim_\mathcal{D}$ is transitive.
\end{itemize}

Hence $\sim_\mathcal{D}$ is an equivalence relation.
\end{proof}

\begin{defin}\label{def:ultraproduct}
With $\{\mathcal{M}_i\}_{i\in I}$, $\mathcal{D}$, and $\sim_\mathcal{D}$ as above, the \emph{ultraproduct} $\prod\mathcal{M}_i/\mathcal{D}$, is an $\mathcal{L}$-structure, defined as follows. Let $\mathcal{M} = \prod\mathcal{M}_i/\mathcal{D}$.
\begin{itemize}
\item The domain of $\mathcal{M}$, denoted $\prod M_i/\mathcal{D}$, is the set of equivalence classes of $\sim_\mathcal{D}$ in $\prod M_i$. We will denote the equivalence class of $f\in\prod M_i$ by $[f]$, or by $[f(i)|i\in I]$.
\item For each function symbol $f\in\mathcal{L}$, we define the interpretation $f^\mathcal{M}$ by 
\[
f^\mathcal{M}([g_1],\hdots,[g_{n_f}]) = [f^{\mathcal{M}_i}(g_1(i),\hdots,g_{n_f}(i))|i\in I].
\]
\item For each relation symbol $R\in\mathcal{L}$, we define the interpretation $R^\mathcal{M}$ by 
\[
([g_1],\hdots,[g_{n_R}])\in R^\mathcal{M}\,\text{if and only if}\, \{i\in I\,|\,(g_1(i),\hdots,g_{n_R}(i))\in R^{\mathcal{M}_i}\}\in\mathcal{D}.
\]
\item For each constant symbol $c\in\mathcal{L}$, we define the interpretation $c^\mathcal{M}$ by
\[
c^\mathcal{M} = [c^{\mathcal{M}_i}|i\in I].
\] 
\end{itemize}
\end{defin}

Using the ``almost all'' intuition for the ultrafilter $\mathcal{D}$, we can describe the elements of the domain of the ultraproduct as the classes of elements of the Cartesian product which are the same almost everywhere. A relation holds for elements of the ultraproduct if and only if the interpretation of the relation symbol holds for representatives of the element classes in almost all of the $\mathcal{M}_i$. 

Functions are applied in the ultraproduct by applying the interpretation of the function symbol to representatives of the element classes in each $\mathcal{M}_i$.  A constant in the ultraproduct is simply the equivalence class of the interpretation of the constant in each $\mathcal{M}_i$.

But there is something to check before we can accept this definition. In the definitions of the interpretations $f^\mathcal{M}$ and $R^\mathcal{M}$, we chose a representative $g_i$ for each equivalence class $[g_i]$. We must show that the interpretations are independent of our choices of representatives.

\begin{prop}
Let $\mathcal{M} = \prod\mathcal{M}_i/\mathcal{D}$, defined as above. For all function symbols $f$ and relation symbols $R$ in $\mathcal{L}$, the interpretations $f^\mathcal{M}$ and $R^\mathcal{M}$ are well-defined.
\end{prop}
\begin{proof}
Let $f$ be a function symbol. Suppose that for $1\leq j \leq n_f$ we have $g_j,h_j\in\prod M_i$ with $g_j \sim_\mathcal{D} h_j$. If we define $g_f(i) = f^{\mathcal{M}_i}(g_1(i),\hdots,g_{n_f}(i))$ and $h_f(i) = f^{\mathcal{M}_i}(h_1(i),\hdots,h_{n_f}(i))$, we would like to show that $g_f\sim_\mathcal{D} h_f$. For all $j$, let $A_j = \{i\in I\,|\, g_j(i)=h_j(i)\}$. $A_j\in\mathcal{D}$ for all $j$, so $\bigcap_{j=1}^n A_j\in \mathcal{D}$. Now $g_f$ and $h_f$ certainly agree whenever all of the $g_j$ and $h_j$ agree, so $\bigcap_{j=1}^n A_j\subseteq A_f = \{i\in I\,|\, g_f(i)=h_f(i)\}$, hence $A_f\in\mathcal{D}$, and $g_f\sim_\mathcal{D} h_f$.

Let $R$ be a relation symbol. Suppose that for $1\leq j \leq n_R$ we have $g_j,h_j\in\prod M_i$ with $g_j \sim_\mathcal{D} h_j$. If we define $G = \{i\in I\,|\,(g_1(i),\hdots,g_{n_R}(i))\in R^{\mathcal{M}_i}\}$ and $H = \{i\in I\,|\,(h_1(i),\hdots,h_{n_R}(i))\in R^{\mathcal{M}_i}\}$, we would like to show that $G\in \mathcal{D}$ if and only if $H\in\mathcal{D}$. Again, we define $A_j = \{i\in I\,|\, g_j(i)=h_j(i)\}\in\mathcal{D}$ for all $j$. Suppose $G\in\mathcal{D}$. Then $G\cap \bigcap_{j=1}^n A_j\in\mathcal{D}$. Now $(h_1(i),\hdots,h_{n_R}(i))$ is certainly in $R^{\mathcal{M}_i}$ whenever all of the $g_j$ and $h_j$ agree and $(g_1(i),\hdots,g_{n_R}(i))$ is in $R^{\mathcal{M}_i}$, so $G\cap \bigcap_{j=1}^n A_j \subseteq H$, and thus $H\in\mathcal{D}$. The converse follows by the same argument.
\end{proof}

The Fundamental Theorem of Ultraproducts states that the ultraproduct $\prod\mathcal{M}_i/\mathcal{D}$ of $\mathcal{L}$-structures satisfies the $\mathcal{L}$-formula $\phi$ if and only if almost all of the $\mathcal{M}_i$ satisfy $\phi$.

\begin{thm}[Fundamental Theorem of Ultraproducts, {\cite[Theorem 4.1.9]{Chang}}]\label{thm:FToUP}
Let $\mathcal{M}$ be the ultraproduct $\prod\mathcal{M}_i/\mathcal{D}$ of $\mathcal{L}$-structures with ultrafilter $\mathcal{D}$ on index set $I$. Let $\phi(\overline{v})$ be an $\mathcal{L}$-formula with free variables from $\overline{v} = (v_1,\hdots,v_n)$. Then for $\overline{[g]} = ([g_1],\hdots,[g_n])\in\left(\prod M_i/\mathcal{D}\right)^n$, $\prod\mathcal{M}_i/\mathcal{D}\models \phi(\overline{[g]})$ if and only if $\{i\in I\,|\,\mathcal{M}_i\models\phi(\overline{g(i)})\}\in\mathcal{D}$. 
\end{thm}
\begin{proof}
The proof is by induction on terms and formulas.

First, we claim that if $t(\overline{v})$ is an $\mathcal{L}$-term, then $t^\mathcal{M}(\overline{[g]}) = [t^{\mathcal{M}_i}(\overline{g(i)})\,|\,i\in I]$.

If $t$ is a constant symbol $c$, then the claim is true by definition: $c^\mathcal{M} = [c^{\mathcal{M}_i}\,|\,i\in I]$.

If $t$ is a variable $v_j$, then $t^\mathcal{M}(\overline{[g]}) = [g_j] = [g_j(i)\,|\,i\in I] = [t^{\mathcal{M}_i}(\overline{g(i)})\,|\,i\in I]$.

If $t$ is $f(t_1(\overline{v}),\hdots,t_{n_f}(\overline{v}))$, where $f$ is a function symbol and $t_1,\hdots,t_{n_f}$ are $\mathcal{L}$-terms for which the claim is true, then 
\begin{eqnarray*}
t^\mathcal{M}(\overline{[g]}) &=& f^\mathcal{M}(t_1^\mathcal{M}(\overline{[g]}),\hdots,t_{n_f}^\mathcal{M}(\overline{[g]}))\\
&=& f^\mathcal{M}([t_1^{\mathcal{M}_i}(\overline{g(i)})\,|\,i\in I],\hdots,[t_{n_f}^{\mathcal{M}_i}(\overline{g(i)})\,|\,i\in I]) \,\,\mbox{by induction}\\
&=& [f^{\mathcal{M}_i}(t_1^{\mathcal{M}_i}(\overline{g(i)}),\hdots,t_{n_f}^{\mathcal{M}_i}(\overline{g(i)}))\,|\,i\in I] \,\,\mbox{interpretation of $f$}\\
&=& [t^{\mathcal{M}_i}(\overline{g(i)})\,|\,i\in I].
\end{eqnarray*}

Having established that terms behave as expected under interpretation, we may move on to proving the theorem. We begin with atomic formulas. 

If $\phi(\overline{v})$ is $t_1(\overline{v}) = t_2(\overline{v})$, where $t_1$ and $t_2$ are $\mathcal{L}$-terms, then 
\begin{eqnarray*}
\mathcal{M}\models\phi(\overline{[g]}) &\mbox{iff}& t_1^\mathcal{M}(\overline{[g]}) = t_2^\mathcal{M}(\overline{[g]})\\
&\mbox{iff}& [t_1^{\mathcal{M}_i}(\overline{g(i)})\,|\,i\in I] = [t_2^{\mathcal{M}_i}(\overline{g(i)})\,|\,i\in I] \,\,\mbox{applying the claim}\\
&\mbox{iff}& \{i\in I\,|\,t_1^{\mathcal{M}_i}(\overline{g(i)}) = t_2^{\mathcal{M}_i}(\overline{g(i)})\}\in\mathcal{D}\\
&\mbox{iff}& \{i\in I\,|\, \mathcal{M}_i\models \phi(\overline{g(i)})\}\in\mathcal{D}.
\end{eqnarray*}

If $\phi(\overline{v})$ is $R(t_1(\overline{v}),\hdots,t_{n_R}(\overline{v}))$, where $R$ is a relation symbol and $t_1,\hdots,t_{n_R}$ are $\mathcal{L}$-terms, then
\begin{eqnarray*}
\mathcal{M}\models\phi(\overline{[g]}) &\mbox{iff}& (t_1^\mathcal{M}(\overline{[g]}),\hdots,t_{n_R}^\mathcal{M}(\overline{[g]}))\in R^\mathcal{M}\\
&\mbox{iff}& ([t_1^{\mathcal{M}_i}(\overline{g(i)})\,|\,i\in I],\hdots,[t_{n_R}^{\mathcal{M}_i}(\overline{g(i)})\,|\,i\in I])\in R^\mathcal{M} \,\,\mbox{applying the claim}\\
&\mbox{iff}& \{i\in I\,|\, (t_1^{\mathcal{M}_i}(\overline{g(i)}),\hdots,t_{n_R}^{\mathcal{M}_i}(\overline{g(i)}))\in R^{\mathcal{M}_i}\}\in\mathcal{D}\,\,\mbox{interpretation of $R$}\\
&\mbox{iff}& \{i\in I\,|\, \mathcal{M}_i\models \phi(\overline{g(i)})\}\in\mathcal{D}.
\end{eqnarray*}

If $\phi(\overline{v})$ is $\lnot\psi(\overline{v})$, where $\psi$ is an $\mathcal{L}$-formula for which our assertion is true, then
\begin{eqnarray*}
\mathcal{M}\models\phi(\overline{[g]}) &\mbox{iff}& \mathcal{M}\not\models\psi(\overline{[g]})\\
&\mbox{iff}& \{i\in I\,|\,\mathcal{M}_i\models\psi(\overline{g(i)})\}\notin\mathcal{D}\,\,\mbox{by induction}\\
&\mbox{iff}& I\setminus\{i\in I\,|\,\mathcal{M}_i\models\psi(\overline{g(i)})\}\in\mathcal{D}\,\,\mbox{since $\mathcal{D}$ is an ultrafilter}\\
&\mbox{iff}& \{i\in I\,|\,\mathcal{M}_i\not\models\psi(\overline{g(i)})\}\in\mathcal{D}\\
&\mbox{iff}& \{i\in I\,|\, \mathcal{M}_i\models \phi(\overline{g(i)})\}\in\mathcal{D}.
\end{eqnarray*}

If $\phi(\overline{v})$ is $\psi(\overline{v})\land\theta(\overline{v})$, where $\psi$ and $\theta$ are $\mathcal{L}$-formulas for which our assertion is true, then
\begin{eqnarray*}
\mathcal{M}\models\phi(\overline{[g]}) &\mbox{iff}& \mathcal{M}\models\psi(\overline{[g]})\,\,\mbox{and}\,\,\mathcal{M}\models\theta(\overline{[g]})\\
&\mbox{iff}& \{i\in I\,|\,\mathcal{M}_i\models\psi(\overline{g(i)})\}\in\mathcal{D}\,\,\mbox{and}\,\, \{i\in I\,|\,\mathcal{M}_i\models\theta(\overline{g(i)})\}\in\mathcal{D},
\end{eqnarray*}
by induction. Let $A = \{i\in I\,|\,\mathcal{M}_i\models\psi(\overline{g(i)})\}$ and $B = \{i\in I\,|\,\mathcal{M}_i\models\theta(\overline{g(i)})\}$. If $A\in\mathcal{D}$ and $B\in\mathcal{D}$, then $A\cap B\in\mathcal{D}$. Conversely, $A\cap B\subseteq A$ and $A\cap B\subseteq B$, so if $A\cap B\in\mathcal{D}$, then $A\in\mathcal{D}$ and $B\in\mathcal{D}$. Now,
\begin{eqnarray*}
A\cap B\in\mathcal{D} &\mbox{iff}& \{i\in I\,|\,\mathcal{M}_i\models\psi(\overline{g(i)})\,\,\mbox{and}\,\,\mathcal{M}_i\models\theta(\overline{g(i)})\}\in\mathcal{D}\\
&\mbox{iff}& \{i\in I\,|\, \mathcal{M}_i\models \phi(\overline{g(i)})\}\in\mathcal{D}.
\end{eqnarray*}

If $\phi(\overline{v})$ is $\exists w\, \psi(\overline{v},w)$, where $w$ is a variable and $\psi$ is an $\mathcal{L}$-formula for which our assertion is true, then $\mathcal{M}\models\phi(\overline{[g]})$ if and only if there exists $[h]\in\prod\mathcal{M}_i/\mathcal{D}$ such that $\mathcal{M}\models\psi(\overline{[g]},[h])$, if and only if (by induction) there exists $[h]$ such that $\{i\in I\,|\,\mathcal{M}_i\models\psi(\overline{g(i)},h(i))\}\in\mathcal{D}$. Let $A_{[h]} = \{i\in I\,|\,\mathcal{M}_i\models\psi(\overline{g(i)},h(i))\}$ and let $B = \{i\in I\,|\,\mathcal{M}_i\models \phi(\overline{g(i)})\}$. We would like to show that there exists $[h]$ such that $A_{[h]}\in\mathcal{D}$ if and only if $B\in\mathcal{D}$.

Suppose there exists such an $[h]$. Then $A_{[h]}\subseteq B$, since for $i\in A$, $\mathcal{M}_i\models\psi(\overline{g(i)},h(i))$, so $\mathcal{M}_i\models\phi(\overline{g(i)})$, since $h(i)$ satisfies the existential quantifier in $\mathcal{M}_i$. So $B\in\mathcal{D}$. Conversely, suppose $B\in\mathcal{D}$. Let $h\in\prod\mathcal{M}$ be such that for all $i\in B$, $h(i)\in\mathcal{M}_i$ is an element which satisfies the existential quantifier, and for $i\notin B$, $h(i)$ is an arbitrary element. Then $B\subseteq A_{[h]}$, so $A_{[h]}\in\mathcal{D}$.

This completes the proof by induction on formulas. We do not need to consider formulas constructed using $\lor$ or $\forall$, since these can be re-written to use only $\lnot$, $\land$, and $\exists$.
\end{proof}

\subsubsection*{Applications of Ultraproducts}
The Fundamental Theorem of Ultraproducts has a number of elegant consequences.

\begin{cor}\label{cor:el_class}
If $\mathcal{K}$ is an elementary class of $\mathcal{L}$-structures and $\{\mathcal{M}_i\}_{i\in I}$ is a collection of $\mathcal{L}$-structures in $\mathcal{K}$, indexed by an infinite set $I$, then for any ultrafilter $\mathcal{D}$ on $I$, $\prod\mathcal{M}_i/\mathcal{D}$ is in $\mathcal{K}$. 
\end{cor}
\begin{proof}
Let $T$ be a set of axioms for $\mathcal{K}$. For any $\phi\in T$, $\{i\in I\,|\, \mathcal{M}_i\models\phi\} = I$, since $\mathcal{M}_i\in\mathcal{K}$ for all $i\in I$. Now $I\in\mathcal{D}$, so $\prod\mathcal{M}_i/\mathcal{D}\models\phi$ by Theorem \ref{thm:FToUP}. Thus $\prod\mathcal{M}_i/\mathcal{D}\models T$, and $\prod\mathcal{M}_i/\mathcal{D}$ is in the class $\mathcal{K}$.
\end{proof}

\begin{exa}\label{exa:UP_char0}
Let $\mathcal{D}$ be a nonprincipal ultrafilter on the set of primes, $P$. By Corollary \ref{cor:el_class}, $\prod\mathbb{F}_p((t))/\mathcal{D}$ and $\prod\mathbb{Q}_p/\mathcal{D}$ are valued fields with cross section, since we saw in Example \ref{exa:VFel} that the class of valued fields with cross section is elementary. Also, $\prod\mathbb{Q}_p/\mathcal{D}$ has characteristic zero, since $\mathbb{Q}_p\models Char_0$ for all primes $p$.

Consider the characteristic of $\prod\mathbb{F}_p((t))/\mathcal{D}$. For any prime $p$, $\mathbb{F}_q((t))\models Char_p$ if and only if $q = p$, so $\{q\in P\,|\,\mathbb{F}_q((t))\models Char_p\} = \{p\}$. This is a finite set, so its complement is in the Frechet filter $\mathcal{D}_F$ on $P$. But $\mathcal{D}$ contains $\mathcal{D}_F$ by Lemma~\ref{lem:npuf}, so $P\setminus \{p\}\in\mathcal{D}$, and thus $\{p\}\notin\mathcal{D}$. Hence $\prod\mathbb{F}_p((t))/\mathcal{D}\models\lnot Char_p$ for all $p\in P$. This means that $\prod\mathbb{F}_p((t))/\mathcal{D}$ has characteristic zero.  

\end{exa}

This result suggests that the ultraproducts $\prod\mathbb{F}_p((t))/\mathcal{D}$ and $\prod\mathbb{Q}_p/\mathcal{D}$ are very similar structures. The next lemma supports this intuition and sheds some light on the structure of the ultraproducts as valued fields.

\begin{lem}\label{lem:UP_resval}
For any nonprincipal ultrafilter $\mathcal{D}$, the residue class fields and value groups of $\prod\mathbb{Q}_p/\mathcal{D}$ and $\prod\mathbb{F}_p((t))/\mathcal{D}$ are isomorphic. Specifically, 
\begin{itemize}
\item $\mathfrak{v}((\prod\mathbb{Q}_p/\mathcal{D})^*)\cong \prod\mathbb{Z}/\mathcal{D}\cong\mathfrak{v}((\prod\mathbb{F}_p((t))/\mathcal{D})^*)$, isomorphic as groups,
\item $\overline{\prod\mathbb{Q}_p/\mathcal{D}}\cong \prod\mathbb{F}_p/\mathcal{D}\cong\overline{\prod\mathbb{F}_p((t))/\mathcal{D}}$, isomorphic as fields, and
\item $char(\prod\mathbb{F}_p/\mathcal{D}) = 0$, so the residue class fields have characteristic zero.
\end{itemize}
\end{lem}
\begin{proof}
Let $\mathcal{M} = \prod\mathbb{Q}_p/\mathcal{D}$ and take $[f]\in M$, the domain of $\mathcal{M}$, with $[f]\neq [0]$. Then $f\not\sim_\mathcal{D} 0$, so $\{p\in P\,|\,f(p)\neq 0\}\in\mathcal{D}$, and we can find $g\sim_\mathcal{D} f$ such that $g(p)\neq 0$ for all $p$. We will use $g$ as our representative element for the equivalence class $[f]$. By definition, $\mathfrak{v}^\mathcal{M}([f]) = [\mathfrak{v}^{\mathbb{Q}_p}(g(p))\,|\,p\in P] = [p^{n_p}\,|\,p\in P]$, for some integers $n_p$. The value group consists of all equivalence classes of this form. Let $\phi$ be the function which maps $[p^{n_p}\,|\,p\in P]$ to $[n_p\,|\,p\in P]\in\prod\mathbb{Z}/\mathcal{D}$. It is easy to see that $\phi$ is a bijection, and $\phi([p^{n_p}\,|\,p\in P] \cdot [p^{m_p}\,|\,p\in P]) = \phi([p^{n_p+m_p}\,|\,p\in P]) = [n_p+m_p\,|\,p\in P] = [n_p\,|\,p\in P] + [m_p\,|\,p\in P] = \phi([p^{n_p}\,|\,p\in P])) + \phi([p^{m_p}\,|\,p\in P])$, so $\phi$ is an isomorphism of groups.

The same argument shows that $\mathfrak{v}(\prod\mathbb{F}_p((t))/\mathcal{D})\cong \prod\mathbb{Z}/\mathcal{D}$, with the minor difference that elements of the value group are of the form $[t^{n_p}\,|\,p\in P]$, where the $n_p$ are integers, since we choose $t$ instead of $p$ as a prime element for $F_p((t))$.

Thus, $\mathfrak{v}((\prod\mathbb{Q}_p/\mathcal{D})^*)\cong \mathfrak{v}((\prod\mathbb{F}_p((t))/\mathcal{D})^*)$.

Now consider $\overline{\mathcal{M}} = \mathcal{O}/\{[f]\in\mathcal{O}\,|\,\mathcal{M}\models 1<\mathfrak{v}([f])\}$ (here we use $t_1<t_2$ as an abbreviation for $(t_1\leq t_2)\land\lnot(t_1=t_2)$). 
This is $\mathcal{O}/\{[f]\in\mathcal{M}\,|\,\{p\,|\,\mathbb{Q}_p\models 1<\mathfrak{v}(f(p))\}\in\mathcal{D}\}$. The ideal in the denominator is the maximal idea $I_1$.

Consider the set $R$ of distinct $[g]\in\mathcal{O}$ such that $g(p)\in\{0,\hdots,p-1\}$ for all $p\in P$. We would like to show that this set is a complete set of representatives for the residue class field. 

Take $[g],[h]\in R$, $[g]\neq [h]$. Suppose $[g]\equiv[h]$ (mod $I_1$). Then we have $[g]+[f] = [h]$ for some $[f]\in I_1$. Now if $[f]\in I_1$, then $\{p\,|\,\mathbb{Q}_p\models1<\mathfrak{v}(f(p))\}\in \mathcal{D}$, and since $[g]\neq [h]$, $\{p\,|\,g(p)\neq h(p)\}\in\mathcal{D}$. But for any $p$ for which both $g(p)\neq h(p)$ and $1<\mathfrak{v}(f(p))$, we have $g(p)+f(p)\neq h(p)$, since $g(p)$ and $h(p)$ are taken from $\{0,\hdots,p-1\}$, and their difference cannot be divisible by $p$.

So $\{p\,|\,\mathbb{Q}_p\models1<\mathfrak{v}(f(p))\}\cap \{p\,|\,g(p)\neq h(p)\} \subseteq \{p\,|\,g(p)+f(p)\neq h(p)\}\in \mathcal{D}$, and hence $g+f\not\sim_D h$, which is a contradiction. Thus distinct elements of $R$ are representatives for distinct equivalence classes mod $I_1$. 

But every equivalence class mod $I_1$ has a representative in $R$, since we can take any representative $f$ of the class and reduce each coordinate $f(p)$ mod $p$ to an element of $\{0,\hdots,p-1\}$. Letting $[g]$ be the resulting element of $R$, the difference $[f]-[g] = [f-g]$ is an element of $I_1$, since each coordinate of $f-g$is divisible by $p$, so $[g]$ is a representative for the class.

Thus $R$ is a complete set of representatives for the classes mod $I_1$. Interpreting each $[g]\in R$ as an element of $\prod\mathbb{F}_p/\mathcal{D}$ in the obvious way, it is easy to verify that the resulting map from $\overline{\mathcal{M}}$ to $\prod\mathbb{F}_p/\mathcal{D}$ is an isomorphism of fields.

The same argument holds for $\prod\mathbb{F}_p((t))/\mathcal{D}$, the only difference being that we note that the difference between distinct $g(p)$ and $h(p)$ from $\{0,\hdots,p-1\}$ cannot be divisible by $t$. 

Thus, $\overline{\prod\mathbb{Q}_p/\mathcal{D}}\cong\overline{\prod\mathbb{F}_p((t))/\mathcal{D}}$.

Finally, the residue class fields have characteristic zero by the same argument given in Example~\ref{exa:UP_char0}. That is, $\mathbb{F}_q\models Char_p$ only when $q=p$, so $\prod\mathbb{F}_p/\mathcal{D}\models\lnot Char_p$ for all $p$, and thus $char(\prod\mathbb{F}_p/\mathcal{D}) = 0$. 
\end{proof}

Next we will prove a theorem which is a significant step toward a proof of the Ax-Kochen Principle.

\begin{thm}\label{thm:template}
Let $\{\mathcal{M}_i\}_{i\in I}$ and $\{\mathcal{N}_i\}_{i\in I}$ be collections of $\mathcal{L}$-structures, indexed by the same infinite set $I$. Suppose that for all nonprincipal ultrafilters $\mathcal{D}$, $\prod\mathcal{M}_i/\mathcal{D}\equiv\prod\mathcal{N}_i/\mathcal{D}$. Then for any $\mathcal{L}$-sentence $\phi$, $\mathcal{M}_i\models\phi$ for all but finitely many $i$ if and only if $\mathcal{N}_i\models\phi$ for all but finitely many $i$.
\end{thm}
\begin{proof}
Suppose that $\mathcal{M}_i\models\phi$ for all but finitely many $i$. Then $A = \{i\in I\,|\,\mathcal{M}_i\models\phi\}\in\mathcal{D}_F$, the Frechet filter on $I$. By Lemma \ref{lem:npuf}, all nonprincipal ultrafilters extend the Frechet filter, so $A\in\mathcal{D}$ for all nonprincipal ultrafilters $\mathcal{D}$. By Theorem \ref{thm:FToUP}, $\prod\mathcal{M}_i/\mathcal{D}\models\phi$, and by elementarily equivalence, $\prod\mathcal{N}_i/\mathcal{D}\models\phi$. Again by Theorem \ref{thm:FToUP}, $B = \{i\in I\,|\,\mathcal{N}_i\models\phi\}\in\mathcal{D}$ for all nonprincipal ultrafilters $\mathcal{D}$, so $B\in\bigcap\mathcal{D} = \mathcal{D}_F$ by Lemma \ref{lem:npuf}. Thus $\mathcal{N}_i\models\phi$ for all but finitely many $i$.

The converse follows symmetrically.
\end{proof}

Now that we have established Theorem~\ref{thm:template}, the Ax-Kochen Principle will be proven if we can demonstrate the elementary equivalence of the ultraproducts $\prod\mathbb{F}_p((t))/\mathcal{D}$ and $\prod\mathbb{Q}_p/\mathcal{D}$ for all nonprincipal ultrafilters $\mathcal{D}$ on the set $P$ of all primes. But to do so, we will need to appeal to more powerful techniques from model theory and dig deeper into the algebraic properties of $\mathbb{F}_p((t))$ and $\mathbb{Q}_p$.

\subsection{Types and Saturated Models}\label{subsec:TaSM}

This section is largely concerned with a useful technique for proving that two structures are isomorphic, the back-and-forth argument. We will begin with a demonstration of a simple back-and-forth argument in the case of countable dense linear orders without endpoints. We will then generalize the property of these structures which makes the back-and-forth argument possible by discussing types and saturated models.

We will often refer to ordinal and cardinal numbers, and we will assume some knowledge of transfinite arithmetic and transfinite induction. For more information, see Appendix~\ref{sec:OCTI}.

Let $\mathcal{L}_O = \{<\}$, where $<$ is a binary relation symbol, and let DLO be the $\mathcal{L}_O$-theory of dense linear orders without endpoints:
\begin{enumerate}
\item $\forall x\,\lnot(x<x)$
\item $\forall x\forall y\forall z\,((x<y\land y<z)\rightarrow x<z)$
\item $\forall x\forall y\, (x<y\lor x=y\lor y<x)$
\item $\forall x\forall y\, (x<y\rightarrow \exists z\,(x<z\land z<y))$
\item $\forall x\exists y\exists z\, (y<x\land x<z)$
\end{enumerate}

One example of a model for DLO is the $\mathcal{L}_O$-structure $\langle\mathbb{Q},<\rangle$, where the interpretation of the symbol $<$ is the usual ordering of $\mathbb{Q}$. It is well known that $\mathbb{Q}$ is countable. We will use a back-and-forth argument to show that up to isomorphism, $\langle\mathbb{Q},<\rangle$ is the only countable model for DLO. Note that this means up to $\mathcal{L}_O$-isomorphism as an $\mathcal{L}_O$-structure; we view $\mathbb{Q}$ only as an ordered set, not as a field.

\begin{thm}[{\cite[Theorem 2.4.1]{Marker}}]\label{thm:DLO}
Let $\mathcal{A}$ and $\mathcal{B}$ be $\mathcal{L}_O$-structures with domains $A$ and $B$ such that $\mathcal{A}\models \text{DLO}$, $\mathcal{B}\models\text{DLO}$, and $|A|=|B|=\aleph_0$. Then $\mathcal{A}\cong \mathcal{B}$.
\end{thm}
\begin{proof}
Since $A$ and $B$ are countable, we can choose enumerations $a_0,a_1,\hdots$ and $b_0,b_1,\hdots$ of $A$ and $B$. We will inductively construct a sequence of functions $f_i:A_i\rightarrow B_i$, between finite subsets $A_i\subset A$ and $B_i\subset B$, satisfying the following properties for each $i\geq 0$:
\begin{enumerate}
\item For all $j<i$, $A_j\subseteq A_i$, $B_j\subseteq B_i$, and $f_j\subseteq f_i$. That is, if $a\in A_j$, then $a\in A_i$, $f_j(a)\in B_i$, and $f_i(a) = f_j(a)$. 
\item If $a < a'$, with $a,a'\in A_i$, then $f_i(a)<f_i(a')$. That is, each $f_i$ is an $\mathcal{L}_O$-homomorphism between the substructures of $\mathcal{A}$ and $\mathcal{B}$ with domains $A_i$ and $B_i$.
\item We have $\{a_0,\hdots,a_{i-1}\}\subseteq A_i$ and $\{b_0,\hdots,b_{i-1}\}\subseteq B_i$.
\item The function $f_i$ is a bijection between $A_i$ and $B_i$. 
\end{enumerate}

Given such a sequence of functions, we let $f = \bigcup_{i=0}^\infty f_i:\bigcup_{i=0}^\infty A_i\rightarrow \bigcup_{i=0}^\infty B_i$. Property 1 ensures that $f$ is well-defined, for if $a\in\bigcup_{i=0}^\infty A_i$, there is some $N\geq0$ such that $a\in A_j$ for all $j\geq N$, and the $f_j$ agree on $a$ for all $j\geq N$. Property 2 ensures that $f$ is an $\mathcal{L}_O$-homomorphism. Property 3 ensures that $\bigcup_{i=0}^\infty A_i = A$ and $\bigcup_{i=0}^\infty B_i = B$, since the $a_i$ and $b_i$ enumerate all of $A$ and $B$.  Property 4 ensures that $f$ is a bijection $A\rightarrow B$, and thus an isomorphism between $\mathcal{A}$ and $\mathcal{B}$.

For the base case, let $A_0 = B_0 = f_0 = \emptyset$. The four properties are trivially satisfied.

Given $f_i$ satisfying the four properties, we first extend $f_i$ to a function $g_{i+1}:A_{i+1}'\rightarrow B_{i+1}'$ (with $A_i\subseteq A_{i+1}'$ and $B_i\subseteq B_{i+1}'$) by ensuring that $a_{i}$ is in the domain. Then we extend $g_{i+1}$ to the next function $f_{i+1}: A_{i+1}\rightarrow B_{i+1}$ in the sequence by ensuring that $b_{i}$ is in the range. Going ``back and forth'' in this way, we will ensure that $f_{i+1}$ satisfies property 3.

If $a_{i+1}\in A_i$, then we simply let $A_{i+1}'=A_i$, $B_{i+1}'=B_i$, and $g_{i+1} = f_i$. Otherwise, we must pick an element $b^*\in B$ onto which to map $a_{i+1}$. Property 4 requires that $b^*\notin B_i$ (otherwise $g_{i+1}$ would not be injective), and property 2 requires that for all $a\in A_i$, $a<a_{i+1}$ if and only if $f_i(a)<b^*$.

Exactly one of the following holds:
\begin{enumerate}
\item $a_{i+1}$ is less than every element of $A_i$, or
\item $a_{i+1}$ is greater than every element of $A_i$, or
\item since $A_i$ is finite, there exists a greatest element less than $a_{i+1}$, $\alpha$, and a least element greater than $a_{i+1}$, $\beta$, such that $\alpha < \beta$.
\end{enumerate}

In the first case, since $\mathcal{B}\models \text{DLO}$, it has no greatest element, so there is some $b^*\in B\setminus B_i$ greater than every element of $B_i$. In the second case, $\mathcal{B}$ has no least element, so there is some $b^*\in B\setminus B_i$ less than every element of $B_i$. In the third case, $\mathcal{B}$ is dense, so we can find $b^*\in B\setminus B_i$ such that $f_i(\alpha) < b^* < f_i(\beta)$. In any case, we have satisfied $b^*\notin B_i$ and $a<a_{i+1}$ if and only if $f_i(a)<b^*$.

Now define $A_{i+1}'=A_i\cup\{a_{i+1}\}$, $B_{i+1}'=B_i\cup\{b^*\}$, and $g_{i+1}:A_{i+1}'\rightarrow B_{i+1}'$ such that $g_{i+1}(a_{i+1}) = b^*$ and for all $a\in A_i$, $g_{i+1}(a) = f_i(a)$. 

The other direction is symmetric. If $b_{i+1}\in B_i$, then we simply let $A_{i+1}=A_{i+1}'$, $B_{i+1}=B_{i+1}'$, and $f_{i+1} = g_{i+1}$. Otherwise, we must pick an element $a^*\in A$ to map onto $b_{i+1}$. We must have $a^*\notin A_i$ (otherwise $f_{i+1}$ would not be well-defined), and property 2 requires that for all $a\in A_i$, $a<a^*$ if and only if $g_i(a)<b_{i+1}$. Using the fact that $\mathcal{A}\models \text{DLO}$, we can pick such an $a^*$ in the same way we picked $b^*$.

Now define $A_{i+1}=A_{i+1}'\cup\{a^*\}$, $B_{i+1}=B_{i+1}'\cup\{b_{i+1}\}$, and $f_{i+1}$ such that $f_{i+1}(a^*) = b_{i+1}$ and for all $a\in A_i$, $f_{i+1}(a) = g_{i+1}(a)$. Note that we have maintained properties 1-4, and by induction we can construct the required sequence of functions.
\end{proof}

The key property of models of DLO which allows the back-and-forth argument to work is this: given a finite subset of the domain, if we specify a place in the ordering relative to the elements of the subset where we would like to find some new element, we are guaranteed to be able to find such an element, provided that its existence would not break the linearity of the ordering.

\subsubsection*{Types}
To our toolbox of formulas, sentences, and theories, we add a new way of expressing first-order properties, types. Types will allow us to express properties of elements of a structure relative to a distinguished set of other elements. More precisely, given a set $A$ of elements of a structure $\mathcal{M}$, a type captures the relationships that other elements could have relative to the elements of $A$. 

Let $\mathcal{M}$ be an $\mathcal{L}$-structure with domain $M$. Given $A\subseteq M$, let $\mathcal{L}_A$ be the language $\mathcal{L}\cup\{c_a\,|\,a\in A\}$ where each $c_a$ is a new distinct constant symbol. We view $\mathcal{M}$ is an $\mathcal{L}_A$-structure by interpreting $c_a^{\mathcal{M}} = a$ for each new constant symbol.

When we are working over $\mathcal{L}$, sentences and formulas may only refer explicitly to the elements of $M$ which are interpretations of the constant symbols of $\mathcal{L}$. By expanding the language, we are allowing sentences to refer to the elements of $A$. Let $Th_A(\mathcal{M}) = \{\phi \,|\, \phi \,\mbox{is an $\mathcal{L}_A$-sentence, and}\, \mathcal{M}\models\phi\}$. This theory is an extension of $Th(\mathcal{M})$, consisting of all sentences which are true in $\mathcal{M}$, when we are allowed to explicitly refer to the elements of $A$. 

Earlier (Definition~\ref{def:sat}), we defined satisfiability of an $\mathcal{L}$-theory. There is also a concept of satisfiability of a set of $\mathcal{L}$-formulas.

\begin{defin}\label{def:formulasat}
A set of $\mathcal{L}$-formulas $S$ with free variables $v_1,\hdots,v_n$ is called \emph{satisfiable} if there is an $\mathcal{L}$-structure $\mathcal{M}$ with domain $M$ and elements $a_1,\hdots,a_n\in M$ such that $\mathcal{M}\models\phi(a_1,\hdots,a_n)$ for all formulas $\phi\in S$. Note that the elements of the domain are fixed. The same value must be substituted for the same variable across all formulas.
\end{defin}

\begin{defin}\label{def:type}
Given an $\mathcal{L}$-structure $\mathcal{M}$ with domain $M$ and a subset $A\subseteq M$, an \emph{$n$-type} over $A$ is a set $P$ of $\mathcal{L}_A$ formulas in free variables $v_1,\hdots,v_n$, such that $Th_A(\mathcal{M})\cup P$ is satisfiable. If for all $\mathcal{L}_A$-formulas $\phi$ in free variables $v_1,\hdots,v_n$, either $\phi\in P$ or $\lnot\phi\in P$, then $P$ is called \emph{complete}.
\end{defin}

The satisfiability condition here means that there is some $\mathcal{L}_A$-structure $\mathcal{N}$ with domain $N$ which is a model for $Th_A(\mathcal{M})$, and that there are elements $b_1,\hdots,b_n\in N$ such that $\mathcal{N}\models \phi(b_1,\hdots,b_n)$ for all $\phi\in P$. We say that the elements $b_1,\hdots,b_n$ realize $P$ in $\mathcal{N}$. If $P$ is not realized in $\mathcal{M}$ we say that $\mathcal{M}$ omits the type $P$.

\begin{exa}\label{exa:omitted_type}
We return to the example of $\langle\mathbb{Q},<\rangle$ as an $\mathcal{L}_O$-structure. Let $A = \mathbb{N} \subset \mathbb{Q}$. We will define two types over $A$.

Let $P = \{c_0<v_1, c_1<v_1, c_2<v_1,\hdots\}$. Note that the elements of $\mathbb{N}$ appear (represented by constant symbols) in the formulas of $P$. In order to show that $P$ is a 1-type over $A$, we must show that $Th_A(\langle\mathbb{Q},<\rangle)\cup P$ is satisfiable. Let $\Delta\subset Th_A(\langle\mathbb{Q},<\rangle)\cup P$ be a finite subset. Only finitely many of the formulas in $P$ appear in $\Delta$, so we let $i$ be the maximum integer such that $i<v_1\in\Delta$. Then for all $\phi(v_1) = j<v_1 \in\Delta\cap P$, $\langle\mathbb{Q},<\rangle\models \phi(i+1)$. The other sentences in $\Delta$ are true in $\langle\mathbb{Q},<\rangle$ by definition, so $\langle\mathbb{Q},<\rangle\models \Delta$. Thus $Th_A(\langle\mathbb{Q},<\rangle)\cup P$ is finitely satisfiable, and therefore satisfiable by Compactness.

So $P$ is a 1-type over $A$, but $\langle\mathbb{Q},<\rangle$ omits $P$, since there is no rational number which is greater than every natural number.

Let $Q = \{\phi(v_1)\,|\,\langle\mathbb{Q},<\rangle\models\phi(\frac{1}{2})\}$. The element $\frac{1}{2}$ realizes $Q$ in $\langle\mathbb{Q},<\rangle$, so $Q\cup Th_{\mathbb{Q}}(\langle\mathbb{Q},<\rangle)$ is clearly satisfiable, and $Q$ is a 1-type over $A$. Moreover, for any $\mathcal{L}_A$-formula $\psi$ in one free variable, $\langle\mathbb{Q},<\rangle\models \psi(\frac{1}{2})$ or $\langle\mathbb{Q},<\rangle\models\lnot\psi(\frac{1}{2})$, so either $\psi\in Q$ or $\lnot\psi\in Q$. Thus $Q$ is complete.
\end{exa}

There is a useful generalization of the type $Q$ in Example $\ref{exa:omitted_type}$. For any $\mathcal{L}$-structure $\mathcal{M}$ with domain $M$, $A\subseteq M$, and elements $m_1,\hdots,m_n\in M$, we define $tp^{\mathcal{M}}(m_1,\hdots,m_n/A) = \{\phi(v_1,\hdots,v_n)\,|\,\mathcal{M}\models\phi(m_1,\hdots,m_n)\}$.
By the argument given in the example, this type, called the complete type of $a_1,\hdots,a_n$ over $A$, is a complete $n$-type which is realized in $\mathcal{M}$.

\subsubsection*{Saturated Models}
A $\kappa$-saturated model realizes all types over sets of cardinality less than $\kappa$.

\begin{defin}\label{def:saturated}
Let $T$ be a complete theory with infinite models in a countable language $\mathcal{L}$. Let $\kappa$ be an infinite cardinal. A model $\mathcal{M}\models T$ with domain $M$ is called \emph{$\kappa$-saturated} if for all $A\subset M$ with $|A|<\kappa$ every type over $A$ is realized in $\mathcal{M}$.
\end{defin}

In Theorem~\ref{thm:DLO}, we constructed an isomorphism between any two countable models for $DLO$. Using a similar argument, we can construct partial elementary bijections between subsets of the domains of $\kappa$-saturated models. 

\begin{thm}[{\cite[Lemma 5.1.11]{Chang}}]\label{thm:saturatedbijection}
Let $\kappa$ be an infinite cardinal, and let $\mathcal{M}$ and $\mathcal{N}$ be $\kappa$-saturated models of a complete theory $T$ with domains $M$ and $N$ respectively. Let $A\subseteq M$ and $B\subseteq N$ be subsets such that $|A| = |B| = \kappa$. Then there is a partial elementary bijection $f:\widetilde{A}\rightarrow \widetilde{B}$, where $A\subseteq \widetilde{A}\subseteq M$ and $B\subseteq \widetilde{B}\subseteq N$, and $|\widetilde{A}|=|\widetilde{B}|=\kappa$. By a partial elementary bijection, we mean that $\mathcal{M}\models\phi(x_1,\hdots,x_n)$ for $\phi$ an $\mathcal{L}$-formula and $x_1,\hdots,x_n\in \widetilde{A}$ if and only if $\mathcal{N}\models\phi(f(x_1),\hdots,f(x_n))$. 

\end{thm}
\begin{proof}
Let $(a_\alpha\,:\,\alpha<\kappa)$ and $(b_\alpha\,:\,\alpha<\kappa)$ be enumerations of $A$ and $B$ respectively.

We will inductively construct a sequence of functions $(f_\alpha\,:\,\alpha<\kappa)$ such that each $f_\alpha$ is a partial elementary bijection with domain $A_\alpha\subset M$ and range $B_\alpha\subset N$. We require $f_\alpha$ to satisfy the following properties for all $\alpha<\kappa$:
\begin{enumerate}
\item For all $\beta<\alpha$, $A_\beta\subseteq A_\alpha$, $B_\beta\subseteq B_\alpha$, and $f_\beta\subseteq f_\alpha$.
\item We have $a_\alpha\in A_{\alpha+1}$ and $b_\alpha\in B_{\alpha+1}$.
\item The function $f_\alpha$ is a bijection between $A_\alpha$ and $B_\alpha$.
\item We have $|A_\alpha| \leq |2\alpha| < \kappa$ and $|B_\alpha|\leq |2\alpha| < \kappa$.
\end{enumerate}

Given such a sequence of functions, we let $\widetilde{A}=\bigcup_{\alpha<\kappa}A_\alpha$, $\widetilde{B}=\bigcup_{\alpha<\kappa}B_\alpha$, and $f = \bigcup_{\alpha<\kappa} f_\alpha$. Properties 1 and 3 guarantee that $f$ is a well-defined elementary bijection $\widetilde{A}\rightarrow\widetilde{B}$, property 2 guarantees that $A\subseteq\widetilde{A}$ and $B\subseteq\widetilde{B}$, and property 4 guarantees that $|\widetilde{A}|=|\widetilde{B}|=\kappa$.

For the base case, let $A_0 = B_0 = f_0 = \emptyset$. The properties are trivially satisfied.

If $\alpha$ is a limit ordinal, we define $A_\alpha = \bigcup_{\beta<\alpha}A_\beta$, $B_\alpha = \bigcup_{\beta<\alpha}B_\beta$, and $f_\alpha = \bigcup_{\beta<\alpha}f_\beta$. Property 1 is clearly satisfied. Property 2 only requires certain elements to be in the domain and range of successor ordinals, so it is trivially satisfied. Property 3 is satisfied, since every element of $B_\alpha$ is in the range of some $f_\beta$ for $\beta<\alpha$, since the $f_\beta$ are surjective, and any two elements in $A_\alpha$ are in some $A_\beta$ for $\beta<\alpha$, so they are sent to distinct elements, since the $f_\beta$ are injective. Property 4 is true by transfinite arithmetic: $\alpha$ is the limit of the $\beta<\alpha$, so $|2\alpha|$ is the limit of $|2\beta|$ for $\beta<\alpha$, which bounds the cardinality of $\bigcup_{\beta<\alpha}A_\beta$ above. The argument for the cardinality of $B_\alpha$ is the same. Finally, $f_\alpha$ is elementary, since the same holds for all $f_\beta$, $\beta<\alpha$.

In the successor case, given $f_\alpha$ satisfying the properties, we first extend $f_\alpha$ to a function $g_{\alpha+1}:A_{\alpha+1}'\rightarrow B_{\alpha+1}'$ by ensuring that $a_{\alpha}$ is in the domain. Then we extend $g_{\alpha+1}$ to the next function $f_{\alpha+1}: A_{\alpha+1}\rightarrow B_{\alpha+1}$ in the sequence by ensuring that $b_{\alpha}$ is in the range. 

If $a_{\alpha}\in A_\alpha$, then we simply let $g_{\alpha+1} = f_\alpha$. Otherwise, we must pick an element $n\in N$ onto which to map $a_\alpha$. Consider the language $\mathcal{L}_{A_\alpha}$, which is $\mathcal{L}$ extended with a new constant symbol for each element in $A_\alpha$. We may consider $\mathcal{N}$ as a $\mathcal{L}_{A_\alpha}$-structure by interpreting the constant symbol $c_a$ (representing $a\in A_\alpha$) as $f_\alpha(a)\in B_\alpha$. We will identify the languages $\mathcal{L}_{A_\alpha}$ and $\mathcal{L}_{B_\alpha}$ by choosing the same constant symbol, $c_a$, to represent $a\in A_\alpha$ and $f_\alpha(a)\in B_\alpha$.  

Now for any $\mathcal{L}_{A_\alpha}$-sentence $\phi$, enumerate the new constant symbols which appear in $\phi$, $c_{m_1},\hdots,c_{m_t}$, and let $m_1,\hdots,m_t$ and $n_1,\hdots,n_t$ be their interpretations in $\mathcal{M}$ and $\mathcal{N}$, respectively. Form an $\mathcal{L}$-sentence $\psi$ by replacing each new constant symbol $c_{m_i}$ with a distinct variable $v_i$. Then $\mathcal{M}\models\phi$ if and only if $\mathcal{M}\models\psi(m_1,\hdots,m_t)$. Since $f_\alpha$ is elementary and $f_\alpha(m_i) = n_i$ for all $i$, $\mathcal{M}\models\psi(m_1,\hdots,m_t)$ if and only if $\mathcal{N}\models\psi(n_1,\hdots,n_t)$, if and only if $\mathcal{N}\models\phi$. Thus $Th_{A_\alpha}(\mathcal{M}) = Th_{B_\alpha}(\mathcal{N})$. 

We can easily show by the method in Example~\ref{exa:omitted_type} that $tp^\mathcal{M}(a_\alpha/A_\alpha)$ is a complete type realized in $\mathcal{M}$. Thus $Th_{A_\alpha}(\mathcal{M})\cup tp^\mathcal{M}(a_\alpha/A_\alpha)$ is satisfiable, so $Th_{B_\alpha}(\mathcal{N})\cup tp^\mathcal{M}(a_\alpha/A_\alpha)$ is satisfiable, and $tp^\mathcal{M}(a_\alpha/A_\alpha)$ is a complete type over $B_\alpha$ in $\mathcal{N}$, since we have identified the languages $\mathcal{L}_{A_\alpha}$ and $\mathcal{L}_{B_\alpha}$. Since $|B_\alpha|<\kappa$, and $\mathcal{N}$ is $\kappa$-saturated, $tp^\mathcal{M}(a_\alpha/A_\alpha)$ is realized in $\mathcal{N}$. Let $n$ be an element of $N$ which realizes this type.

Note that $n\notin B_\alpha$. For otherwise, letting $\phi$ be the $\mathcal{L}_{A_\alpha}$-formula $v_1=c_{n}$, $\mathcal{N}\models \phi(n)$, so $\mathcal{M}\models \phi(a_\alpha)$, and $a_\alpha = m$, where $m$ is the interpretation of $c_{n}$ in $A_\alpha$, and hence $a_\alpha\in A_\alpha$. But this contradicts the assumption that $a_\alpha\notin A_\alpha$.

Now define $g_{\alpha+1}$ by extending $f_\alpha$ such that $g_{\alpha+1}(a_\alpha) = n$. It is evident that $g_{\alpha+1}$ is elementary, for $\mathcal{M}\models\phi(a_\alpha,m_1,\hdots,m_t)$ with $m_1,\hdots,m_t\in A_\alpha$ if and only if $\phi(v_1,c_{m_1},\hdots,c_{m_t})\in tp^\mathcal{M}(a_\alpha/A_\alpha)$, if and only if $\mathcal{N}\models\phi(n,f_\alpha(m_1),\hdots,f_\alpha(m_t))$. 

The other direction is symmetric. If $b_{\alpha}\in B_\alpha$, then we simply let $f_{\alpha+1} = g_{\alpha+1}$. Otherwise, we may pick an element $m\in M$ to map onto $b_{\alpha}$ such that $m$ realizes $tp^\mathcal{N}(b_\alpha/B_\alpha)$. We define $f_{\alpha+1}$ by extending $g_{\alpha+1}$ such that $f_{\alpha+1}(m) = b_\alpha$. By the same argument, $f_{\alpha+1}$ is elementary.

Properties 1 and 2 are clearly satisfied. It is also clear that $f_{\alpha+1}$ is a surjection, since for every element we have added to the range, we have added an element to the domain mapping to it. By the observations that $m\notin A_\alpha$ and $n\notin B_\alpha$, $f_{\alpha+1}$ is well-defined and injective, so Property 3 is satisfied. Finally, we have added at most two elements to the domain and range, so since $|A_\alpha| \leq |2\alpha|$, $|A_{\alpha+1}|\leq |2(\alpha+1)|$. The same argument holds for the cardinality of $B_{\alpha+1}$, so Property 4 is satisfied.

Thus by transfinite induction we are able to construct the required sequence of functions. This completes the proof.
\end{proof}

Note that if a model $\mathcal{M}$ with domain $M$ is $\kappa$-saturated, then we must have $\kappa\leq|M|$. For otherwise, if $|M|<\kappa$, then $\mathcal{M}$ would realize every type over $M$, the entire domain. In particular, it would realize the type $\{\lnot(v_1=c_m)\,|\,m\in M\}$. Any element $x\in M$ which realizes this type is not equal to any element of $M$, which is a contradiction. 

If $\mathcal{M}$ is $|M|$-saturated, that is, as saturated as possible, then we simply say that $\mathcal{M}$ is saturated. As a corollary to the last theorem, saturated models of a given cardinality are unique up to isomorphism.

\begin{cor}\label{cor:saturatedisomorphism}
If $\mathcal{M}$ and $\mathcal{N}$ are saturated models of a complete theory, $T$, and they have the same cardinality $\kappa$, then $\mathcal{M}\cong\mathcal{N}$. 
\end{cor}
\begin{proof}
We apply Theorem~\ref{thm:saturatedbijection}, taking as our subsets the entire domains $M$ and $N$ of $\mathcal{M}$ and $\mathcal{N}$, since $|M|=|N|=\kappa$. Then there is a function $f:M\rightarrow N$ which is an elementary bijection, and thus an isomorphism, between $\mathcal{M}$ and $\mathcal{N}$.
\end{proof}

\subsubsection*{Existence of Saturated Models}
Now that the utility of saturated models for demonstrating isomorphism is clear, we will take up the issue of their existence. The following lemma shows that if a type $P$ is omitted, we can always find an elementary extension in which $P$ is realized. 

\begin{lem}\label{lem:realizing_types}
Let $\mathcal{M}$ be an infinite $\mathcal{L}$-structure with domain $M$, $A\subseteq M$, and $P$ an $n$-type over $A$. Then there exists an elementary extension of $\mathcal{M}$, $\mathcal{N}$, such that $P$ is realized in $\mathcal{N}$. If $\mathcal{L}$ is countable, we can take $\mathcal{N}$ to have the same cardinality as $\mathcal{M}$. 
\end{lem}
\begin{proof}
Since $P$ is a type over $A$, $P\cup Th_A(\mathcal{M})$ is satisfiable. Let $\mathcal{N}_0$ be an $\mathcal{L}_A$-structure which satisfies $P\cup Th_A(\mathcal{M})$, and let $x_1,\hdots,x_n$ be the elements realizing $P$ in $\mathcal{N}_0$.

Let $\Gamma = P\cup Diag_{el}(\mathcal{M})$. We will apply Compactness to prove that $\Gamma$ is satisfiable by showing that $\mathcal{N}_0$ satisfies every finite subset of $\Gamma$.

Note that $P$ consists of $\mathcal{L}_A$-formulas (with constant symbols for each element of $A$) and $Diag_{el}(\mathcal{M})$ consists of $\mathcal{L}_\mathcal{M}$-sentences (with constant symbols for each element of $M$). Since $\mathcal{L}_A\subseteq\mathcal{L}_\mathcal{M}$, we can consider the formulas in $P$ as $\mathcal{L}_\mathcal{M}$-formulas.

Let $\Delta$ be a finite subset of $\Gamma$. There are finitely many $\mathcal{L}_\mathcal{M}$-formulas $\phi_1,\hdots,\phi_s\in P\cap\Delta$ and finitely many $\mathcal{L}_\mathcal{M}$-sentences $\psi_1,\hdots,\psi_t\in Diag_{el}(\mathcal{M})\cap\Delta$. Let $\Phi(v_1,\hdots,v_n)$ be the $\mathcal{L}_\mathcal{M}$-formula $\phi_1\land \phi_2\land\hdots\land\phi_s$, and let $\Psi$ be the $\mathcal{L}_{\mathcal{M}}$-sentence $\psi_1\land\psi_2\land\hdots\land\psi_t$.

Let $c_{a_1},\hdots,c_{a_j},c_{b_1},\hdots,c_{b_k}$ be the new constant symbols of $\mathcal{L}_\mathcal{M}$ which appear in the formulas and sentences of $\Delta$. The symbol $c_{a_i}$ corresponds to the element $a_i\in A$ and the symbol $c_{b_i}$ corresponds to the element $b_i\in M\setminus A$.
 
Now to show that the $\mathcal{L}_A$-structure $\mathcal{N}_0$ satisfies $\Delta$, we must turn it into an $\mathcal{L}_\mathcal{M}$ structure by assigning interpretations to the constant symbols $c_m$ for all $m\in M$. But since $c_{b_1}\hdots,c_{b_k}$ are the only symbols appearing in $\Delta$, all other $c_m$ may be assigned interpretations arbitrarily. Then it will suffice to show that $\mathcal{N}_0\models\Phi(x_1,\hdots,x_n)\land\Psi$.

Let $\Psi'(w_1,\hdots,w_k)$ be the $\mathcal{L}_A$-formula formed by replacing each constant symbol $c_{b_i}$ in $\Psi$ with the variable $w_i$. Let $\theta$ be the $\mathcal{L}_A$-sentence $\exists w_1\hdots\exists w_k\,\Psi'(w_1,\hdots,w_k)$. Now $\mathcal{M}\models \Psi'(b_1,\hdots,b_j)$, so $\mathcal{M}\models\theta$, and thus $\theta\in Th_A(\mathcal{M})$. But $\mathcal{N}_0\models Th_A(\mathcal{M})$, so $\mathcal{N}_0\models\theta$. 

Thus there are elements $y_1,\hdots,y_k$ in the domain of $\mathcal{N}_0$ such that $\mathcal{N}_0\models\Psi'(y_1,\hdots,y_k)$. Interpreting the constant $c_{b_i}$ as $y_i$ for each $i$, we see that $\mathcal{N}_0\models\Psi$ as an $\mathcal{L}_\mathcal{M}$-structure. But also $\mathcal{N}_0\models\Phi(x_1,\hdots,x_n)$ since $x_1,\hdots,x_n$ realize $P$ in $\mathcal{N}_0$, so $\mathcal{N}_0\models\Phi(x_1,\hdots,x_n)\land\Psi$, and thus $\mathcal{N}_0$ satisfies $\Delta$.

Hence $\Gamma$ is finitely satisfiable, and there is a model $\mathcal{N}$ for $\Gamma$ of cardinality $|M|$ by Theorem~\ref{thm:cardinal_compactness}. Now $P$ is realized in $\mathcal{N}$, and moreover $\mathcal{N}\models Diag_{el}(\mathcal{M})$, so there is an elementary embedding of $\mathcal{M}$ into $\mathcal{N}$ by Lemma~\ref{lem:diag}. Identifying $\mathcal{M}$ with $j(\mathcal{M})$, we can view $\mathcal{N}$ as an elementary extension of $\mathcal{M}$.
\end{proof}

Now that we can add elements to realize types, we can construct saturated models for certain cardinalities.

\begin{thm}[{\cite[Theorem 4.3.12]{Marker}}]\label{thm:saturatedexistence}
Let $\mathcal{L}$ be a countable language, and let $\mathcal{M}$ be an infinite $\mathcal{L}$-structure with domain $M$. Let $\kappa$ be an infinite cardinal. Then there is an elementary extension $\mathcal{N}$ of $\mathcal{M}$ with domain $N$ such that $|N|\leq|M|^\kappa$, and $\mathcal{N}$ is $\kappa^+$-saturated. 
\end{thm}
\begin{proof}
First we will note that in order to prove that a model is $\kappa$-saturated, it suffices to prove that the model realizes all 1-types. The general case follows by induction: If $P$ is an $n$-type over $A$, let $Q = \{\phi(v_1,\hdots,v_{n-1})\,|\,\phi\in P\}$, the $(n-1)$-type consisting of all formulas in $P$ which do not include the last variable $v_n$. By induction, $Q$ is realized by some $a_1,\hdots,a_{n-1}\in M$. Let $R = \{\psi(c_{a_1},\hdots,c_{a_{n-1}},v_n)\,|\,\psi(v_1,\hdots,v_n)\in P\}$, which is a 1-type (in the free variable $v_n$) over $A\cup\{a_1,\hdots,a_{n-1}\}$. Since adding finitely many elements does not increase the cardinality of $A$, $R$ is realized in $\mathcal{M}$ by the base case. Suppose $b$ realizes $R$. Then $a_1,\hdots,a_{n-1},b$ realizes $P$.

Returning to the proof of the theorem, we will first apply Lemma~\ref{lem:realizing_types} repeatedly to build a chain of elementary extensions of $\mathcal{M}$, each of which satisfies a particular type.

Our claim is that there exists an elementary extension $\mathcal{M}'$ of $\mathcal{M}$ with $|M'|\leq|M|^\kappa$ such that for all $A\subseteq M$ with $|A|\leq\kappa$, every 1-type over $A$ is realized in $\mathcal{M}'$.

We need to pin down how many types we may need to satisfy. The number of subsets of $M$ with cardinality less than or equal to $\kappa$ is less than the number of functions, $\kappa \rightarrow M$, since each such subset is the range of one of these functions. This set of functions has cardinality $|M|^\kappa$. 

Now given a subset $A\subseteq M$ with $|A|\leq \kappa$, the language $\mathcal{L}_A$ has cardinality at most $\kappa$, since $\mathcal{L}$ is countable. The set of $\mathcal{L}_A$-formulas is a subset of the set of finite strings with symbols from $\mathcal{L}_A$ plus our finite set of boolean connectors, quantifiers, etc. The cardinality of the set of finite strings of any given length $l$ is the cardinality of the set of functions, $\{1,\hdots,l\}\rightarrow (\mathcal{L}_A\cup\{v_1,\land,\lor,\hdots\})$, and this set of functions has cardinality $\kappa^l = \kappa$. Now there are $\aleph_0$ values for $l$, so the set of $\mathcal{L}_A$-formulas has cardinality $\aleph_0 \kappa = \kappa$, since $\aleph_0\leq \kappa$. 

Now types are elements of the power set of the set of $\mathcal{L}_A$-formulas, so the cardinality of the set of $\mathcal{L}_A$-formulas is at most the cardinality of the power set, that is, $2^\kappa$.

Hence the total number of 1-types over all subsets $A$, with $|A|\leq\kappa$, is bounded above by $|M|^\kappa 2^\kappa = |M|^\kappa$, since $2<|M|$. 

Let $(P_\alpha\,:\,\alpha<|M|^\kappa)$ be an enumeration of all such types. We will build an elementary chain $(\mathcal{M}_\alpha\,:\,\alpha<|M|^\kappa)$ as follows:
\begin{enumerate}
\item Let $\mathcal{M}_0 = \mathcal{M}$. 
\item For $\alpha$ a limit ordinal, let $\mathcal{M}_\alpha = \bigcup_{\beta<\alpha}\mathcal{M}_\beta$. 
\item For all $\alpha$, apply Lemma~\ref{lem:realizing_types} to find an elementary extension $\mathcal{M}_{\alpha+1}$ of $\mathcal{M}_\alpha$ such that $|M_{\alpha+1}| = |M_\alpha|$ and $\mathcal{M}_{\alpha+1}$ realizes $P_{\alpha}$.
\end{enumerate}

Now let $\mathcal{M}'=\bigcup_{\alpha<|M|^\kappa}\mathcal{M}_\alpha$. Since every type $P_\alpha$ is realized in some $\mathcal{M}_\alpha$, $\mathcal{M}'$ realizes every such type. It remains to show that $|M'|\leq|M|^\kappa$. 

We will show by induction that for all $\alpha<|M|^\kappa$, $|M_\alpha|\leq|M|^\kappa$. In the base case, $|M_0| = |M| \leq |M|^\kappa$. If $|M_\alpha|\leq|M|^\kappa$, then $|M_{\alpha+1}| = |M_\alpha|\leq|M|^\kappa$. If $\alpha$ is a limit ordinal, then $|M_\alpha|$ is the union of a chain of sets of cardinality at most $|M|^\kappa$, so its cardinality is the limit of the cardinalities of these sets, which is bounded above by $|M|^\kappa$. So $|M_\alpha|\leq|M|^\kappa$. 

Now $\mathcal{M}'$ is the union of a chain of sets of cardinalities at most $|M|^\kappa$, so it has cardinality at most $|M|^\kappa$. This completes the proof of the claim.

We have that $\mathcal{M}'$ realizes every 1-type over every subset $A\subset M$ with $|A|\leq\kappa$, but we do not yet have the $\mathcal{M}'$ is $\kappa^+$-saturated, since its domain is larger than that of $\mathcal{M}$, and thus there are additional types to realize. 

We build another elementary chain $(\mathcal{N}_\alpha\,:\,\alpha<\kappa^+)$ as follows:
\begin{enumerate}
\item Let $\mathcal{N}_0 = \mathcal{M}$. 
\item For $\alpha$ a limit ordinal, let $\mathcal{N}_\alpha = \bigcup_{\beta<\alpha}\mathcal{N}_\beta$. 
\item For all $\alpha$, apply the previous claim to find an elementary extension $\mathcal{N}_{\alpha+1}$ of $\mathcal{N}_\alpha$ such that for all $A\subseteq N_\alpha$ (where $N_\alpha$ is the domain of $\mathcal{N}_\alpha$) with $|A|\leq\kappa$, every 1-type over $A$ is realized in $\mathcal{N}_{\alpha+1}$. By the claim, $|N_{\alpha+1}|\leq |N_\alpha|^\kappa$. 
\end{enumerate}

Now let $\mathcal{N}=\bigcup_{\alpha<\kappa^+}\mathcal{N}_\alpha$. Let $N$ be the domain of $\mathcal{N}$. For every $A\subseteq N$ such that $|A|<\kappa^+$, $A$ is contained in the domain of some $\mathcal{N}_\alpha$, and every 1-type over $A$ is realized in $\mathcal{N}_{\alpha+1}$, and therefore in $\mathcal{N}$. Thus $\mathcal{N}$ is $\kappa^+$-saturated. It remains to show that $|N|\leq|M|^\kappa$. 

We will show by induction that for all $\alpha<\kappa^+$, $|N_\alpha|\leq|M|^\kappa$. In the base case, $|N_0| = |M| \leq |M|^\kappa$. If $|{N}_\alpha|\leq|{M}|^\kappa$, then $|{N}_{\alpha+1}|\leq |{N}_\alpha|^\kappa\leq(|{M}|^\kappa)^\kappa = |{M}|^\kappa$. If $\alpha$ is a limit ordinal, then $|{N}_\alpha|$ is the union of a chain of sets of cardinality at most $|{M}|^\kappa$, so $|{N}_\alpha|\leq|{M}|^\kappa$. 

Now $\mathcal{N}$ is the union of a chain of sets of cardinalities at most $|{M}|^\kappa$ so it has cardinality at most $|{M}|^\kappa$. This completes the proof.
\end{proof}

\begin{cor}\label{cor:ch}
If we assume the Continuum Hypothesis, there is a saturated model of $Th(\mathcal{M})$ with cardinality $\aleph_1$. 
\end{cor}
\begin{proof}
By the L\"owenheim-Skolem theorem (Theorem~\ref{thm:L-S}), there is a model $\mathcal{M}'\models Th(\mathcal{M})$ of cardinality $\aleph_0$. Applying Theorem~\ref{thm:saturatedexistence}, there is a $\aleph_1$-saturated elementary extension $\mathcal{N}$ of $\mathcal{M}'$ with domain $N$ of cardinality at most $\aleph_0^{\aleph_0}$. If we assume the Continuum Hypothesis, $\aleph_0^{\aleph_0} = \aleph_1$. Since an $\aleph_1$-saturated model must have cardinality at least $\aleph_1$, $|N| = \aleph_1$, and hence $\mathcal{N}$ is a saturated model for $Th(\mathcal{M})$.
\end{proof}

The existence of saturated models makes many results in model theory easier to prove, including the Ax-Kochen Principle. There are methods to eliminate the Continuum Hypothesis from some proofs which use saturated models, one of which is to employ a generalization of saturated models, called special models. For more information, see Appendix~\ref{sec:SM}. We will assume the existence of saturated models in order to simplify our arguments.

\newpage

\section{The Ax-Kochen Principle}\label{sec:AKP}
\subsection{Hensel's Lemma}\label{subsec:HL}

One of our key tools in establishing the Ax-Kochen Principle will be Hensel's lemma. In valued fields in which Hensel's lemma holds, one can lift information about polynomials over the residue class field to polynomials over the valuation ring. 

\begin{defin}
Let $F$ be a valued field.  We say that $F$ is \emph{Henselian} if $F$ has the following property, which is one of the formulations of Hensel's lemma:

Let $f,g_0,h_0\in \mathcal{O}[x]$ be monic polynomials. If the images of $g_0$ and $h_0$ in $\overline{F}[x]$, $\overline{g_0}$ and $\overline{h_0}$, are relatively prime, and if $\overline{g_0}\overline{h_0} = \overline{f}$, then there exist $g,h\in\mathcal{O}[x]$ such that $\overline{g} = \overline{g_0}$, $\overline{h} = \overline{h_0}$, and $f = gh$. 
\end{defin}

In this section, we will show that all complete discrete valued fields, and in particular the fields $\mathbb{Q}_p$ and $\mathbb{F}_p((t))$, are Henselian. 

Throughout this section, we will assume that all valued fields have cross section, and we will write our value groups multiplicatively. This is a change from the notation in Section ~\ref{subsec:VF}. In particular, we will take as the value group of a discrete valued field the group $\{\pi^n\,|\,n\in\mathbb{Z}\}$, where $\pi$ is a prime element. We must specify what is meant by homomorphism of valued fields and valued subfield in this context.

\begin{defin}\label{def:VF-hom}
Let $F$ and $F'$ be valued fields (with cross section), and let $\mathfrak{v}_F:F\rightarrow F$ and $\mathfrak{v}_{F'}:F'\rightarrow F'$ be their valuations. We say that a function $f:F\rightarrow F'$ is a \emph{homomorphism of valued fields} if it is a field homomorphism which preserves valuations, that is, if $f\circ \mathfrak{v}_F = \mathfrak{v}_{F'}\circ f$. We say that a subfield $K\subseteq F$ is a \emph{valued subfield} if $\mathfrak{v}_F(K)\subseteq K$. In this case, $K$ is a valued field with cross section whose valuation is the restriction of $\mathfrak{v}_F$ to $K$.
\end{defin}

We will assume familiarity with the resultant, an algebraic construction which gives information about the common roots of polynomials. The necessary facts about the resultant are developed in Appendix~\ref{sec:TR}. We will use the following results:

\begin{resprime}
Let $R$ be a ring. If $f,g\in R[x]$ are relatively prime, $Res(f,g)\neq 0$.
\end{resprime}

\begin{henselhelper}
Let $R$ be a subring of a field $K$, and let $g,h\in R[x]$ with $deg(g) = m$, $deg(h) = n$. If $\rho = Res(g,h)\neq 0$, then for all $l\in R[x]$ such that $deg(l)\leq m+n-1$, there exist $\phi,\psi\in R[x]$ with $deg(\phi)\leq n-1$, $deg(\psi)\leq m-1$ such that $g\phi + h\psi = \rho l$. 
\end{henselhelper}

Most of the work of proving Hensel's lemma in complete discrete valued fields is done in the following more general lemma.

\begin{lem}[{\cite[Theorem 4.3.1]{Borevich}}]\label{lem:henselhelper2}
Let $F$ be a complete discrete valued field with valuation $\mathfrak{v}$ and prime element $\pi$. Let $f\in \mathcal{O}[x]$ be a polynomial of degree $m+n$. Suppose there are polynomials $g_0,h_0\in \mathcal{O}[x]$ of degrees $m$ and $n$ respectively such that
\begin{enumerate}
\item $f$ and $g_0h_0$ have the same leading coefficient,
\item $Res(g_0,h_0)\neq 0$, and
\item letting $\pi^r=\mathfrak{v}(Res(g_0,h_0))$, we have $f\equiv g_0h_0$ (mod $\pi^{2r+1}$). 
\end{enumerate}
Then there exist polynomials $g,h\in\mathcal{O}[x]$ of degrees $m$ and $n$ respectively such that
\begin{enumerate}
\item $f=gh$,
\item $g\equiv g_0$ and $h\equiv h_0$ (mod $\pi^{r+1}$), and
\item $g$ and $g_0$ have the same leading coefficient, as do $h$ and $h_0$.
\end{enumerate}
\end{lem}
\begin{proof}
For all $k\geq 0$, we will define polynomials $g_k$ and $h_k$ so that $deg(g_k) = deg(g_0)$, $deg(h_k) = deg(h_0)$, and $f\equiv g_kh_k$ (mod $\pi^{2r+k+1}$). 

We will accomplish this by defining $\phi_k,\psi_k\in \mathcal{O}[x]$ for all $k\geq 1$ and letting $g_k = g_0 + \pi^{r+1}\phi_1 + \hdots + \pi^{r+k}\phi_k$ and $h_k = h_0 + \pi^{r+1}\psi_1 + \hdots + \pi^{r+k}\psi_k$. If we require that $deg(\phi_k)\leq m-1$ and $deg(\psi_k)\leq n-1$, then we will have $deg(g_k) = deg(g_0)$ and $deg(h_k) = deg(h_0)$.

In the base case, we simply have $f\equiv g_0h_0$ (mod $\pi^{2r+1}$) by assumption. 

Suppose we have $f\equiv g_{k-1}h_{k-1}$ (mod $\pi^{2r + k}$). Then $f = g_{k-1}h_{k-1} + \pi^{2r+k}l$, for some $l\in \mathcal{O}[x]$. 

By inductive assumption, $g_0$ and $g_{k-1}$ have the same leading coefficient, as do $h_0$ and $h_{k-1}$. The leading coefficient of $f$ is the same as the product of the leading coefficients of $g_0$ and $h_0$ by assumption, so it is the same as the product of the leading coefficients of $g_{k-1}$ and $h_{k-1}$. Thus $l$ must have degree less than $m+n$. 

Also, $g_{k-1}\equiv g_0$, $h_{k-1}\equiv h_0$ (mod $\pi^{r+1}$) by construction, so $Res(g_{k-1},h_{k-1})\equiv Res(g_0,h_0)$ (mod $\pi^{r+1}$). But $\mathfrak{v}(Res(g_0,h_0)) = \pi^r$ by assumption, so $\mathfrak{v}(Res(g_{k-1},h_{k-1})) = \pi^r$, and $Res(g_{k-1},h_{k-1}) = \pi^ru$ for some unit $u$. 

Applying Lemma~\ref{lem:henselhelper}, there exist polynomials $\phi_k,\psi_k\in \mathcal{O}[x]$ such that $g_{k-1}\psi_k + h_{k-1}\phi_k = (\pi^ru)(u^{-1}l) = \pi^rl$ with $deg(\phi_k)\leq m-1$ and $deg(\psi_k)\leq n-1$.

Now we define $g_k = g_0 + \pi^{r+1}\phi_1 + \hdots + \pi^{r+k}\phi_k$ and $h_k = h_0 + \pi^{r+1}\psi_1 + \hdots + \pi^{r+k}\psi_k$. We need to show that $f\equiv g_kh_k$ (mod $\pi^{2r+k+1}$). 

Expanding, 
\begin{eqnarray*}
g_kh_k &=& (g_{k-1} + \pi^{r+k}\phi_k)(h_{k-1} + \pi^{r+k}\psi_k)\\
&=& g_{k-1}h_{k-1} + \pi^{r+k}(g_{k-1}\psi_k + h_{k-1}\phi_k) + \pi^{2r+2k}\phi_k\psi_k\\
&=& (f - \pi^{2r+k}l) + \pi^{2r+k}l + \pi^{2r+2k}\phi_k\psi_k\\
&=& f + \pi^{2r+2k}\phi_k\psi_k.
\end{eqnarray*}

So $f\equiv g_kh_k$ (mod $\pi^{2r+k+1}$), as was to be shown, since $2k\geq k+1$ when $k\geq 1$. 

Having established the claim by induction, we let $g = g_0 + \pi^{r+1}\phi_1 + \pi^{r+2}\phi_2 + \hdots$, which is an element of $\mathcal{O}[x]$, since $F$ is complete. Similarly, we let $h = h_0 + \pi^{r+1}\psi_1 + \pi^{r+2}\psi_2 + \hdots\in\mathcal{O}[x]$. 

Since the degrees of the $\phi_i$ are all less than $m$ and the degrees of the $\psi_i$ are all less than $n$, the leading coefficient of $g$ is the same as that of $g_0$, and the leading coefficient of $h$ is the same as that of $h_0$. We also clearly have $g\equiv g_0$ and $h\equiv h_0$ (mod $\pi^{r+1}$).

Finally, $f\equiv g_kh_k$ (mod $\pi^{2r+k+1}$) for all $k\geq 1$, and $g_kh_k\equiv gh$ (mod $\pi^{r+k+1}$) for all $k\geq 1$, so $f\equiv gh$ (mod $\pi^{r+k+1}$) for all $k\geq 1$. Letting $k$ to go infinity, we have $f = gh$ in $\mathcal{O}$.
\end{proof}

\begin{thm}\label{thm:DVFhensel}
All complete discrete valued fields are Henselian.
\end{thm}
\begin{proof}

Let $F$ be a complete discrete valued field, and let $f,g_0,h_0\in \mathcal{O}[x]$ be monic polynomials, such that the images of $g_0$ and $h_0$ in $\overline{F}[x]$, $\overline{g_0}$ and $\overline{h_0}$, are relatively prime, and $\overline{g_0}\overline{h_0} = \overline{f}$. We would like to show that there exist $g,h\in\mathcal{O}[x]$ such that $\overline{g} = \overline{g_0}$, $\overline{h} = \overline{h_0}$, and $gh = f$. 

Let $m = deg(g_0)$ and $n = deg(h_0)$. Since $f$, $g_0$, and $h_0$ are monic, $\overline{f}$, $\overline{g_0}$, and $\overline{h_0}$ are also monic. Then $deg(\overline{g_0}) = m$, $deg(\overline{h_0}) = n$, and since $\overline{f} = \overline{g_0}\overline{h_0}$, $deg(f) = m + n$. Also, $f$ and $g_0h_0$ have the same leading coefficient, $1$. This satisfies condition 1 of Lemma~\ref{lem:henselhelper2}.

Let $\rho = Res(g_0,h_0)$. Since $\rho$ is computed from the coefficients of $g_0$ and $h_0$ by addition and multiplication, we can mod out by $\pi$, and $\overline{\rho} = Res(\overline{g_0},\overline{h_0})$. Now $\overline{g_0}$ and $\overline{h_0}$ are relatively prime in $\overline{F}$, so $Res(\overline{g_0},\overline{h_0}) \neq 0$. Thus $\rho\not\equiv 0$ (mod $\pi$), and in particular $\rho\neq 0$, satisfying condition 2. Hence $\mathfrak{v}(\rho) = \pi^0$. Finally, $\overline{f} = \overline{g_0}\overline{h_0}$, so $f\equiv g_0h_0$ (mod $\pi$), satisfying condition 3 of Lemma~\ref{lem:henselhelper2} with $r = 0$. 

The lemma gives us polynomials $g,h\in\mathcal{O}[x]$ such that $g \equiv g_0$ (mod $\pi$), $h \equiv h_0$ (mod $\pi$), and $f = gh$, as required by Hensel's lemma.
\end{proof}

\subsubsection*{Consequences of Hensel's Lemma}
Henselian valued fields have a number of properties which will be useful in the proof of the Ax-Kochen Principal. The most important is that the residue class field of a Henselian valued field can be embedded as a subfield in the case of characteristic zero.

\begin{lem}[{\cite[Lemma 5.4.13 (ii)]{Chang}}]\label{lem:residue_subfield}
Let $F$ be a Henselian valued field with valuation $\mathfrak{v}$ such that $char(\overline{F}) = 0$. Then there exists a subfield $F_0\subseteq \mathcal{O}$ such that $F_0\cong \overline{F}$, where the isomorphism is given by $\phi(x) = \overline{x}$.  
\end{lem}
\begin{proof}
Since $char(\overline{F}) = 0$, then we must also have $char(F) = 0$, for if $1 + \hdots + 1 = 0$, then $\overline{1} + \hdots + \overline{1} = \overline{0}$. Thus there is a natural embedding of the rationals into $F$. Our first step will be to show that rationals in $F$ are contained in $\mathcal{O}$. 

Since $\mathcal{O}$ is a ring, all integers in $F$ are elements of $\mathcal{O}$. Recall that $\mathfrak{v}(1) = 1$, since $\mathfrak{v}$ is a multiplicative group homomorphism. Since $char(\overline{F}) = 0$, if $n$ is a positive integer, $\overline{n} = \overline{1} + \hdots + \overline{1} \neq 0$, so $n$ is not in the maximal ideal $I_1$, and we do not have $\mathfrak{v}(n)>1$. But $n\in\mathcal{O}$, so $\mathfrak{v}(n) = 1$. By Lemma~\ref{lem:VF}, $\mathfrak{v}(-n) = \mathfrak{v}(n) = 1$, so all integers have valuation $0$. 

Now if $m/n$ is a rational in $F$, $\mathfrak{v}(m/n) = \mathfrak{v}(m) \mathfrak{v}(n)^{-1} = 1$. So the rationals are a subfield of the valuation ring $\mathcal{O}$.

By a simple application of Zorn's lemma, the rationals are contained in a maximal subfield $F_0$ of $\mathcal{O}$. More explicitly, the union of a chain of fields is a field, so every chain of subfields of $\mathcal{O}$ has an upper bound, and by Zorn's lemma, the set of subfields of $\mathcal{O}$ has maximal elements.

Since only elements with valuation $1$ have inverses in $\mathcal{O}$, $\mathfrak{v}(x) = 1$ for all nonzero $x\in F_0$. Thus $0$ is the only element of $F_0$ with valuation greater than $1$, and the kernel of the residue map $\phi$ is trivial, so $\phi$ maps $F_0$ isomorphically onto a subfield $G_0$ of $\overline{F}$. We will use Hensel's lemma to show that every element of $\overline{F}$ must be in $G_0$, proving their equality.

Suppose that $a\in\overline{F}$ and $a$ is algebraic over $G_0$. Then there is a monic irreducible polynomial $p_0(x)\in G_0[x]$ such that $p_0(a) = 0$. Choosing preimages for all coefficients of $p_0$ under $\phi$, we obtain a polynomial $p\in F_0[x]$ such that $\overline{p} = p_0$. Now $\overline{p}$ factors in $\overline{F}[x]$ as $\overline{p}(x) = q_0(x)(x-a)$, where $q_0(x)$ and $(x-a)$ are relatively prime since $char(\overline{F}) = 0$. Applying Hensel's lemma, there are polynomials $q,r\in \mathcal{O}[x]$ such that $\overline{q} = q_0$, $\overline{r} = x-a$, and $p = qr$. 

Let $c$ and $d$ be the leading coefficients of $q$ and $r$ respectively. Since $p$ is monic, $cd = 1$, so $\mathfrak{v}(c) \mathfrak{v}(d) = 1$, but $c,d\in\mathcal{O}$, so $\mathfrak{v}(c) = \mathfrak{v}(d) = 1$. Thus neither have residue $0$, and we have $deg(q) = deg(q_0)$ and $deg(r) = deg(x-a) = 1$. 

Let $r = b_1x + b_0$. Then $y = -b_0/b_1$ is a root of $r$ and therefore a root of $p$. Now $F_0$ is isomorphic to $G_0$, and $\overline{p}$ is irreducible in $G_0$, so $p$ is irreducible in $F_0$, and hence $y\notin F_0$. 

We showed that $b_1$ has valuation $1$. Also $b_0\in\mathcal{O}$, so $\mathfrak{v}(b_0)\geq 1$, and so $\mathfrak{v}(y) = \mathfrak{v}(b_0) \mathfrak{v}(b_1)^{-1} \geq 1$, and $y\in\mathcal{O}$. Since all the generators of $F_0[y]$ are in the ring $\mathcal{O}$,  $F_0[y]\subseteq \mathcal{O}$. This contradicts the maximality of $F_0$ as a subfield of $\mathcal{O}$. Hence there are no elements of $\overline{F}$ algebraic over $G_0$. 

Now suppose that $a\in\overline{F}$ and $a$ is transcendental over $G_0$. We will use the same contradiction strategy. Pick some $y\in\mathcal{O}$ such that $\overline{y} = a$. Now for any nonzero $p(x)\in F_0[x]$, $\overline{p(y)} = \overline{p}(a) \neq 0$, since $a$ is transcendental over $G_0$. In particular, $p(y)\neq 0$, so $y$ is transcendental over $F_0$. Also, $p(y)$ does not have residue $0$, but $p(y)\in\mathcal{O}$, so $\mathfrak{v}(p(y)) = 1$. Thus for any $q,r\in F_0[x]$, $\mathfrak{v}(q(y)/r(y)) = 1$, so $q(y)/r(y)\in\mathcal{O}$. All elements of the transcendental extension have this form, so $F_0(y)$ is contained in $\mathcal{O}$, once again contradicting the maximality of $F_0$.

Since there are no elements of $\overline{F}$ algebraic or transcendental over $G_0$, $G_0 = \overline{F}$, and thus $\phi$ is an isomorphism onto $\overline{F}$, and $F_0\cong \overline{F}$. 
\end{proof}

We will now introduce the Henselization of a valued field, which is similar in concept to algebraic closure. Intuitively, the Henselization of a valued field is the minimal Henselian valued field containing it. The Henselization is defined by means of a universal property.

\begin{defin}\label{def:henselization}
Let $G$ be a valued field. A Henselian valued field $K$ is said to be a \emph{Henselization} of $G$ if
\begin{enumerate}
\item the field $G$ is a valued subfield of $K$, and
\item if $F$ is a Henselian valued field and $\mu:G\rightarrow F$ is an embedding of valued fields, then $\mu$ extends uniquely to an embedding $\lambda:K\rightarrow F$ of valued fields.
\end{enumerate}
\end{defin}

The universal property guarantees that the Henselization is unique up to unique isomorphism. For suppose that $K$ and $K'$ are both Henselizations of the valued field $G$. Then $G$ certainly embeds into both $K$ and $K'$, so there exist unique embeddings of valued fields $\lambda:K\rightarrow K'$ and $\lambda':K'\rightarrow K$. Then $\lambda'\circ\lambda:K\rightarrow K$ is an embedding of valued fields. But by the universal property, there is a unique embedding of valued fields $K\rightarrow K$, and this must be the identity. So $\lambda'\circ\lambda = id_K$, and $\lambda$ is an isomorphism.

The Henselization of any valued field $F$ may be constructed as follows. Let $F^s$ be the separable closure of $F$ in its algebraic closure (in characteristic $0$, this is just the algebraic closure). Extend the valuation $\mathfrak{v}$ to a valuation $\mathfrak{v}^s$ on $F^s$, and let $L$ be the subgroup of the Galois group of $F^s$ over $F$ consisting of all automorphisms which are also valued field homomorphisms, that is, they preserve $\mathfrak{v}^s$. The fixed field of $L$ is the Henselization of $F$. In the special case that $F$ has a completion, $\widehat{F}$, this construction corresponds to the separable closure (or relative algebraic closure in the case of characteristic $0$) of $F$ in $\widehat{F}$. We have the following lemma, which we will state here without proof.

\begin{lem}[{\cite[Theorem 5.2]{Ribenboim}}]\label{lem:henselization_exists}
Every valued field $F$ (with valuation $\mathfrak{v}_F$) has a Henselization $K$ (with valuation $\mathfrak{v}_K$), which is unique up to unique isomorphism of valued fields. The value groups and residue class fields of $K$ and $F$ are isomorphic as groups and fields respectively.
\end{lem}

The next lemmas consist of a number of facts about Henselizations and extensions of Henselian fields which will be necessary for the back-and-forth argument in Section~\ref{subsec:EEE}. The proofs of these facts are technical, and they rely on too much of the theory of valued fields to develop in this thesis. They can be found in full in Ribenboim \cite{Ribenboim} and Chang and Keisler \cite{Chang}. 

\begin{lem}[{\cite[Theorem 5.1]{Ribenboim}}]\label{lem:unique_extension}
The valuation on a Henselian valued field extends in a unique way to an algebraic extension. Equivalently, let $F_0$ and $G_0$ be Henselian valued fields with valuations $\mathfrak{v}_{F_0}$ and $\mathfrak{v}_{G_0}$ which are isomorphic as valued fields. Let $F$ and $G$ be algebraic extensions of $F_0$ and $G_0$ respectively which are isomorphic as fields, and let $\phi$ be the isomorphism. Given extensions $\mathfrak{v}_F$ and $\mathfrak{v}_G$, of $\mathfrak{v}_{F_0}$ and $\mathfrak{v}_{G_0}$ which are valuations on $F$ and $G$ respectively, $\phi$ is an isomorphism of valued fields.
\end{lem}

\begin{lem}[{\cite[Lemma 5.4.13 (vi)]{Chang}}]\label{lem:closure_under_roots}
Let $F$ be a valued field with valuation $\mathfrak{v}$, $F_0$ a valued subfield, and $\widetilde{F_0}$ the relative algebraic closure of $F_0$ in $F$, that is, the set of all elements in $F$ algebraic over $F_0$. Then $\mathfrak{v}(\widetilde{F_0}^*) = \{x\in\mathfrak{v}(F^*)\,|\,x^n\in\mathfrak{v}(F_0^*), n\in\mathbb{Z}\}$, the closure under roots of $\mathfrak{v}(F_0^*)$ in $\mathfrak{v}(F^*)$. If $F$ is a Henselian valued field such that $\overline{F} = \overline{F_0}$, $char(\overline{F}) = 0$, and $\mathfrak{v}(\widetilde{F_0^*}) = \mathfrak{v}(F_0^*)$ (that is, $\mathfrak{v}(F_0^*)$ is already closed under roots), then $\widetilde{F_0}$ is a Henselization of $F_0$. 
\end{lem}

\begin{lem}[{\cite[Lemma 5.4.13 (vii)]{Chang}}]\label{lem:transcendental_extension}
Let $F$ and $G$ be Henselian valued fields (with valuations $\mathfrak{v}_F$ and $\mathfrak{v}_G$) with Henselian valued subfields $F_0$ and $G_0$ respectively such that $F_0\cong G_0$ by an isomorphism $f$. Suppose $x\in F$ is transcendental over $F_0$ and $y\in G$ is transcendental over $G_0$. Suppose further that $\mathfrak{v}(F_0(x)^*) = \mathfrak{v}(F_0^*)$, $\overline{F_0(x)} = \overline{F_0}$, and for all $a\in F_0$, $f(\mathfrak{v}_F(x-a)) = \mathfrak{v}_G(y-f(a))$. Then $\mathfrak{v}_G(G_0(y)^*) = \mathfrak{v}_G(G_0^*)$, $\overline{G_0(y)} = \overline{G_0}$, and $f$ can be extended to an isomorphism $F_0(x)\cong G_0(y)$.
\end{lem}

\begin{lem}[{\cite[Lemma 5.4.13 (viii)]{Chang}}]\label{lem:transcendental_value_group}
Let $F$ be a Henselian valued field with a Henselian valued subfield $F_0$, and suppose $x\in F$ is transcendental over $F_0$. If $\overline{F_0(x)} \cong \overline{F_0}$ and $\mathfrak{v}(F_0^*)$ is nontrivial, then adjoining $x$ does not increase the cardinality of the value group: $|\mathfrak{v}(F_0(x)^*)| = |\mathfrak{v}(F_0^*)|$. 
\end{lem}

\subsection{Establishing Elementary Equivalence}\label{subsec:EEE}

We are now in a position to prove the Ax-Kochen Principle. We saw in Theorem~\ref{thm:template} that it suffices to prove the elementary equivalence of the ultraproducts $\prod\mathbb{Q}_p/\mathcal{D}$ and $\prod\mathbb{F}_p((t))/\mathcal{D}$ for all nonprincipal ultrafilters $\mathcal{D}$ on the set of all primes $P$. 

In Example~\ref{exa:UP_char0}, we applied Corollary~\ref{cor:el_class} to show that for any nonprincipal ultrafilter $\mathcal{D}$, $\prod\mathbb{Q}_p/\mathcal{D}$ and $\prod\mathbb{F}_p((t))/\mathcal{D}$ are valued fields with cross section. We would like to show that these valued fields are Henselian.

\begin{lem}\label{lem:HVF}
The class of Henselian valued fields with cross section is elementary.
\end{lem}
\begin{proof}
We have already exhibited a set of first-order axioms in the language $\mathcal{L}_{VF}$ for the class of valued fields with cross section (see Example~\ref{exa:VFel}), so it remains to extend these axioms to include Hensel's lemma.

In order to express Hensel's lemma, we must be able to make statements about polynomials whose coefficients lie in the domain of our structure. We will represent a polynomial of degree at most $n$ by an $(n+1)$-tuple of coefficients and make appropriate statements about these coefficients.

We first note we can quantify over polynomials by using one quantifier for each coefficient. For a polynomial of degree at most $n$ we will write $\exists_n f(x)$ and $\forall_n f(x)$ as abbreviations for $\exists a_0\hdots\exists a_n$ and $\forall a_0\hdots\forall a_n$. The $(n+1)$-tuple $(a_0,\hdots,a_n)$ will then represent $f(x)$. Here we are using the variables $a_i$ for clarity. Formally, they are choices of $v_j$ from our infinite set of variables $\mathcal{V} = \{v_1,v_2,\hdots\}$.

For the following abbreviations, suppose $f(x)$ is represented by $(a_0,\hdots,a_n)$, and $g(x)$ is represented by $(b_0,\hdots,b_m)$, with $m\geq n$.

We can express the statement that a polynomial has coefficients in the valuation ring $\mathcal{O}$ by requiring that all of its coefficients have valuation at least $1$. We will write $f\in\mathcal{O}[x]$ as an abbreviation for 
\[
(1\leq\mathfrak{v}(a_0))\land\hdots\land(1\leq\mathfrak{v}(a_n).
\]

We can easily express the statement that a polynomial is monic. We will write $Monic(f)$ as an abbreviation for 
\[
a_n = 1.
\]

We can express equality of polynomials by stating the equality of the coefficients. We will write $f = g$ as an abbreviation for 
\[
(a_0 = b_0)\land\hdots\land(a_n = b_n)\land(0=b_{n+1})\land\hdots\land(0=b_m).
\]

We can form new polynomials by addition, subtraction, and multiplication. We will write $f+g$ to mean the polynomial which is represented by $(a_0+b_0,\hdots,a_n+b_n,b_{n+1},\hdots,b_m)$, and subtraction is just addition with an application of the additive inverse function, $-$, to each coefficient of $g$. We will write $fg$ to mean the polynomial which is represented by the $(nm+1)$-tuple $(a_0b_0,a_0b_1+a_1b_0,\hdots,a_nb_m)$.

We will also need to work with the images of polynomials in the residue class field. To express equality of the images of two polynomials in the residue class field, we will state that their difference is a polynomial whose coefficients all have valuation greater than $1$ (and thus is the zero polynomial in the residue class field). Using the standard abbreviation $(x<y)$ for $(x\leq y)\land\lnot(x=y)$, we will write $ResEq(f,g)$ as an abbreviation for 
\[
(1<\mathfrak{v}(a_0 + -(b_0))\land\hdots\land(1<\mathfrak{v}(a_n + -(b_n)))\land(1<\mathfrak{v}(-(b_{n+1})))\land\hdots\land(1<\mathfrak{v}(-(b_m)))).
\]

Finally, Hensel's lemma includes the statement that the images of two polynomials in the residue class field are relatively prime. We will write $ResRelPrime(f,g)$ as an abbreviation for 
\[
\lnot(\exists_m p(x)\exists_m q(x)\exists_m r(x)\,(p\in\mathcal{O}[x])\land(q\in\mathcal{O}[x])\land(r\in\mathcal{O}[x])\land ResEq(pq,f) \land ResEq(pr,g)).
\] 

Using these abbreviations, the following first-order sentence, $Hensel_n$, expresses Hensel's lemma for polynomials of degree at most $n$.

\begin{eqnarray*}
\forall_n f(x) \forall_n g_0(x) \forall_n h_0(x)&&((f\in\mathcal{O}[x])\land (g_0\in\mathcal{O}[x])\land (h_0\in\mathcal{O}[x])\\
&&\land Monic(f)\land Monic(g_0)\land Monic(h_0)\\
&&\land ResRelPrime(g_0,h_0) \land ResEq(g_0h_0,f)) \rightarrow \\
(\exists_n g(x) \exists_n h(x) &&(g\in\mathcal{O}[x])\land (h\in\mathcal{O}[x])\\
&&\land ResEq(g,g_0)\land ResEq(h,h_0)\land (f=gh))
\end{eqnarray*}

Appending the infinite set of sentences $\{Hensel_n\,|\,n\in\mathbb{N}\}$ to the axioms for the theory of valued fields with cross section, we obtain a set of axioms for the theory of Henselian valued fields with cross section.
\end{proof}

Together with Corollary~\ref{cor:el_class}, Lemma~\ref{lem:HVF} shows that for any ultrafilter $\mathcal{D}$, $\prod\mathbb{Q}_p/\mathcal{D}$ and $\prod\mathbb{F}_p((t))/\mathcal{D}$ are Henselian. In the next theorem, we will work with the value groups of these valued fields as $\mathcal{L}_G$-structures, where $\mathcal{L}_G = \{\cdot,1\}$ is the language of groups, with symbols interpreted in the natural way. Similarly, we will work with the residue class fields as $\mathcal{L}_F$-structures, where $\mathcal{L}_F = \{+,\cdot,-,0,1\}$ is the language of fields, with symbols interpreted in the natural way. 

In Lemma~\ref{lem:UP_resval}, we showed that the value groups and residue class fields of the ultraproducts are isomorphic. The following general theorem shows that this is enough to prove that the ultraproducts themselves are elementarily equivalent. The proof, which is quite lengthy, uses a back-and-forth argument, properties of valuations, types and saturation, and the lemmas at the end of Section~\ref{subsec:HL}.

\begin{thm}[{\cite[Theorem 5.4.12]{Chang}}]\label{thm:el_eq}
Suppose that $F$ and $G$ are Henselian valued fields with valuations $\mathfrak{v}_F$ and $\mathfrak{v}_G$ respectively such that $\mathfrak{v}_F(F^*)\equiv \mathfrak{v}_G(G^*)$ (as $\mathcal{L}_G$-structures), $\overline{F} \equiv \overline{G}$ (as $\mathcal{L}_F$-structures), and $char(\overline{F}) = char(\overline{G}) = 0$. Then $F\equiv G$. 
\end{thm}
\begin{proof}
Suppose $F_1$ and $G_1$ are valued subfields of $F$ and $G$. We will write $f_1:F_1\leftrightarrow G_1$ if and only if $f_1$ is an isomorphism between $F_1$ and $G_1$, and $f_1$ restricted to the value group of $F_1$ is a partial elementary bijection between $\mathfrak{v}_F(F_1^*)$ and $\mathfrak{v}_G(G_1^*)$ as subsets of $\mathfrak{v}_F(F^*)$ and $\mathfrak{v}_G(G^*)$. 

The plan for the proof is as follows:
\begin{enumerate}
\item We will show that we can reduce to the case in which $F$ and $G$ are saturated models of cardinality $\aleph_1$.
\item We will show that the residue class fields $\overline{F}$ and $\overline{G}$ are relatively algebraically closed subfields of $F$ and $G$ respectively, and there exists $f_0:\overline{F}\leftrightarrow\overline{G}$.
\item We will show that given $f_1:F_1\leftrightarrow G_1$ between relatively algebraically closed valued subfields of $F$ and $G$, such that $\overline{F}\subseteq F_1$, $\overline{G}\subseteq G_1$, $f_0\subseteq f_1$, and $\mathfrak{v}_F(F_1^*)$ and $\mathfrak{v}_G(G_1^*)$ are countable, then given any element $x\in F$, we can extend $f_1$ to $f_2:F_2\leftrightarrow G_2$ such that $x\in F_2$, $F_1\subseteq F_2$, $G_1\subseteq G_2$, $f_1\subseteq f_2$, and $\mathfrak{v}_F(F_2)$ and $\mathfrak{v}_G(G_2)$ are countable.
\item We will show that the same extension result holds if we exchange the roles of $F$ and $G$. 
\item We will use these results and a back-and-forth argument to construct an isomorphism between $F$ and $G$. 
\end{enumerate}

(1) \emph{Reducing to the saturated case.} Assuming the Continuum Hypothesis, there exist saturated models $\mathcal{F}\models Th(F)$ and $\mathcal{G}\models Th(G)$ of cardinality $\aleph_1$, by Theorem~\ref{cor:ch}. Once again, we stress that the Continuum Hypothesis merely simplifies our arguments, and there are methods for eliminating it from the proof (see Appendix~\ref{sec:SM}). 

We would like to show that $\mathcal{F}$ and $\mathcal{G}$ satisfy the conditions of the theorem. Since the class of Henselian valued fields is elementary, the axioms for the class are a subset of both $Th(F)$ and $Th(G)$, so $\mathcal{F}$ and $\mathcal{G}$ are Henselian valued fields. 

For any $\mathcal{L}_F$-formula $\phi$ about the residue class field, we can transform $\phi$ into an $\mathcal{L}_{VF}$-formula $\phi'$ about the valued field as follows. First, we will restrict all variables to the valuation ring by replacing every instance of a quantifier $\exists v_i\, \psi$ or $\forall v_i\, \psi$ (where $\psi$ is some formula) with $\exists v_i\,(1\leq\mathfrak{v}(v_i))\land\psi$ or $\forall v_i\,(1\leq\mathfrak{v}(v_i))\rightarrow \psi$. Additionally, if $v_1,\hdots,v_j$ are free variables in $\phi$, we restrict these to the valuation ring as well by adding to $\phi$: $(1\leq\mathfrak{v}(v_1))\land\hdots\land(1\leq\mathfrak{v}(v_j)\land\phi$. Now we will replace equality by congruence modulo the maximal ideal $I_1$ by replacing every instance of $t_1 = t_2$ (where $t_1$ and $t_2$ are terms) with $1 < \mathfrak{v}(t_1 + -(t_2))$. 

It should be easy to convince yourself that for every $\mathcal{L}_{F}$-sentence $\phi$, $\overline{\mathcal{F}}\models \phi$ if and only if $\mathcal{F}\models \phi'$. But $\mathcal{F}\equiv F$, so $\mathcal{F}\models\phi'$ if and only if $F\models\phi'$ if and only if $\overline{F}\models\phi$. By the same argument, $\overline{\mathcal{G}}\models \phi$ if and only if $\overline{G}\models\phi$. But $\overline{F}\equiv\overline{G}$, so $\overline{\mathcal{F}}\models\phi$  if and only if $\overline{\mathcal{G}}\models\phi$, and thus $\overline{\mathcal{F}}\equiv\overline{\mathcal{G}}$.

In particular, the $\mathcal{L}_F$-sentences $\lnot Char_p$ for each prime $p$ in the theory $Char_0$ can each be transformed by this method. Call the theory made up of these transformed sentences $Char_0'$. Since the residue class field of $F$ has characteristic zero, $\overline{F}\models Char_0$, so $F\models Char_0'$, and since $F\equiv \mathcal{F}$, $\mathcal{F}\models Char_0'$, and thus $\overline{\mathcal{F}}\models Char_0$. Since $\overline{\mathcal{G}}\equiv\overline{\mathcal{F}}$, both residue class fields have characteristic zero.

Similarly, for any $\mathcal{L}_G$-formula $\phi$ about the value group, we can transform $\phi$ into an $\mathcal{L}_{VF}$-formula $\phi'$ about the valued field by restricting all variables to the value group. We replace every instance of a quantifier $\exists v_i\, \psi$ or $\forall v_i\, \psi$ (where $\psi$ is some formula) with $\exists v_i\,V(v_i)\land\psi$ or $\forall v_i\,V(v_i)\rightarrow \psi$. Additionally, if $v_1,\hdots,v_j$ are free variables in $\phi$, we restrict these to the value group as well by adding to $\phi$: $V(v_1)\land\hdots\land V(v_j)\land\phi$. Again, $\mathfrak{v}_F(\mathcal{F}^*)\models \phi$ if and only if $\mathcal{F}\models \phi'$, so $\mathfrak{v}_F(\mathcal{F}^*)\equiv\mathfrak{v}_G(\mathcal{G}^*)$ by the same argument.

Thus $\mathcal{F}$ and $\mathcal{G}$ satisfy the conditions of the theorem. It suffices to prove the theorem in the special case of saturated models of cardinality $\aleph_1$, since then we will have $\mathcal{F}\equiv\mathcal{G}$. But $F\equiv\mathcal{F}$ and $G\equiv\mathcal{G}$, so we will have $F\equiv G$, completing the proof in general. 

For the remainder of the proof, we will assume that $F$ and $G$ are saturated models of cardinality $\aleph_1$. In order to prove that $F\equiv G$, we will prove the stronger condition (by Theorem~\ref{thm:isomorphism}) that  $F\cong G$. Keep in mind that we are only working with the saturated case. We do not claim that the valued fields are isomorphic in the general case, just elementarily equivalent.

Note that we cannot simply apply Corollary~\ref{cor:saturatedisomorphism} to show that $F\cong G$, even though these models are saturated. The corollary requires the additional assumption that $F\equiv G$, which is exactly what we are trying to prove. However, we will use this corollary to show that $\overline{F}\cong\overline{G}$. 

As an additional consequence of the method of transforming formulas about the residue class field and value group into formulas about the valued field, we will show that $\overline{F}$ and $\overline{G}$ are saturated models, and except in the trivial case $\mathfrak{v}_F(F^*) = \mathfrak{v}_G(G^*) = \{1\}$, $\mathfrak{v}_F(F^*)$ and $\mathfrak{v}_G(G^*)$ are saturated models.

Note that the residue class fields are infinite by characteristic zero, and except in the trivial case, the value groups are infinite by the linear order properties. Any type in the residue class field or value group of $F$ over a set of cardinality at most $\aleph_1$ can be transformed into a type in $F$ over a set of cardinality at most $\aleph_1$, and since all such types are realized in $F$, this type is realized in the residue class field or value group, and hence these models are $\aleph_1$-saturated. The same argument holds for the residue class field and value group of $G$. 

Now any $\aleph_1$-saturated model has cardinality at least $\aleph_1$. But the value group of $F$ is a subset of $F$, so it has cardinality at most $\aleph_1$. And the residue class field of $F$ embeds into $F$ by Lemma~\ref{lem:residue_subfield}, so it has cardinality at most $\aleph_1$. The same argument holds in $G$, so $\overline{F}$, $\overline{G}$, $\mathfrak{v}_F(F^*)$, and $\mathfrak{v}_G(G^*)$ are all saturated, and $|\overline{F}|=|\overline{G}|=|\mathfrak{v}_F(F^*)|=|\mathfrak{v}_G(G^*)|=\aleph_1$.

(2) \emph{The base case: residue class fields.} By Lemma~\ref{lem:residue_subfield}, $\overline{F}$ and $\overline{G}$ are isomorphic to subfields of $F$ and $G$. We will identify the residue class fields with these subfields. Since $\overline{F}$ and $\overline{G}$ are saturated, and by assumption $\overline{F}\equiv\overline{G}$, $\overline{F}\cong\overline{G}$ by Corollary~\ref{cor:saturatedisomorphism}. Call the isomorphism $f_0$. We have $f_0:\overline{F}\leftrightarrow\overline{G}$, since $f_0$ is an elementary bijection between $\mathfrak{v}_F(\overline{F}^*)$ and $\mathfrak{v}_G(\overline{G}^*)$, as these value groups only contain a singe element, $1$, which is already a constant symbol in the language.

Now we can dispense with the trivial case $\mathfrak{v}_F(F^*)=\mathfrak{v}_G(G^*)=\{1\}$, for in this case, $\overline{F} = F$ and $\overline{G} = G$, so $f_0$ provides the desired isomorphism $F\cong G$. We will assume for the remainder of the proof that we are not in the trivial case, and thus the value groups of $F$ and $G$ are saturated of cardinality $\aleph_1$. 

It remains to show that $\overline{F}$ and $\overline{G}$ are relatively algebraically closed in $F$ and $G$. Let $p\in \overline{F}[t]$. Since $\overline{F}\subseteq \mathcal{O}_F$ by Lemma~\ref{lem:residue_subfield}, we can consider $p$ as an element of $\mathcal{O}_F[t]$. Let $x\in F$ such that $\mathfrak{v}_F(x)<1$. Substituting $x$ for $t$, $p(x)$ is a sum of terms of the form $a_mx^m$, where $\mathfrak{v}_F(a_m) = 1$ if $a_m\neq 0$. Thus $\mathfrak{v}_F(a_mx^m) = \mathfrak{v}_F(x)^m < 1$. Each of the powers of $\mathfrak{v}_F(x)$ is distinct, so $p(x)$ is a sum of terms with distinct values, and by Lemma~\ref{lem:VF} (4), $\mathfrak{v}_F(p(x))$ is the minimum of these, which is less than $1$. But $\mathfrak{v}_F(0) > 1$, so $x$ is not a root of $p$. 

Thus every root of $p$ in $F$ is an element of $\mathcal{O}_F$. Let $x$ be one such root. Then $\overline{x}$ is defined, and since $\overline{p} = p$, $p(\overline{x}) = \overline{p(x)} = \overline{0} = 0$. So $p$ already has the root $\overline{x}\in \overline{F}$, and we can factor $p$ in $\overline{F}$ as $(t-\overline{x})q$, where $q$ is a polynomial of lower degree. Now if $x\neq \overline{x}$, $x$ is still a root of $q$, and we can apply the same argument again to factor $q$. Repeating this process until we reach a polynomial of degree $1$, we see that we must have $x = \overline{x}$. So every root of $p$ in $F$ is in $\overline{F}$, and $\overline{F}$ is relatively algebraically closed.

The same argument shows that $\overline{G}$ is relatively algebraically closed in $G$. 

(3-4) \emph{The inductive step: extending isomorphisms.} We have $f_1:F_1\leftrightarrow G_1$ between relatively algebraically closed subfields of $F$ and $G$, such that $\overline{F}\subseteq F_1$, $\overline{G}\subseteq G_1$, $f_0\subseteq f_1$, and $\mathfrak{v}_F(F_1^*)$ and $\mathfrak{v}_G(G_1^*)$ are countable. Suppose $x\in F_1$. Then $x$ is already in the domain, so we can easily satisfy (3) by simply taking $f_2 = f_1$. 

Otherwise, suppose $x\notin F_1$. Since $F_1$ is relatively algebraically closed in $F$, $x$ is transcendental over $F_1$. We will first prove (3) in two special cases, then prove the general case. 

\emph{Case 1: Adjoining $x$ to $F_1$ does not change the value group of $F_1$.} That is, $\mathfrak{v}_F(F_1(x)^*) = \mathfrak{v}_F(F_1^*)$.

First, note that since $\overline{F}\subseteq F_1\subset F_1(x)$, and the residue class field of $\overline{F}$ is already the whole residue class field of $F$, we have $\overline{F_1(x)} = \overline{F} = \overline{F_1}$.

Since $F_1$ is relatively algebraically closed, $\widetilde{F_1} = F_1$, where $\widetilde{F_1}$ is the relative algebraic closure of $F_1$ in $F$. Moreover, $F$ is Henselian, $\overline{F_1} = \overline{F}$, $char(\overline{F}) = 0$, and trivially $\mathfrak{v}_F(\widetilde{F_1^*}) = \mathfrak{v}_F(F_1^*)$, so we can apply Lemma~\ref{lem:closure_under_roots} to show that $\widetilde{F_1}$ is a Henselization of $F_1$. In particular, $F_1$ is already Henselian. The same argument applied to $G_1$ shows that $G_1$ is Henselian.

We have established that $F_1$ and $G_1$ are Henselian, that $x$ is transcendental over $F_1$, and that adjoining $x$ does not change the value group or residue class field of $F_1$. In order to apply Lemma~\ref{lem:transcendental_extension}, it remains to find $y\in G$ transcendental over $G_1$ such that for all $a\in F_1$, $f_1(\mathfrak{v}_F(x-a)) = \mathfrak{v}_G(y-f_1(a))$. 

We will find our $y$ by using the fact that $G$ is saturated. That is, we will express the valuation condition required on $y$ as a 1-type, which must be realized in $G$.

Since $\mathfrak{v}_F(F_1(x))$ is countable, the set $\{\mathfrak{v}_F(x-b)\,|\,b\in F_1\}$ is countable, and thus there is a countable subset $A_1\subset F_1$ such that for all $b\in F_1$, there exists $a\in A_1$ with $\mathfrak{v}_F(x-a) = \mathfrak{v}_F(x-b)$. 

Let $S = f(A_1) \cup \mathfrak{v}_G(G_1) \subset G_1$, and let $\mathcal{L}_{S}$ be $\mathcal{L}_{VF}\cup \{c_{s}\,|\,s\in S\}$, the language of valued fields extended with a new constant symbol for each element of $f_1(A_1)$ and for each element of the value group $\mathfrak{v}_G(G_1)$. For all $a\in A_1$, let $\phi_a(v_1)$ be the $\mathcal{L}_{S}$-formula
\[
c_{f_1(\mathfrak{v}(x-a))} = \mathfrak{v}(v_1 - c_{f_1(a)}).
\]
Note that $\mathfrak{v}_F(x-a)\in \mathfrak{v}_F(F_1(x)) = \mathfrak{v}_F(F_1)$, so $f_1(\mathfrak{v}(x-a))\in \mathfrak{v}_G(G_1)$, and $c_{f_1(\mathfrak{v}(x-a))}$ is a constant symbol in $\mathcal{L}_{S}$. 

Let $P = \{\phi_a\,|\,a\in A_1\}$. We would like to show that $P$ is a 1-type over $S$, so we must show that $P\cup Th_{A}(G)$ is satisfiable. We will show that every finite subset of $P\cup Th_{A}(G)$ is satisfiable, then apply Compactness.

\emph{Claim:} For every finite set $A\subset A_1$, there is $y_A\in G$ such that for all $a\in A$, $f_1(\mathfrak{v}_F(x-a)) = \mathfrak{v}_G(y_A-f_1(a))$.

Choose $b\in A$ such that $w = \mathfrak{v}_F(x-b)$ takes on its maximum value. We have $\mathfrak{v}_F(F_1) = \mathfrak{v}_F(F_1(x))$, so $w\in \mathfrak{v}_F(F_1)$. Now for each positive integer $n$, we have seen that $\mathfrak{v}_F(n) = 1$, so $\mathfrak{v}_F(nw) = 1\cdot w = w$. Thus for all $n$ and all $a\in A$, 
\begin{eqnarray*}
\mathfrak{v}_F(b-nw-a) &\geq& \min(\mathfrak{v}_F(b-x),\mathfrak{v}_F(nw),\mathfrak{v}_F(x-a))\\
&\geq& \min(w,w,\mathfrak{v}_F(x-a)) \\
&\geq& \mathfrak{v}_F(x-a).
\end{eqnarray*}

Now by Lemma~\ref{lem:VF}, equality holds above whenever $\mathfrak{v}_F(x-a) < w$. We claim that this is the case for all but at most one $n$. For suppose we have $m<n$ with $\mathfrak{v}_F(b-mw-a)>\mathfrak{v}_F(x-a)$ and $\mathfrak{v}_F(b-nw-a)>\mathfrak{v}_F(x-a)$. Then
\begin{eqnarray*}
w &=& \mathfrak{v}_F((n-m)w) \\
&\geq& \min(\mathfrak{v}_F(b-mw-a),\mathfrak{v}_F(-b+nw+a))\\
&>& \mathfrak{v}_F(x-a),
\end{eqnarray*}
in which case equality holds above and $\mathfrak{v}_F(b-nw-a) = \mathfrak{v}_F(x-a)$, a contradiction.

Since $A$ is finite, and for each $a$ there is at most one positive integer $n$ such that $\mathfrak{v}_F(b-nw-a) \neq \mathfrak{v}_F(x-a)$, we can choose $n$ such that from all $a\in A$, $\mathfrak{v}_F(b-nw-a) = \mathfrak{v}_F(x-a)$. 

Let $y_A = f_1(b-nw)$. Then for all $a\in A$, 
\begin{eqnarray*}
f_1(\mathfrak{v}_F(x-a)) &=& f_1(\mathfrak{v}_F(b-nw-a))\\
&=& \mathfrak{v}_G(y_A-f_1(a)),
\end{eqnarray*}
since $f_1\circ\mathfrak{v}_F = \mathfrak{v}_G\circ f_1$. This completes the proof of the claim.

Let $\Delta$ be any finite subset of $P\cup Th_S(G)$. Let $A$ be the subset of $A_1$ consisting of all $a$ such that $\phi_a(v_1) \in \Delta$. Applying the claim, there is $y_A\in G$ such that for all $a\in A$, $f_1(\mathfrak{v}_F(x-a)) = \mathfrak{v}_G(y_A-f_1(a))$. That is, $G\models \phi_a(y_A)$ for all $a\in A$. Clearly, $G$ also satisfies all $\mathcal{L}_S$-sentences of $Th_S(G)$ in $\Delta$, so $\Delta$ is satisfiable. By Compactness (Theorem~\ref{thm:compactness}), $P\cup Th_S(G)$ is satisfiable. 

Hence $P$ is a 1-type over $S$. Since $f_1(A_1)$ is countable, and $\mathfrak{v}_G(G_1)$ is countable by assumption, $S$ is countable. Now $G$ is $\aleph_1$-saturated, so $P$ is realized by an element $y\in G$. 

Now we have that for all $a\in A_1$, $f_1(\mathfrak{v}_F(x-a)) = \mathfrak{v}_G(y-f_1(a))$. We must show that the same is true for all $b\in F_1$. 

Let $\mathfrak{v}_F(x-b) = d$ (the valuation of $x-b$ cannot be $0$, since then we would have $x = b$, but $x\notin F_1$). Since $\overline{F_1(x)} = \overline{F_1}$, there exists $b'\in F_1$ such that $\overline{b'} = \overline{(x-b)d^{-1}}$, that is, $\mathfrak{v}_F((x-b)d^{-1} - b') > 1$. Multiplying both sides by $d = \mathfrak{v}_F(d)$ (by the cross section property),
\begin{eqnarray*}
\mathfrak{v}_F((x-b)d^{-1} - b')\mathfrak{v}_F(d) &>& d \\
\mathfrak{v}_F(x - b - b'd) &>& \mathfrak{v}_F(x-b).
\end{eqnarray*}

By the definition of $A_1$, there exists $a\in A_1$ with $\mathfrak{v}_F(x - (b + b'd)) = \mathfrak{v}_F(x-a)$, and thus $\mathfrak{v}_F(x-a) = \mathfrak{v}_F(x - (b + b'd)) > \mathfrak{v}_F(x-b)$. By Lemma~\ref{lem:VF} (3) and (4),
\begin{eqnarray*}
\mathfrak{v}_F(a - b) &=& \mathfrak{v}_F((x-b) - (x-a))\\
&=& \min(\mathfrak{v}_F(x-b),\mathfrak{v}_F(-(x-a)))\\
&=& \min(\mathfrak{v}_F(x-b),\mathfrak{v}_F(x-a))\\
&=& \mathfrak{v}_F(x-b)\\
&<& \mathfrak{v}_F(x-a).
\end{eqnarray*}

Applying $f_1$, $\mathfrak{v}_G(f_1(a) - f_1(b)) = f_1(\mathfrak{v}_F(a-b)) < f_1(\mathfrak{v}_F(x-a)) = \mathfrak{v}_G(y-f_1(a))$, since $a\in A_1$. Hence,
\begin{eqnarray*}
\mathfrak{v}_G(y-f_1(b)) &=& \mathfrak{v}_G((y - f_1(a)) + (f_1(a) - f_1(b)))\\
&=& \min(\mathfrak{v}_G(y-f_1(a)),\mathfrak{v}_G(f_1(a)-f_1(b))) \\
&=& \mathfrak{v}_G(f_1(a)-f_1(b)) \\
&=& f_1(\mathfrak{v}_F(a-b)) \\
&=& f_1(\mathfrak{v}_F(x-b),
\end{eqnarray*}
as was to be shown.

Finally, we conclude that $y\notin G_1$, for if $y\in G_1$, then $f_1^{-1}(y)\in F_1$, so $0 = \mathfrak{v}_G(y-y) = f_1(\mathfrak{v}_F(x - f_1^{-1}(y)))$, so $x = f_1^{-1}(y)$, and $x\in F_1$, contradicting our choice of $x$. 

We have satisfied all of the hypotheses of Lemma~\ref{lem:transcendental_extension}. The lemma tells us that $\mathfrak{v}_G(G_1(y)*) = \mathfrak{v}_G(G_1)$, $\overline{G_1(y)} = \overline{G_1}$, and $f_1$ can be extended to an isomorphism $g_1:F_1(x)\cong G_1(y)$. 

We have not yet finished satisfying the conditions of (3). In particular, $F_1(x)$ and $G_1(y)$ are not necessarily relatively algebraically closed. By assumption, $F_1$ is relatively algebraically closed, so by Lemma~\ref{lem:closure_under_roots}, $\mathfrak{v}(F_1^*)$ is closed under roots in $\mathfrak{v}(F^*)$. But $\mathfrak{v}(F_1(x)^*) = \mathfrak{v}(F_1^*)$, so $\mathfrak{v}(F_1(x)^*)$ is closed under roots. Again by Lemma~\ref{lem:closure_under_roots}, the relative algebraic closure of $F_1(x)$ in $F$ is a Henselization of $F_1(x)$. Call this field $F_2$. The same argument shows that $G_2$, the relative algebraic closure of $G_1(x)$ in $G$, is a Henselization of $G_1(x)$. 

By Lemma~\ref{lem:unique_extension}, $g_1$ can be extended to an isomorphism $f_2:F_2\cong G_2$. By Lemma~\ref{lem:henselization_exists}, $\mathfrak{v}_F(F_2^*) = \mathfrak{v}_F(F_1(x)^*) = \mathfrak{v}_F(F_1^*)$, and $\mathfrak{v}_G(G_2^*) = \mathfrak{v}_G(G_1(y)^*) = \mathfrak{v}_G(G_1^*)$. In particular, $\mathfrak{v}_F(F_2)$ remains countable, and $f_2$ remains a partial elementary bijection between $\mathfrak{v}_F(F_2^*)$ and $\mathfrak{v}_G(G_2^*)$. This completes the proof of the first special case.

\emph{Case 2: The element $x$ is in the value group of $F$.} That is, $x\in\mathfrak{v}(F^*)$. 

The function $f_1$ restricted to $\mathfrak{v}_F(F_1^*)$ is a partial elementary bijection onto $\mathfrak{v}_G(G_1^*)$. Since $\mathfrak{v}_G(G^*)$ is $\aleph_1$-saturated and $\mathfrak{v}_F(F_1^*)$ is countable, we can choose an element $y\in\mathfrak{v}_G(G^*)$ which realizes the complete type of $x$ in $\mathfrak{v}_F(F)$ over $\mathfrak{v}_F(F_1^*)$, where we interpret the constant symbol corresponding to an element of $\mathfrak{v}_F(F_1^*)$ in $\mathfrak{v}_G(G^*)$ by its image under $f_1$. 

Thus, letting $V$ be the subgroup of $\mathfrak{v}_F(F^*)$ generated by $\mathfrak{v}_F(F_1^*)$ and $x$, and letting $W$ be the subgroup of $\mathfrak{v}_G(G^*)$ generated by $\mathfrak{v}_G(G_1^*)$ and $y$, the restriction of $f_1$ extends to a partial elementary bijection between $V$ and $W$ by mapping $x$ to $y$. Since we have only added one generator to a countable group in each case, $V$ and $W$ are countable.

Define an extension of $f_1$, $g_1:F_1(x)\rightarrow G_1(y)$, by
\[
g_1\left(\frac{d_0+\hdots+d_mx^m}{e_0+\hdots+e_nx^n}\right) = \frac{f_1(d_0)+\hdots+f_1(d_m)y^m}{f_1(e_0)+\hdots+f_1(e_n)y^n},
\]
with all coefficients $d_i$ and $e_j$ in $F_1$. Since $f_1$ is a field isomorphism, it is easy to check that $g_1$ is a field isomorphism. Checking that it is an isomorphism of valued fields takes a little more work.

Let $p(x) = e_0 + \hdots + e_nx^n$ with coefficients in $F_1$. Suppose that for some indices $r,s$ with $r<s$, $e_r\neq 0$, and $e_s\neq 0$, we have $\mathfrak{v}_F(e_rx^r) = \mathfrak{v}_F(e_sx^s)$. Then $x^r\mathfrak{v}_F(e_r) = x^s\mathfrak{v}_F(e_s)$, since $\mathfrak{v}_F(x) = x$ by the cross section property, and $x^{s-r} = \mathfrak{v}_F(e_s)(\mathfrak{v}_F(e_r))^{-1}\in\mathfrak{v}_F(F_1)$. But $F_1$ is relatively algebraically closed, so by Lemma~\ref{lem:closure_under_roots}, $\mathfrak{v}_F(F_1)$ is closed under roots, and thus $x\in \mathfrak{v}_F(F_1)\subseteq F_1$, contradicting our choice of $x$.

Thus for all distinct nonzero coefficients $e_r,e_s$, $\mathfrak{v}_F(e_rx^r)\neq \mathfrak{v}_F(e_sx^s)$, and there is a term $e_qx^q$ of least valuation. By Lemma~\ref{lem:VF} (4), $\mathfrak{v}_F(p(x)) = \mathfrak{v}_F(e_q)x^q\in V$. Since the valuation of any polynomial is in $V$, the valuation of any rational function must also be in $V$, so $\mathfrak{v}_F(F_1(x)^*) = V$. The same argument shows that $\mathfrak{v}_G(G_1(y)^*) = W$. We have established that $F_1(x)$ and $G_1(y)$ are valued subfields of $F$ and $G$ respectively, since $\mathfrak{v}_F(F_1(x)^*) = V \subseteq F_1(x)$ and similarly for $G_1(y)$. 

We have $g_1(p(x)) = f_1(e_0) + \hdots + f_1(e_n)y^n$, and the same argument as above shows that $\mathfrak{v}_G(f_1(e_0) + \hdots + f_1(e_n)y^n) = \mathfrak{v}_G(f_1(e_q))y^q$. Further, since $f_1 \circ \mathfrak{v}_F = \mathfrak{v}_G \circ f_1$,
\begin{eqnarray*}
g_1(\mathfrak{v}_F(p(x))) &=& g_1(\mathfrak{v}_F(e_q)x^q)\\
&=& f_1(\mathfrak{v}_F(e_q))y^q\\
&=& \mathfrak{v}_G(f_1(e_q))y^q\\
&=& \mathfrak{v}_G(g_1(p(x))),
\end{eqnarray*}
so $g_1\circ\mathfrak{v}_F = \mathfrak{v}_G\circ g_1$, and $g_1$ is an isomorphism of valued fields.

We will now establish the conditions of (3) by first passing to Henselizations, then closing the value groups under roots, and finally taking relative algebraic closures, all aided by the lemmas of Section~\ref{subsec:HL}.

By Lemma~\ref{lem:henselization_exists}, $F_1(x)$ and $G_1(y)$ have Henselizations $F_3$ and $G_3$. The fields $F$ and $G$ are Henselian, so by the definition of Henselization, $F_3$ and $G_3$ embed as valued subfields of $F$ and $G$ respectively. Since Henselizations are unique up to isomorphism, and $F_1(x)\cong G_1(y)$, there is an isomorphism of valued fields $g_3:F_3\cong G_3$. The lemma also tells us that $\mathfrak{v}_F(F_3^*) = V$ and $\mathfrak{v}_G(G_3^*) = W$.

Let $\widetilde{V}$ and $\widetilde{W}$ be the closures under roots of $V$ and $W$ in $\mathfrak{v}_F(F)$ and $\mathfrak{v}_G(G)$ respectively. Since $V$ and $W$ are countable, and in closing under roots we add at most one element for each natural number power and each element, $\widetilde{V}$ and $\widetilde{W}$ are countable. Moreover, for every element added to $\widetilde{V}$, there is a corresponding element of $\mathfrak{v}_G(G^*)$ added to $\widetilde{W}$, since $\mathfrak{v}_G(G^*)$ is saturated and $g_1$ restricted to $V$ is a partial elementary bijection onto $W$. Thus the restriction of $g_1$ to $\widetilde{V}$ can be extended to a partial elementary bijection of $\widetilde{V}$ onto $\widetilde{W}$. 

Let $F_4$ and $G_4$ be the subfields of $F$ and $G$ generated by $F_3\cup \widetilde{V}$ and $G_3\cup\widetilde{W}$ respectively. The field $F_4$ is algebraic over $F_3$, since every generator of $F_4$ not in $F_4$ is the root of some polynomial with coefficients in $\mathfrak{v}_F(F_3^*)\subset F_3$. Now letting $\widetilde{F_3}$ be the relative algebraic closure of $F_3$ in $F$, we have $F_4\subseteq \widetilde{F_3}$, so $\mathfrak{v}_F(F_4^*)\subseteq \mathfrak{v}_F(\widetilde{F_3}^*) = \widetilde{V}$ by Lemma~\ref{lem:closure_under_roots}. But $\widetilde{V}\subseteq\mathfrak{v}_F(F_4^*)$, so $\mathfrak{v}_F(F_4^*) = \widetilde{V}$. The same argument shows that $\mathfrak{v}_G(G_4^*) = \widetilde{W}$.

Now the extension of $g_1$ on $V$ to a partial elementary bijection between $\widetilde{V}$ and $\widetilde{W}$ together with the isomorphism $g_3:F_3\cong G_3$ defines a field isomorphism $g_4:F_4\cong G_4$. By Lemma~\ref{lem:unique_extension}, this field isomorphism is also be a valued field isomorphism.

The value groups of $F_4$ and $G_4$, $\widetilde{V}$ and $\widetilde{W}$, are closed under roots, so by Lemma~\ref{lem:closure_under_roots}, the relative algebraic closures of $F_4$ and $G_4$, $\widetilde{F_4}$ and $\widetilde{G_4}$, are Henselizations of $F_4$ and $G_4$ respectively. Once again, since Henselizations are unique up to isomorphism, there is an isomorphism of valued fields $f_2:\widetilde{F_4}\rightarrow\widetilde{G_4}$ extending $g_4$. 

Henselizations have the same value groups as their base fields by Lemma~\ref{lem:henselization_exists}, so $\mathfrak{v}_F(\widetilde{F_4}^*) = \widetilde{V}$ and $\mathfrak{v}_G(\widetilde{G_4}^*) = \widetilde{W}$. We have already seen that $\widetilde{V}$ and $\widetilde{W}$ are countable and that the isomorphism restricts to a partial elementary bijection between them. Hence, taking $F_2 = \widetilde{F_4}$ and $G_2 = \widetilde{G_4}$, we have $f_2:F_2\leftrightarrow G_2$. This completes the proof of the second special case.

\emph{The general case.}
We have $x$ transcendental over $F_1$, and we may assume that $x$ is not in the value group of $F$ and that the value group of $F_1(x)$ strictly contains the value group of $F_1$, since these cases have been dealt with. The idea now is to repeatedly apply the second special case to first adjoin each new element of the value group which would be added upon adjoining $x$. Then when we adjoin $x$ to the result, no further elements are added to the value group, and we are done by the first special case.

We have already established that $\overline{F_1(x)} = \overline{F_1} = \overline{F}$, and we have dealt with the case in which the value group is trivial, so by Lemma~\ref{lem:transcendental_value_group}, $\mathfrak{v}_F(F_1(x)^*)$ is countable.

Let $\{x_i\,|\,i\in\mathbb{N}\}$ be an enumeration of the elements of $\mathfrak{v}_F(F_1(x)^*)$ not in $\mathfrak{v}_F(F_1^*)$. Let $F_{x_0} = F_1$ and $G_{x_0} = G_1$. Applying the second special case, for each $i\in\mathbb{N}$, we can extend $f_{x_i}$ to $f_{x_{i+1}}:F_{x_{i+1}}\leftrightarrow G_{x_{i+1}}$, with $x_i\in F_{x_{i+1}}$. 

Let $F_2 = \bigcup_{i\in\mathbb{N}} F_{x_i}$, $G_2 = \bigcup_{i\in\mathbb{N}} G_{x_i}$, and $f_2 = \bigcup_{i\in\mathbb{N}} f_{x_i}$. Then $f_2:F_2\leftrightarrow G_2$, and $\mathfrak{v}_F(F_1(x)^*)\subseteq \mathfrak{v}_F(F_2^*)$. 

But we are not quite done, because adjoining $x$ to $F_2$ may add elements to the value group of $F_2$. So we repeat this argument, finding for each $i\geq 2$ an extension of $f_{i-1}$, $f_i:F_i\leftrightarrow G_i$ such that $\mathfrak{v}_F(F_{i-1}(x)^*)\subseteq\mathfrak{v}_F(F_i^*)$.

Let $F_\omega = \bigcup_{i\geq 1} F_i$, $G_\omega = \bigcup_{i\geq 1} G_i$, and $f_\omega = \bigcup_{i\geq 1} f_i$. Then $f_\omega:F_\omega\leftrightarrow G_\omega$. Consider $F_\omega(x)$. For element $x'\in F_\omega(x)$, $x'\in F_i(x)$ for some $i$, and thus $x'\in F_{i+1}\subseteq F_\omega$. So adjoining $x$ to $F_\omega$ does not add any elements to the value group, and we have reduced to the first special case.

All arguments given above hold with the roles of $F$ and $G$ reversed, so we have also established (4). 

(5) \emph{The back-and-forth argument.} Let $(a_\alpha\,:\alpha<\aleph_1)$ and $(b_\alpha\,:\alpha<\aleph_1)$ be enumerations of $F$ and $G$ respectively. We start with the isomorphism $f_0:\overline{F}\leftrightarrow\overline{G}$ established in (2) and inductively build a chain of isomorphisms $(f_\alpha\,:\,\alpha<\aleph_1)$ such that for each $\alpha$, $a_\alpha$ is in the domain of $f_{\alpha+1}$ (using (3)) and $b_\alpha$ is in the range of $f_{\alpha+1}$ (using (4)). By the familiar back-and-forth argument, $f = \bigcup_{\alpha<\aleph_1}f_\alpha$ is an isomorphism $F\cong G$. This completes the proof.
\end{proof}

Now that the heavy lifting is done, what remains is putting together the pieces.

\begin{thm}[Ax-Kochen Principle]\label{thm:AKP}
Let $\phi$ be an $\mathcal{L}_{VF}$-sentence. Then $\mathbb{Q}_p\models\phi$ for all but finitely many primes $p$ if and only if $\mathbb{F}_p((t))\models\phi$ for all but finitely many primes $p$. 
\end{thm}
\begin{proof}
The class of Henselian valued fields is elementary by Lemma~\ref{lem:HVF} and therefore closed under ultraproduct by Corollary~\ref{cor:el_class}. For all $p$, $\mathbb{Q}_p$ and $\mathbb{F}_p((t))$ are complete discrete valued fields, so by Theorem~\ref{thm:DVFhensel} they are Henselian valued fields. Thus for any ultrafilter $\mathcal{D}$ on the set of primes, $\prod\mathbb{Q}_p/\mathcal{D}$ and $\prod\mathbb{F}_p((t))/\mathcal{D}$ are Henselian valued fields.

In Lemma~\ref{lem:UP_resval}, we saw that for any nonprincipal ultrafilter $\mathcal{D}$, the residue fields $\overline{\prod\mathbb{Q}_p/\mathcal{D}}$ and $\overline{\prod\mathbb{F}_p((t))/\mathcal{D}}$ have characteristic zero. We also showed that $\mathfrak{v}(\prod\mathbb{Q}_p/\mathcal{D})\cong \prod\mathbb{Z}/\mathcal{D}\cong\mathfrak{v}(\prod\mathbb{F}_p((t))/\mathcal{D})$ and $\overline{\prod\mathbb{Q}_p/\mathcal{D}}\cong \prod\mathbb{F}_p/\mathcal{D}\cong\overline{\prod\mathbb{F}_p((t))/\mathcal{D}}$. By Theorem~\ref{thm:isomorphism}, isomorphism implies elementary equivalence, so these Henselian valued fields have elementarily equivalent value groups and residue class fields. These are the conditions of Theorem~\ref{thm:el_eq}, so $\prod\mathbb{Q}_p/\mathcal{D}\equiv \prod\mathbb{F}_p((t))/\mathcal{D}$.

Since the elementary equivalence holds for any nonprincipal ultrafilter, applying Theorem~\ref{thm:template} completes the proof.
\end{proof}

\subsection{The Ax-Kochen Theorem}\label{subsec:AKT}

The Ax-Kochen Principle can be used to prove a whole family of theorems about the $p$-adic fields, but its most famous application is the Ax-Kochen Theorem, which addresses Artin's conjecture that $\mathbb{Q}_p$ is $C_2$ for all primes $p$. After traveling far afield, we finally return to nontrivial zeros of homogeneous polynomials.

An important subtlety to the Ax-Kochen Theorem arises from the fact that the property $C_2$ cannot be expressed as a first-order $\mathcal{L}_{VF}$-sentence, since we cannot quantify over polynomials of all degrees. Thus, we cannot apply the Ax-Kochen Principle to prove that $\mathbb{Q}_p$ is $C_2$ for all but finitely many $p$. 

However, when we restrict our attention to polynomials of a fixed degree, we can express a property which is equivalent to $C_2(d)$ as an $\mathcal{L}_{VF}$-sentence, and we can apply the Ax-Kochen Principle to prove that $\mathbb{Q}_p$ is $C_2(d)$ for all but finitely many $p$. Note that the finite set of exceptional primes may be different for each degree $d$.

\begin{thm}[Ax-Kochen Theorem]\label{thm:AKT}
For all degrees $d>0$, there exists a finite set of primes $P(d)$ such that for all $p\notin P(d)$, if $f$ is a homogeneous polynomial over $\mathbb{Q}_p$ of degree $d$ in $n$ variables such that $n>d^2$, then $f$ has a nontrivial zero in $\mathbb{Q}_p^n$. 
\end{thm}
\begin{proof}
Let $C_2(d)$ be the property that every homogeneous polynomial of degree $d$ in $n$ variables such that $n>d^2$ has a nontrivial zero. 

We first show that $C_2(d)$ is equivalent to the property that every homogeneous polynomial of degree $d$ in $d^2+1$ variables has a nontrivial zero. We will call this property $\phi_d$. Clearly $C_2(d)$ implies $\phi_d$. Conversely, if $f(x_1,\hdots,x_n)$ is a homogeneous polynomial of degree $d$ with $n>d^2$, then setting the extra variables to $0$, $g(x_1,\hdots,x_{d^2+1}) = f(x_1,\hdots,x_{d^2+1},0,\hdots,0)$ is either the zero polynomial or a homogeneous polynomial of degree $d$ in $d^2+1$ variables. In the first case, any nontrivial choice of values for the $x_1,\hdots,x_{d^2+1}$ is a nontrivial zero of $f$. In the second case, if $\phi_d$ holds, then $g$ has a nontrivial zero $(\alpha_1,\hdots,\alpha_{d^2+1})$, so $(\alpha_1,\hdots,\alpha_{d^2+1},0,\hdots,0)$ is a nontrivial zero of $f$.

We would like to express the property $\phi_d$ as an $\mathcal{L}_{VF}$-sentence in order to apply the Ax-Kochen Principle. To do this, we need to quantify over all possible homogeneous polynomials of degree $d$ in $d^2+1$ variables. 

Each monomial of such a polynomial has degree $d$, so it is a choice of $d^2+1$ exponents $n_1,\hdots,n_{d^2+1}\in\mathbb{N}$ for the $d^2+1$ variables, such that $\sum_{i=1}^{d^2+1}n_i = d$. Letting $\theta(d)$ be the total number of such choices, we can enumerate all possible monomials as $m_1,\hdots,m_{\theta(d)}$. Then each homogeneous polynomial of degree $d$ in $d^2+1$ variables is uniquely determined by a choice of $\theta(d)$ coefficients $a_1,\hdots,a_{\theta(d)}$, one for each $m_i$, such that at least one of the coefficients is nonzero.

If $m_i$ is the monomial $x_1^{n_1}\hdots x_{d^2+1}^{n_{d^2+1}}$, then we define the $\mathcal{L}_{VF}$-term $m_i(x_1,\hdots,x_{d^2+1})$ in free variables $(x_1,\hdots,x_{d^2+1})$ to be 
\[
\underbrace{x_1\cdot\hdots\cdot x_1}_{n_1\,\text{times}}\cdot\hdots\cdot \underbrace{x_{d^2+1}\cdot\hdots\cdot x_{d^2+1}}_{n_{d^2+1}\,\text{times}}.
\] 

We will use the variables $a_i$ and $x_i$ for clarity. Formally, they are choices of $v_j$ from our infinite set of variables $\mathcal{V}=\{v_1,v_2,\hdots\}$. We can express the property $\phi_d$ with the following $\mathcal{L}_{VF}$-sentence:

\begin{eqnarray*}
\forall a_1\hdots\forall a_{\theta(d)}&&\lnot((a_1=0)\land\hdots\land(a_{\theta(d)}=0))\rightarrow\\
 (\exists x_1\hdots \exists x_{d^2+1}&& \lnot((x_1=0)\land\hdots\land(x_{d^2+1}=0)) \land \\ 
 &&(a_1 \cdot m_1(x_1,\hdots,x_{d^2+1}) + \hdots\\
 && + a_{\theta(d)}\cdot m_{\theta(d)}(x_1,\hdots,x_{d^2+1}) = 0)).
\end{eqnarray*}

Now for a valued field $F$, $F\models \phi_d$ if and only if $F$ has the property $C_2(d)$. By Theorem~\ref{thm:F((t))}, $\mathbb{F}_p((t))$ is $C_2$, and therefore has the property $C_2(d)$, for all primes $p$. Thus for all primes $p$, $\mathbb{F}_p((t))\models\phi_d$. 

Applying the Ax-Kochen Principle, $\mathbb{Q}_p\models\phi_d$ for all but finitely many $p$, and thus $\mathbb{Q}_p$ has the property $C_2(d)$ for all but finitely many $p$. Let $P(d)$ be this finite exceptional set.

Then for all $p\notin P(d)$, $C_2(d)$ says that if $f$ is a homogeneous polynomial over $\mathbb{Q}_p$ of degree $d$ in $n$ variables such that $n>d^2$, then $f$ has a nontrivial zero in $\mathbb{Q}_p^n$.
\end{proof}

\newpage

\appendix

\section{Ordinals, Cardinals, and Transfinite Induction}\label{sec:OCTI}

This appendix gives a very brief and relatively informal overview of the transfinite numbers. The interested reader is encouraged to find a more thorough development, for instance in Jech's \emph{Set Theory} \cite{Jech}.

There are two types of transfinite numbers, ordinals and cardinals. Intuitively, ordinals generalize ordered numbers (``first'', ``second'', ``third''), while cardinals generalize amount, (``one'', ``two'', ``three''). Since the standard set theoretic construction defines cardinals as special types of ordinals, we will take up ordinals first.

\subsubsection*{Ordinals}

Ordinals represent order relations which are linear and well founded; that is, there a least element, and every element has a unique element immediately following it in the order. In this way, they generalize the order properties of sets of natural numbers, and, as we will see, provide a structure upon which induction makes sense. 

We will begin with an informal description of ordinals, and then present the set theoretic construction. We start with a canonical least ordinal, $0$, which represents the ordering on the empty set. Aside from $0$, there are two types of ordinals, successor ordinals and limit ordinals.

Given an ordinal $\alpha$, there is a successor ordinal $\alpha+1$ which represents the ordering of $\alpha$ with an additional element appended which is greater than all the elements in the ordering $\alpha$.

Given an infinite set of ordinals, $C$, there is a limit ordinal which represents the ordering on all elements in all the orderings in $C$. The first limit ordinal (also the first infinite ordinal) is $\omega$, which is the limit of the ordinals $\{0,1,2,\hdots\}$ and represents the ordering on the set of all natural numbers. Note that since the successor of a finite ordinal is still an ordering on finitely many elements, we cannot arrive at $\omega$ through the successor operation by appending elements one by one, only by the limit construction.

The table below demonstrates the order relations represented by a few ordinals. The circles are ordered left to right. The successor operation is indicated by adding a circle on the right, and the limit operation is represented by ($\hdots$). The ordinal $\omega 2$ is the limit of the ordinals $\{0,1,\hdots,\omega,\omega+1,\hdots\}$.

\begin{eqnarray*}
0 && \\
1 && \circ \\
2 && \circ \circ \\
\vdots && \vdots \\
\omega && \circ \circ \circ \circ \hdots \\
\omega + 1 && \circ \circ \circ \circ \hdots \circ \\
\omega + 2 && \circ \circ \circ \circ \hdots \circ \circ \\
\vdots && \vdots \\
\omega 2 && \circ \circ \circ \circ \hdots \circ \circ \circ \circ \hdots \\
\vdots && \vdots
\end{eqnarray*}

Ordinals are quite useful for indexing infinite collections and performing induction in infinite settings. For example, if $C$ is a chain of sets with a least element and order relation (defined by inclusion) corresponding to the ordinal $\beta$, we can index the elements of $C$ by $(C_\alpha\,:\,\alpha<\beta)$. 

If there is a proposition $P_\alpha$ for each ordinal $\alpha$ (collections of propositions like this often correspond to collections of objects indexed by ordinals), then we can prove that $P_\alpha$ is true for all $\alpha$ by a method similar to induction on the natural numbers. The main difference is that we must also deal with the limit case.

\begin{athm}[Transfinite Induction, {\cite[Theorem A.8]{Marker}}]
Suppose that $P_\alpha$ is a proposition for each ordinal $\alpha$. Suppose that
\begin{enumerate}
\item $P_0$ is true,
\item if $P_\alpha$ is true, then $P_{\alpha+1}$ is true, and
\item if $\alpha$ is a limit ordinal and $P_\beta$ is true for all $\beta < \alpha$, then $P_\alpha$ is true.
\end{enumerate}
Then $P_\alpha$ is true for all ordinals $\alpha$.
\end{athm}

Examples of transfinite induction in this thesis can be found in the proofs of theorems requiring back and forth arguments, most explicitly in Theorem~\ref{thm:saturatedbijection}.

There is a very elegant set theoretic construction of the ordinals. Since the relation $\in$ is the primitive binary relation of set theory, we will construct our ordinals so that they are ordered by $\in$. In particular, we will define an ordinal to be the set containing all ordinals less than it.

We define $0 = \emptyset$, since there are no ordinals less than $0$. 

Given an ordinal $\alpha$, we define $\alpha + 1 = \alpha \cup \{\alpha\}$. Then $\alpha + 1$ is the set containing all the elements of $\alpha$ (all ordinals less than $\alpha$) and $\alpha$ itself.

Given a set of ordinals $C$, we define the limit of $C$ by $\delta = \bigcup_{\alpha\in C} \alpha$. Suppose that $C$ is unbounded above, that is, for each $\alpha\in C$ there is a $\beta\in C$ such that  $\alpha < \beta$. Then $\alpha \in \beta$, and hence $\alpha \in \delta$, so $\alpha < \delta$, and $\delta$ is greater than every element of $C$.

The table below demonstrates the set theoretic representations of a few ordinals.

\begin{eqnarray*}
0 && \emptyset\\
1 && \{0\} = \{\emptyset\} \\
2 && \{0, 1\} = \{\emptyset, \{\emptyset\}\} \\
\vdots && \vdots \\
\omega && \{0, 1, 2, \hdots\} = \{\emptyset, \{\emptyset\}, \{\emptyset \{\emptyset\}\}, \hdots\}\\
\omega + 1 && \{0,1,2,\hdots,\omega\} = \{\emptyset, \{\emptyset\}, \{\emptyset \{\emptyset\}\}, \hdots, \{\emptyset, \{\emptyset\}, \{\emptyset \{\emptyset\}\}, \hdots\}\}\\
\vdots && \vdots
\end{eqnarray*}

By repeatedly taking limits, we can construct larger and larger ordinals. We can construct $\omega 3$ as the limit of $\{\omega2,\omega2 + 1,\hdots\}$. The limit of $\{\omega,\omega2,\omega3,\hdots\}$ is $\omega \omega = \omega^2$. The limit of $\{\omega^2,\omega^22,\omega^23,\hdots\}$ is $\omega^2\omega = \omega^3$, and the limit of $\{\omega,\omega^2,\omega^3,\hdots\}$ is $\omega^\omega$. Continuing in this way, we can construct larger towers $\omega^{\omega^\omega}$, $\omega^{\omega^{\omega^\omega}}$, and so forth. The limit of all these towers is yet another ordinal. 

However, all the ordinals we have discussed so far are still relatively small. To say what we mean by small, we must introduce the notion of cardinality.

\subsubsection*{Cardinals}

We say that two sets have the same cardinality if there is a bijection between them. Finite sets with different numbers of elements clearly have distinct cardinalities, since their elements cannot be put into 1-1 correspondence. With his famous diagonalization argument, Cantor showed that infinite sets can also have distinct cardinalities. 

Formally, we define the cardinality of a set $A$ to be the least ordinal $\alpha$ such that $A$ can be put into bijection with $\alpha$, and we denote this ordinal by $|A|$. A cardinal is an ordinal which is the cardinality of some set.

Note that as a consequence of this definition, we can describe the cardinals as those ordinals which cannot be put into bijection with any ordinals less than themselves, since for any such ordinal $\alpha$, $|\alpha| = \alpha$.

All finite ordinals are cardinals. The first infinite cardinal is $\omega$. When we are working with $\omega$ as a cardinal, we will denote it by $\aleph_0$. This is the cardinality of the set of natural numbers. If a set $A$ has cardinality $\aleph_0$, we say that $A$ is countable, since $A$ can be ``counted'', that is, put into bijection with $\mathbb{N}$.

Cantor also showed that any countable union of countable sets is countable. All the ordinals described above can be constructed as limits of countable sequences of ordinals, so they are all countable.

There is a simple construction of the first uncountable ordinal. Take $C$ to be the set of all countable ordinals. The limit of $C$ is $\aleph_1 = \bigcup_{\alpha\in C} \alpha$. The limit $\aleph_1$ is strictly greater than every countable ordinal, so it must be uncountable. Moreover, it contains only countable ordinals, so it is the least uncountable ordinal, and thus is a cardinal, the smallest cardinal greater than $\aleph_0$.

Repeating this argument, we can construct the next cardinal $\aleph_2$ by taking the limit of all ordinals of cardinality $\aleph_1$. The limit of the cardinals $\{\aleph_0,\aleph_1,\aleph_2,\hdots\}$ is the limit cardinal $\aleph_\omega$, and further limits produce greater cardinals, indexed by the ordinals. If $\kappa = \aleph_\alpha$ for some ordinal $\alpha$, we denote by $\kappa^+$ the next cardinal, $\aleph_{\alpha+1}$.

We can define addition, multiplication, and exponentiation of cardinals. For $\kappa$ and $\lambda$ cardinals and $A$ and $B$ disjoint sets with $|A| = \kappa$ and $|B| = \lambda$, we define $\kappa + \lambda = |A\cup B|$, the cardinality of the union of $A$ and $B$, $\kappa\lambda = |A\times B|$, the cardinality of the cartesian product of $A$ and $B$, and $\kappa^\lambda = |A^B|$, the cardinality of the set of functions from $B$ to $A$. 

The following facts are useful for determining the cardinalities of sets.

\begin{athm}[Cardinal Arithmetic]\label{thm:cardinal_arithmetic}
Let $\kappa$ and $\lambda$ be cardinals. If both $\kappa$ and $\lambda$ are finite, then addition, multiplication, and exponentiation agree with the usual arithmetic of natural numbers. Otherwise,
\begin{enumerate}
\item $\kappa + \lambda = \kappa\lambda = \max(\kappa,\lambda)$,
\item if $\lambda$ is infinite and $\kappa \leq \lambda$, then $\kappa^\lambda = 2^\lambda$, and
\item if $\lambda$ is finite and $\kappa$ is infinite, then $\kappa^\lambda = \kappa$.
\end{enumerate}
\end{athm}

\begin{athm}\label{thm:cardinal_chain}
Let $(A_\alpha\,:\,\alpha<\beta)$ be a chain of sets indexed by the ordinal $\beta$, where $A_\alpha\subseteq A_{\alpha'}$ if $\alpha < \alpha'$. For all $\alpha < \beta$, let $\kappa_\alpha = |A_\alpha|$. Then if $A = \bigcup_{\alpha < \beta} A_\alpha$, $|A|=\bigcup_{\alpha<\beta}\kappa_\alpha$, the limit of the cardinals $\kappa_\alpha$.
\end{athm}

The diagonalization argument provides a different way of constructing distinct infinite cardinals. Cantor showed that for any set $A$, its power set $\mathcal{P}(A)$ has strictly greater cardinality. The cardinality of $\mathcal{P}(A)$ is $2^{|A|}$, since the elements of the power set are in bijection with the functions $A\rightarrow \{0,1\}$. A subset $B\subseteq A$ corresponds to the function $f_B$ defined by $f_B(a) = 1$ if $a\in B$ and $f_B(a) = 0$ if $a\notin B$.

As a consequence of Cantor's Theorem, the sequence $\aleph_0, 2^{\aleph_0}, 2^{2^{\aleph_0}}, \hdots$ is an increasing sequence of distinct cardinals. The terms of this sequence are sometimes denoted $\beth_0,\beth_1,\beth_2,\hdots$, and by taking limits, $\beth_\alpha$ may be defined for any ordinal $\alpha$. For all ordinals $\alpha$, $\aleph_\alpha\leq \beth_\alpha$, but the question of whether the sequences $\aleph_0,\aleph_1,\hdots$ and $\beth_0,\beth_1,\hdots$ differ is independent from the usual axioms of set theory.

\begin{CH}
There are no cardinals between $\aleph_0$ and $2^{\aleph_0}$; that is, $\aleph_1 = 2^{\aleph_0}$.
\end{CH}

\begin{GCH}
For all ordinals $\alpha$, there are no cardinals between $\aleph_\alpha$ and $2^{\aleph_\alpha}$; that is, $\aleph_{\alpha+1} = 2^{\aleph_\alpha}$.
\end{GCH}

\newpage

\section{Special Models}\label{sec:SM}

In our proof of the Ax-Kochen Theorem, we assumed the Continuum Hypothesis in order to use Theorem~\ref{cor:ch}, that all complete theories have saturated models. The Continuum Hypothesis can be eliminated from the proof by replacing saturated models with special models.

\begin{adefin}\label{def:special}
Let $T$ be a complete theory with infinite models in a countable language $\mathcal{L}$. A model $\mathcal{M}\models T$ with domain $M$ is called special if it is the union of an elementary chain of models $(\mathcal{M}_\beta\,:\,\beta<|M|, \beta\,\text{an infinite cardinal})$ such that each $\mathcal{M}_\beta$ is $\beta^+$-saturated. The elementary chain is called a specializing chain of $\mathcal{M}$.
\end{adefin}

Note that in the definition, nothing is required about the cardinalities of the $\mathcal{M}_\beta$.

The analogue of Corollary~\ref{cor:saturatedisomorphism} also holds for special models. 

\begin{athm}[{\cite[Theorem 5.1.17]{Chang}}]\label{thm:specialisomorphism}
If $\mathcal{M}$ and $\mathcal{N}$ are special models of a complete theory $T$ of the same cardinality $\kappa>\aleph_0$, then $\mathcal{M}\cong\mathcal{N}$. 
\end{athm}

The idea of the proof is to use a back and forth argument, where at each stage partial elementary bijections are constructed using Theorem~\ref{thm:saturatedbijection} between subsets of the $\beta^+$-saturated submodels of $\mathcal{M}$ and $\mathcal{N}$.

The advantage of special models is that we can show that all complete theories have special models without appealing to the Continuum Hypothesis. 

\begin{athm}[{\cite[Proposition 5.1.8]{Chang}}]\label{thm:specialexistence}
For any $\mathcal{L}$-structure $\mathcal{M}$, there is a special elementary extension of $\mathcal{M}$.
\end{athm}

The idea of the proof is to construct a chain of $\kappa$-saturated models for increasing cardinalities $\kappa$, then take limits.

The proof of Theorem~\ref{thm:el_eq} can be altered to reduce to the special case instead of the saturated case. Unfortunately, this complicates the argument significantly, since the back and forth argument must take the specializing chains into account. Additionally, the special models guaranteed by Theorem~\ref{thm:specialexistence} may have cardinality larger than $\aleph_1$, so more complicated cardinality and enumeration arguments are required.

\newpage

\section{The Resultant}\label{sec:TR}

\begin{adefin}\label{def:res}
Let $R$ be a ring. The resultant, $Res:R[x]\times R[x]\rightarrow R$, is the function which maps two polynomials over $R$, $f = a_nx^n + a_{n-1}x^{n-1} + \hdots + a_0$ and $g = b_mx^m + b_{m-1}x^{m-1} + \hdots + b_0$ of degrees $n$ and $m$ respectively, to the determinant of the following $(m+n)\times(m+n)$ matrix:
\[
\begin{array}{cc}
m & 
\begin{cases}\\
\\
\\
\\
\\
\end{cases}\\
n &
\begin{cases}\\
\\
\\
\\
\end{cases}
\end{array}
\underbrace{
\left(
\begin{array}{cccccccc}
a_n & \hdotsfor{2} & a_0 & 0 & \hdotsfor{2} & 0\\
0 & a_n & \hdotsfor{2} & a_0 & 0 & \hdots & 0 \\
\vdots & \ddots & \ddots & \ddots  & \ddots &  \ddots & \ddots &  \vdots \\
\vdots & \ddots & \ddots & \ddots  & \ddots &  \ddots & \ddots &  \vdots \\
0 & \hdotsfor{2} & 0 & a_n & \hdotsfor{2} & a_0 \\
b_m & \hdotsfor{3} & b_0 & 0 & \hdots & 0 \\
0 & b_m & \hdotsfor{3} & b_0 & \ddots & 0 \\
\vdots & \ddots & \ddots & \ddots & \ddots &  \ddots & \ddots  & \vdots \\
0 & \hdots & 0 & b_m & \hdotsfor{3} & b_0 \\
\end{array}
\right)
}_{\displaystyle{m+n}}.
\]
\end{adefin}

The next theorem gives an alternate expression for the resultant. We will only use it for the implication that $Res(f,g)\neq 0$ if $f$ and $g$ are relatively prime.

\begin{athm}[{\cite[Proposition 8.3]{Lang}}]\label{thm:resprime}
Let $R$ be a subring of a field $K$, and let $f,g\in R[x]$, with $f = a_nx^n + a_{n-1}x^{n-1} + \hdots + a_0$ and $g = b_mx^m + b_{m-1}x^{m-1} + \hdots + b_0$. Then $Res(f,g) = a_n^mb_m^n\prod_{i=1}^n\prod_{j=1}^m(\alpha_i - \beta_j)$, where the $\alpha_i$ and $\beta_j$ are the roots of $f$ and $g$ in an algebraic closure of $K$. Thus $Res(f,g) = 0$ if and only if $f$ and $g$ have a common root, and if $f$ and $g$ are relatively prime, $Res(f,g)\neq 0$.
\end{athm}
\begin{proof}

Consider the linear equations
\begin{eqnarray*}
x^{m-1}f(x) &=& a_nx^{m+n-1}+a_{n-1}x^{m+n-2}+\hdots +a_0x^{m-1}\\
x^{m-2}f(x) &=& a_nx^{m+n-2}+a_{n-1}x^{m+n-3}+\hdots +a_0x^{m-2}\\
\vdots && \vdots \\
f(x) &=& a_nx^{n}+a_{n-1}x^{n-1}+\hdots +a_0\\
x^{n-1}g(x) &=& b_mx^{m+n-1}+b_{n-1}x^{m+n-2}+\hdots +b_0x^{n-1}\\
x^{n-2}g(x) &=& b_mx^{m+n-2}+b_{n-1}x^{m+n-3}+\hdots +b_0x^{n-2}\\
\vdots && \vdots\\
g(x) &=& b_mx^{m}+b_{m-1}x^{m-1}+\hdots +b_0.
\end{eqnarray*}

Let $C$ be the vector on the left, $(x^{m-1}f(x), x^{m-2}f(x), \hdots, g(x))$. Let $C_1,\hdots,C_{m+n}$ be the vectors of coefficients of $x$, with powers of $x$ aligned. So $C_1 = (a_n,0,\hdots,0,b_m,0,\hdots,0)$, the coefficients of $x^{m+n-1}$, $C_2 = (a_{n-1}, a_n, 0,\hdots,0,b_{m-1},b_m,0,\hdots,0)$, the coefficients of $x^{m+n-2}$, etc. Note that these vectors are the columns of the resultant matrix.

The linear equations can then be expressed as $C=C_1x^{n+m-1} + \hdots + C_{m+n}x^0 $, and the right side of this equality is just the resultant matrix multiplied by the column vector $(x^{n+m-1},\hdots,x^0)$.

If we replace the $(m+n)^{th}$ column of the resultant matrix with the column vector $C$, Cramer's rule tells us that 
\[
\frac{det(C_1,\hdots,C_{m+n-1},C)}{det(C_1,\hdots,C_{m+n})} = x^0 = 1,
\]
since $x^0$ is the $(m+n)^{th}$ entry of $(x^{n+m-1},\hdots,x^0)$. By $det(v_1,\hdots,v_n)$, we mean the determinant of the matrix with columns $v_1,\hdots,v_n$.

So $Res(f,g) = det(C_0,\hdots,C_{m+n}) = det(C_0,\hdots,C_{m+n-1},C)$. Computing this determinant, we find that every term contains a factor of $f(x)$ or $g(x)$ from the column $C$. Grouping the terms divisible by $f$ and those divisible by $g$, we find that there are polynomials $p(x), q(x)\in R[x]$ such that $p(x)f(x) + q(x)g(x) = Res(f,g)$. 

Suppose that $f$ and $g$ have a common root $\alpha$ in an algebraic closure of $K$. Substituting $\alpha$ for $x$ in the equation above, we see that $Res(f,g) = 0$.

Now in the algebraic closure, we can factor $f$ and $g$ as $f=a_n\prod_{i=1}^n (x-\alpha_i)$ and $g = b_m\prod_{j=1}^m (x-\beta_j)$. Comparing the coefficients of powers of $x$, we obtain the following expressions for the coefficients:
\begin{eqnarray*}
a_n &=& a_n\\
a_{n-1} &=& -a_n(\alpha_1 + \hdots + \alpha_n)\\
\vdots && \vdots\\
a_0 &=& (-1)^na_n(\alpha_1\alpha_2\hdots\alpha_n)
\end{eqnarray*}
and similarly for the $b_j$. In this way, we can view the coefficients $a_i$ and $b_j$ as symmetric polynomials $(-1)^ia_nS_i(\alpha_1,\hdots,\alpha_n)$ and $(-1)^jb_mT_j(\beta_1,\hdots,\beta_m)$, where $deg(S_i) = n-i$ and $deg(T_j) = m-j$. 

Now computing the resultant, we see that 
\[
Res(f,g)=
\left|
\begin{array}{cccccccc}
a_nS_n & \hdotsfor{2} & (-1)^na_0S_0 & 0 & \hdotsfor{2} & 0\\
0 & a_nS_n & \hdotsfor{2} & (-1)^na_0S_0 & 0 & \hdots & 0 \\
\vdots & \ddots & \ddots & \ddots  & \ddots &  \ddots & \ddots &  \vdots \\
\vdots & \ddots & \ddots & \ddots  & \ddots &  \ddots & \ddots &  \vdots \\
0 & \hdotsfor{2} & 0 & a_nS_n & \hdotsfor{2} & (-1)^na_0S_0 \\
b_mT_m & \hdotsfor{3} & (-1)^mb_0T_0 & 0 & \hdots & 0 \\
0 & b_mT_m & \hdotsfor{3} & (-1)^mb_0T_0 & \ddots & 0 \\
\vdots & \ddots & \ddots & \ddots & \ddots &  \ddots & \ddots  & \vdots \\
0 & \hdots & 0 & b_mT_m & \hdotsfor{3} & (-1)^mb_0T_0 \\
\end{array}
\right|.
\]
The first $m$ rows have a factor of $a_n$, and the next $n$ rows have a factor of $b_m$, so factoring them out, 
\[
Res(f,g)=
a_n^mb_m^n\left|
\begin{array}{cccccccc}
S_n & \hdotsfor{2} & (-1)^nS_0 & 0 & \hdotsfor{2} & 0\\
0 & S_n & \hdotsfor{2} & (-1)^nS_0 & 0 & \hdots & 0 \\
\vdots & \ddots & \ddots & \ddots  & \ddots &  \ddots & \ddots &  \vdots \\
\vdots & \ddots & \ddots & \ddots  & \ddots &  \ddots & \ddots &  \vdots \\
0 & \hdotsfor{2} & 0 & S_n & \hdotsfor{2} & (-1)^nS_0 \\
T_m & \hdotsfor{3} & (-1)^mT_0 & 0 & \hdots & 0 \\
0 & T_m & \hdotsfor{3} & (-1)^mT_0 & \ddots & 0 \\
\vdots & \ddots & \ddots & \ddots & \ddots &  \ddots & \ddots  & \vdots \\
0 & \hdots & 0 & T_m & \hdotsfor{3} & (-1)^mT_0 \\
\end{array}
\right|.
\]

Computing this determinant as the sum of products of one element from each row and column, we see that as a polynomial in the $\alpha_i$ and $\beta_j$, $Res(f,g)$ has degree $mn$. Terms in the sum of degree $mn$ come from, for instance, picking all of the $S_0$ and all of the $T_m$, or picking all of the $T_0$ and all of the $S_n$. No terms of greater degree can be produced.

But if $\alpha_i = \beta_j$ for any $0\leq i\leq n$ and $0\leq j\leq m$, $Res(f,g) = 0$, so $(\alpha_i-\beta_j)$ divides $Res(f,g)$. Thus $\prod_{i=1}^n\prod_{j=1}^m(\alpha_i - \beta_j)$ divides $Res(f,g)$, but both are polynomials of degree $mn$, so they differ only by a constant factor. By plugging in values for the $\alpha_i$ and $\beta_j$, it is easy to see that this constant factor is $a_n^mb_m^n$. This completes the proof.
\end{proof}

Our application of the resultant in the proof of Hensel's lemma uses the following result.

\begin{alem}[{\cite[Ch. 4 Sec. 3 Lemma]{Borevich}}]\label{lem:henselhelper}
Let $R$ be a subring of a field $K$, and let $g,h\in R[x]$ with $deg(g) = m$, $deg(h) = n$. If $\rho = Res(g,h)\neq 0$, then for all $l\in R[x]$ such that $deg(l)\leq m+n-1$, there exist $\phi,\psi\in R[x]$ with $deg(\phi)\leq n-1$, $deg(\psi)\leq m-1$ such that $g\phi + h\psi = \rho l$. 

\end{alem}
\begin{proof}
Let 
\begin{eqnarray*}
g &=& g_{m+n-1}x^{m+n-1} + \hdots + g_0,\\
h &=& h_{m+n-1}x^{m+n-1} + \hdots + h_0,\\ 
\phi &=& \phi_{m+n-1}x^{m+n-1} + \hdots + \phi_0,\\ 
\psi &=& \psi_{m+n-1}x^{m+n-1} + \hdots + \psi_0,\text{and}\\
l &=& l_{m+n-1}x^{m+n-1} + \hdots + l_0,
\end{eqnarray*}
where we set all excess coefficients to $0$. The values of the $g_j$, $h_j$, and $l_i$ are given. We must find values for the $\phi_k$ and $\psi_k$ such that for all $0\leq i\leq m+n-1$, $\sum_{j+k=i}g_j\phi_k + \sum_{j+k=i}h_j\psi_k = \rho l_i$, that is, $g\phi + h\psi = \rho l$.

This is a system of $m+n$ linear equations in $m+n$ variables, the $\phi_k$ and $\psi_k$. The corresponding matrix, $M$, is the transpose of the resultant matrix for $g$ and $h$. The determinant of this matrix is $Res(g,h) = \rho \neq 0$, so this system has a solution.

Moreover, according to the cofactor formula for the inverse,
\[
M^{-1} = \frac{1}{|M|}C^T = \frac{1}{\rho}C^T,
\]
where $C$ is the cofactor matrix of $M$. Solving $M\left(\begin{array}{c}\phi_k\\\psi_k\end{array}\right) = \rho(l_i)$ for the $\phi_k$ and $\psi_k$, we find $\left(\begin{array}{c}\phi_k\\\psi_k\end{array}\right) = M^{-1}\rho\left(\begin{array}{c}l_i\end{array}\right) = C^T\left(\begin{array}{c}l_i\end{array}\right)\in R^{m+n}$, so all the $\phi_k,\psi_k$ are elements of $R$, and thus $\phi,\psi\in R[x]$. 
\end{proof}

\bibliography{thesisbib}{}

\begin{thebibliography}{vdD04}

\bibitem[AK65]{Ax}
James Ax and Simon Kochen.
\newblock Diophantine problems over local fields. {I}.
\newblock {\em Amer. J. Math.}, 87:605--630, 1965.

\bibitem[BS66]{Borevich}
A.~I. Borevich and I.~R. Shafarevich.
\newblock {\em Number theory}.
\newblock Translated from the Russian by Newcomb Greenleaf. Pure and Applied
  Mathematics, Vol. 20. Academic Press, New York, 1966.

\bibitem[CK73]{Chang}
C.~C. Chang and H.~J. Keisler.
\newblock {\em Model theory}.
\newblock North-Holland Publishing Co., Amsterdam, 1973.
\newblock Studies in Logic and the Foundations of Mathematics, Vol. 73.

\bibitem[Gre69]{Greenberg}
Marvin~J. Greenberg.
\newblock {\em Lectures on forms in many variables}.
\newblock W. A. Benjamin, Inc., New York-Amsterdam, 1969.

\bibitem[Har77]{Hartshorne}
Robin Hartshorne.
\newblock {\em Algebraic geometry}.
\newblock Springer-Verlag, New York, 1977.
\newblock Graduate Texts in Mathematics, No. 52.

\bibitem[Jec03]{Jech}
Thomas Jech.
\newblock {\em Set theory}.
\newblock Springer Monographs in Mathematics. Springer-Verlag, Berlin, 2003.
\newblock The third millennium edition, revised and expanded.

\bibitem[Lan02]{Lang}
Serge Lang.
\newblock {\em Algebra}, volume 211 of {\em Graduate Texts in Mathematics}.
\newblock Springer-Verlag, New York, third edition, 2002.

\bibitem[Mar02]{Marker}
David Marker.
\newblock {\em Model theory}, volume 217 of {\em Graduate Texts in
  Mathematics}.
\newblock Springer-Verlag, New York, 2002.
\newblock An introduction.

\bibitem[Rib99]{Ribenboim}
Paulo Ribenboim.
\newblock {\em The theory of classical valuations}.
\newblock Springer Monographs in Mathematics. Springer-Verlag, New York, 1999.

\bibitem[vdD04]{vddries}
Lou van~den Dries.
\newblock Model theory of valued fields lecture notes.
\newblock \url{http://www.math.uiuc.edu/~vddries/valfields.dvi}, 2004.

\end{thebibliography}
\bibliographystyle{alpha}

\end{document}